\documentclass[11pt]{article}
\usepackage{amsfonts,amssymb,amsmath}
\usepackage{graphicx,graphics,psfrag}

\usepackage[english]{babel}
 \makeatletter
 \@addtoreset{equation}{section}
 \makeatother

 \newcounter{enunciato}[section]

 \newtheorem{ittheorem}{Theorem}
 \newtheorem{itlemma}{Lemma}
 \newtheorem{itproposition}{Proposition}
 \newtheorem{itdefinition}{Definition} 
 \newtheorem{itcorollary}{Corollary}

 \newenvironment{theorem}{\addtocounter{enunciato}{1}
 \begin{ittheorem}}{\end{ittheorem}}

 \newenvironment{lemma}{\addtocounter{enunciato}{1}
 \begin{itlemma}}{\end{itlemma}}

 \newenvironment{corollary}{\addtocounter{enunciato}{1}
 \begin{itcorollary}}{\end{itcorollary}}

\newcommand{\halmos}{\rule{1ex}{1.4ex}}

\newenvironment{proof}{\noindent {\em Proof.}\,}
{\hspace*{\fill}$\halmos$\medskip}

\def \ba {\begin{array}}
\def \ea {\end{array}}

\def \Z {{\mathbb Z}}
\def \R {{\mathbb R}}

\def \N {{\mathbb N}}

\def \bE {\bar E}
\def \hE {\hat E}
\def \bQ {\bar Q}
\def \hQ {\hat Q}

\def \P {{\mathbb P}}
\def \E {{\mathbb E}}

\def \cL {{\mathcal L}}
\def \cD {{\mathcal D}}
\def \cP {{\mathcal P}}

\def \cR {{\mathcal R}}

\def \cC {{\mathcal C}}

\def \tr {{\rm tr}}
\def \cN {{\mathcal N}}

\def \d {\delta}
\def \a {\alpha}

\def \g {\gamma}
\def \o {\omega}

\def \Iq {I^{\rm que}}

\def \hQ {\hat{Q}}

\def \hcR {\hat{\mathcal R}}

\def \whE {\widetilde{\hat E}}
\def \wbE {\widetilde{\bar E}}
\def \wE {\widetilde E}
\def \ho {\hat \omega}
\def \bo{\bar \omega}

\def \hm {\hat \mu}
\def \bm {\bar \mu}
\def \hM {\hat M}
\def \bb {\bar\beta}
\def \hb {\hat\beta}
\def \tM { \bar M}
\def \th {\bar h}
\def \hh {\hat h}
\def \cM {\mathcal M}

\begin{document}
\title{ Copolymer with pinning:\\
variational characterization of the phase diagram}

\author{
F.\ den Hollander\\
A.A.\ Opoku\\
\vspace{0cm}\\
Mathematical Institute, Leiden University,\\ 
P.O.\ Box 9512, 2300 RA Leiden, The Netherlands.
}

\maketitle

\begin{abstract}
This paper studies a polymer chain in the vicinity of a linear interface separating two 
immiscible solvents. The polymer consists of \emph{random monomer types}, while the 
interface carries \emph{random charges}. Both the monomer types and the charges are 
given by i.i.d.\ sequences of random variables. The configurations of the polymer 
are directed paths that can make i.i.d.\ excursions of finite length above and below 
the interface. The Hamiltonian has two parts: a monomer-solvent interaction (``copolymer'') 
and a monomer-interface interaction (``pinning''). The quenched and the annealed 
version of the model each undergo a transition from a \emph{localized phase} (where 
the polymer stays close to the interface) to a \emph{delocalized phase} (where the 
polymer wanders away from the interface). We exploit the approach developed in 
\cite{ChdHo10} and \cite{BodHoOp11} to derive \emph{variational formulas} for the 
quenched and the annealed free energy per monomer. These variational formulas are 
analyzed to obtain detailed information on the critical curves separating the two 
phases and on the typical behavior of the polymer in each of the two phases. Our 
main results settle a number of open questions.

\vskip 0.5truecm
\noindent
{\it AMS} 2000 {\it subject classifications.} 60F10, 60K37, 82B27.

\medskip\noindent
{\it Acknowledgment.} 
FdH was supported by ERC Advanced Grant VARIS 267356, AO by NWO-grant 613.000.913.

\end{abstract}


\newpage

\section{Introduction and main results}
\subsection {The model}
\label{S1.1}

{\bf 1.\ Polymer configuration.} The polymer is modeled by a directed path drawn from
the set
\begin{equation}
\label{pidef}
\Pi=\Big\{\pi=(k,\pi_{k})_{k\in\N_0}\colon\,\pi_0=0,\,\mathrm{sign}(\pi_{k-1})
+\mathrm{sign}(\pi_k) \neq 0,\,\pi_k\in\Z\,\,\forall\,k\in\N\Big\}
\end{equation}
of directed paths in $\N_0\times\Z$ that start at the origin and visit the interface 
$\N_0\times\{0\}$ when switching from the lower halfplane to the upper halfplane, and 
vice versa. Let $P^*$ be the path measure on $\Pi$ under which the excursions away from 
the interface are i.i.d., lie above or below the interface with equal probability, and 
have a length distribution $\rho$ on $\N$  with a \emph{polynomial tail}:
\begin{equation}
\label{rhocond}
\lim_{ {n\to\infty} \atop{\rho(n)>0} } 
\frac{\log\rho(n)}{\log n} = -\alpha \mbox{ for some } \alpha\in [1,\infty). 
\end{equation}
The support of $\rho$ is assumed to satisfy the following non-sparsity condition 
\begin{equation}
\lim_{m\rightarrow\infty}\frac1m \log \sum_{n>m} \rho(n)=0.
\end{equation}

Denote by $\Pi_n,P^*_n$ the restriction of $\Pi,P^*$ to $n$-step paths that end at the 
interface.

\medskip\noindent
{\bf 2.\ Disorder.} Let $\hE$ and $\bE$ be subsets of $\R$. The edges of the paths in $\Pi$ 
are labeled by an i.i.d.\ sequence of $\hE$-valued random variables $\ho=(\ho_i)_{i\in\N}$ 
with common law $\hm$, modeling the random monomer types. The sites at the interface are
labeled by an i.i.d.\ sequence of $\bE$-valued random variables $\bo=(\bo_i)_{i\in\N}$ 
with common law $\bm$, modeling the random charges. In the sequel we abbreviate $\o
=(\o_{i})_{i\in\N}$ with $\o_i=(\ho_i,\bo_i)$ and assume that $\ho$ and $\bo$ are 
independent. We further assume, without loss of generality, that both $\ho_1$ and $\bo_1$ 
have zero mean, unit variance, and satisfy
\begin{equation}
\label{mgffin}
\hM(t) = \log \int_{\hE} e^{-t\ho_1} \hm(d\ho_1) < \infty \quad \forall\,t\in\R,
\qquad 
\tM(t)= \log \int_{\bE} e^{-t\bo_1} \bm(d\bo_1) < \infty \quad \forall\,t\in\R.
\end{equation}
We write $\P$ for the law of $\o$, and $\P_{\ho}$ and $\P_{\bo}$ for the laws of $\ho$ 
and $\bo$.

\medskip\noindent
{\bf 3.\ Path measure.} Given $n\in\N$ and $\o$, the \emph{quenched copolymer with pinning} 
is the path measure given by
\begin{equation}
\label{copwad}
\widetilde P^{\hb,\hh,\bb,\th,\o}_n(\pi)
= \frac{1}{\tilde Z^{\hb,\hh,\bb,\th,\o}_n}
\,\exp\left[\widetilde H_n^{\hb,\hh,\bb,\th,\o}(\pi)\right]\,P^*_n(\pi),
\qquad \pi\in\Pi_n,
\end{equation}
where $\hb,\hh,\bb\geq0$ and $\th\in\R$ are parameters, $\widetilde Z^{\hb,\hh,\bb,\th,\o}_n$ 
is the normalizing partition sum, and
\begin{equation}
\label{copadhamil}
\widetilde H_n^{\hb,\hh,\bb,\th,\o}(\pi)
= \hb\sum_{i=1}^n (\ho_i+\hh)\,\Delta_i+\sum_{i=1}^{n} (\bb\,\bo_i-\th)\delta_i
\end{equation}
is the interaction Hamiltonian, where $\delta_i=1_{\{\pi_i=0\}} \in \{0,1\}$ and 
$\Delta_i=\mathrm{sign}(\pi_{i-1},\pi_i)\in \{-1,1\}$ (the $i$-th edge is below
or above the interface).  

\medskip\noindent
{\bf Key example:}
The choice $\hE=\bE=\{-1,1\}$ corresponds to the situation where the upper halfplane 
consists of oil, the lower halfplane consists of water, the monomer types are either 
hydrophobic ($\ho_i=1$) or hydrophilic ($\ho_i=-1$), and the charges are either positive 
($\bo_i=1$) or negative ($\bo_i=-1$); see Fig.~\ref{fig-copolex}. In \eqref{copadhamil}, 
$\hb$ and $\bb$ are the \emph{strengths} of the monomer-solvent and monomer-interface 
interactions, while $\hh$ and $\th$ are the \emph{biases} of these interactions. If $P^*$ 
is the law of the directed simple random walk on $\Z$,  then \eqref{rhocond} holds 
with $\alpha=\tfrac32$.

\begin{figure}[htbp] 
\vspace{-.50cm}
\begin{center}
\includegraphics[scale = .5]{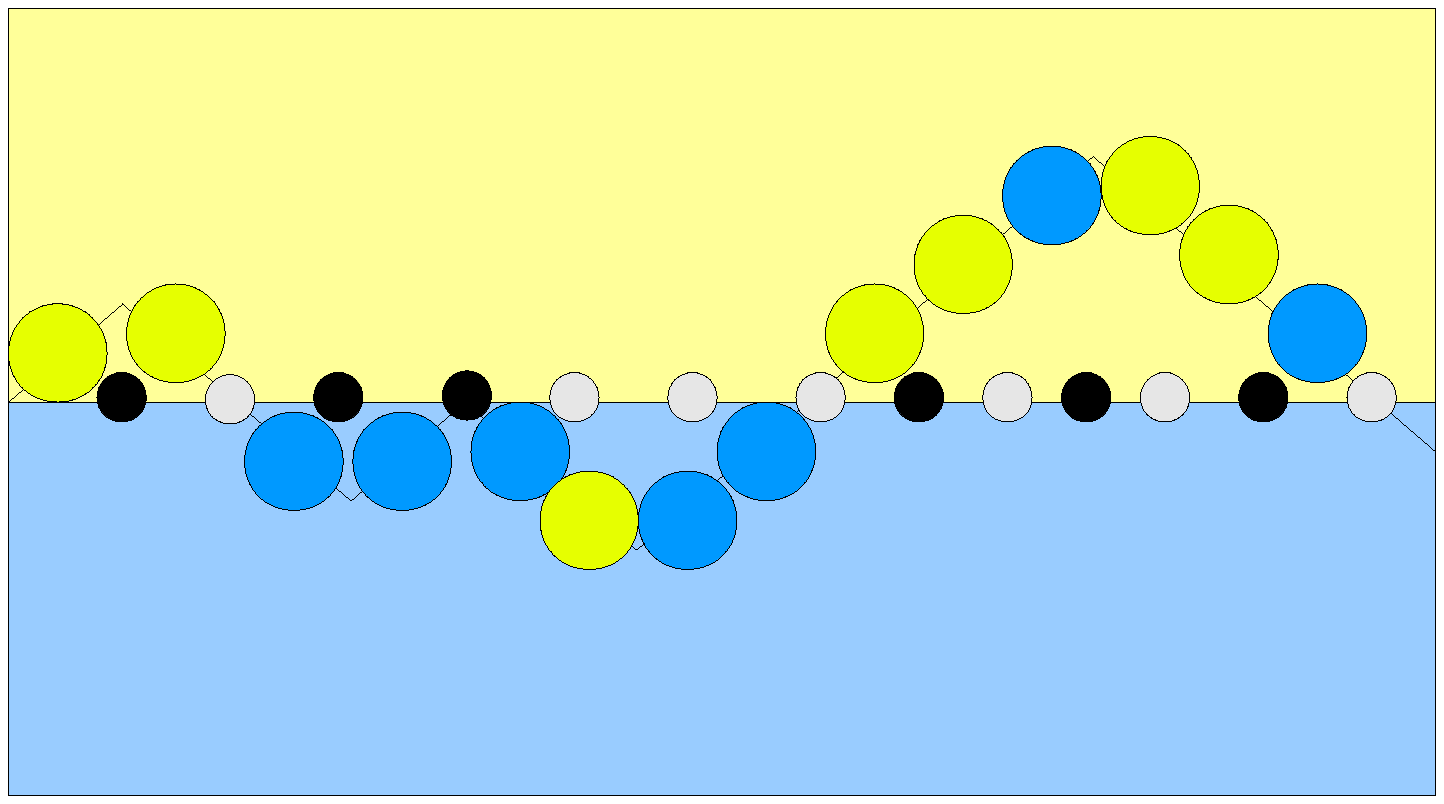}
\end{center}
\vspace{0cm}
\vspace{-10.5cm}
\caption{\small A directed polymer near a linear interface, separating oil in the upper
halfplane and water in the lower halfplane. Hydrophobic monomers in the polymer are 
light shaded, hydrophilic monomers are dark shaded. Positive charges at the interface 
are light shaded, negative charges are dark shaded.} 
\label{fig-copolex}
\vspace{-.1cm}
\end{figure}
 
In the literature, the model without the monomer-interface interaction ($\bb=\bar h=0$) is 
called the \emph{copolymer model}, while the model without the monomer-solvent interaction 
($\hh=\hb=0$) is called the \emph{pinning model} (see Giacomin~\cite{Gi07} and den 
Hollander~\cite{dHo09} for an overview). The model with both interactions is referred to
as the \emph{copolymer with pinning model}. In the sequel, if $k$ is a quantity associated 
with the combined model, then $\hat k$ and $\bar k$ denote the analogous quantities in
the copolymer model, respectively, the pinning model.

 
\subsection{Quenched excess free energy and critical curve}
\label{S1.2}
 
The \emph{quenched free energy} per monomer
\begin{equation}
\label{copadfe}
f^{\rm que}(\hb,\hh,\bb,\th) = \lim_{n\to\infty} \frac1n\,\log\widetilde
Z^{\hb,\hh,\bb,\th,\o}_n
\end{equation}
exists $\o$-a.s.\ and in $\P$-mean (see e.g.\ Giacomin~\cite{Gi04}). By restricting 
the partition sum $\widetilde Z_n^{\hb,\hh,\bb,\th,\o}$ to paths that stay above 
the interface up to time $n$, we obtain, using the law of large numbers for $\ho$, that
$f^{\rm que}(\hb,\hh,\bb,\th)\geq \hb\,\hh$. The {\em quenched excess free energy} per 
monomer
\begin{equation}
\label{copadefe}
g^{\rm que}(\hb,\hh,\bb,\th) = f^{\rm que}(\hb,\hh,\bb,\th)-\hb\,\hh
\end{equation}
corresponds to the Hamiltonian
\begin{equation}
\label{copadmodhamil}
H_n^{\hb,\hh,\bb,\th,\o}(\pi)=\hb\sum_{i=1}^n (\ho_i+\hh)\,\left[\Delta_i-1\right]
+\sum_{i=1}^n (\bb\,\bo_i-\th)\delta_i
\end{equation}
and has two phases
\begin{equation}
\label{copadquephases}
\begin{aligned}
\cL^{\rm que} &= \left\{(\hb,\hh,\bb,\th)\in[0,\infty)^3\times\R\colon\, 
g^{\rm que}(\hb,\hh,\bb,\th)>0\right\},\\
\cD^{\rm que} &= \left\{(\hb,\hh,\bb,\th)\in[0,\infty)^3\times\R\colon\,
g^{\rm que}(\hb,\hh,\bb,\th)=0\right\},
\end{aligned}
\end{equation}
called the {\em quenched localized phase} (where the strategy of staying close to the 
interface is optimal) and the {\em quenched delocalized phase} (where the strategy of 
wandering away from the interface is optimal). The map $\hh\mapsto g^{\rm que}(\hb,\hh,
\bb,\th)$ is non-increasing and convex for every $\hb,\bb\geq0$ and $\th\in\R$. Hence, 
$\cL^{\rm que}$ and $\cD^{\rm que}$ are separated by a single curve 
\begin{equation}
\label{copadcc}
h_c^{\rm que}(\hb,\bb,\th) 
= \inf\left\{\hh\geq 0\colon\,g^{\rm que}(\hb,\hh,\bb,\th)=0\right\},
\end{equation}
called the {\em quenched critical curve}. 

In the sequel we write $\hat g^{\rm que}(\hb,\hh)$, $\hat h_c^{\rm que}(\hb)$, $\hat
\cL^{\rm que}$, $\hat\cD^{\rm que}$ for the analogous quantities in the copolymer model 
($\bb=\th=0$), and $\bar g^{\rm que}(\bb,\th)$, $\bar h_c^{\rm que}(\bb)$, $\bar\cL^{\rm que}$, 
$\bar\cD^{\rm que}$ for the analogous quantities in the pinning model ($\hb= \hh=0$).

 
\subsection{Annealed excess free energy and critical curve}
\label{S1.3}
 
The {\em annealed excess free energy} per monomer is given by
\begin{equation}
\label{copadaefe}
g^{\rm ann}(\hb,\hh,\bb,\th) = \lim_{n\to\infty} \frac1n \log Z_n^{\hb,\hh,\bb,\th}
= \lim_{n\to\infty} \frac1n \log \E\left(Z_n^{\hb, \hh,\bb,\th,\o}\right),
\end{equation}
where $\E$ is the expectation w.r.t.\ the joint disorder distribution $\P$. This also 
has two phases,
\begin{equation}
\label{copadannphases}
\begin{aligned} 
\cL^{\rm ann} &= \left\{(\hb,\hh,\bb,\th)\in [0,\infty)^3\times\R\colon\, 
g^{\rm ann}(\hb,\hh,\bb,\th)>0\right\},\\
\cD^{\rm ann} &= \left\{(\hb,\hh,\bb,\th)\in [0,\infty)^3\times\R\colon\, 
g^{\rm ann}(\hb,\hh,\bb,\th)=0\right\},
\end{aligned}
\end{equation}
called the {\em annealed localized phase} and the {\em annealed delocalized phase}, 
respectively. The two phases are separated by the {\em annealed critical curve}
\begin{equation}
\label{copadacc}
h_c^{\rm ann}(\hb,\bb,\th) = \inf\left\{\hh\geq 0\colon\, 
g^{\rm ann}(\hb,\hh,\bb,\th)=0\right\}.
\end{equation}
Let $\cN(g) = \sum_{n\in\N} e^{-ng}\,\rho(n),\, g\in\R$. We will show in Section~\ref{S3.2} 
that
\begin{equation}
\label{ganncomb}
\begin{aligned}
&g^{\rm ann}(\hb,\hh,\bb,\th) \mbox{ is the unique $g$-value at which}\\
&\log \left[\tfrac12 \cN(g)+\tfrac12 \cN\big(g-[\hat M(2\hb)-2\hb \hh]\big)\right] 
+\bar M(-\bb)-\th \mbox{ changes sign}.
\end{aligned}
\end{equation}

\noindent
{\bf Remark:}  The annealed model is exactly solvable. In fact, sharp asymptotics estimates on the 
annealed partition function that go beyond the free energy can  be derived by using  the techniques 
 in  Giacomin~\cite{Gi07}, Section 2.2.  We will derive variational formulas for  the annealed and 
 the quenched free energies. The  annealed variational problem turns out to be easy, but we need 
 it as an object of comparison in our study of the quenched variational formula.

It follows from \eqref{ganncomb} that for the copolymer model ($\bb=\th=0$) 
\begin{equation}
\label{copannfeg}
\begin{aligned}
\hat{g}^{\rm\, ann}(\hb,\hh) &= 0 \vee [\hat M(2\hb)-2\hb\hh],\\
\hat h_c^{\rm ann}(\hb) &= (2\hb)^{-1}\hM(2\hb),
\end{aligned}
\end{equation}
and for the pinning model ($\hb=\hh=0$)
\begin{equation}
\label{pinannfeg}
\begin{aligned}
&\bar g^{\rm ann}(\bb,\th) \mbox{ is the unique $g$-value for which } 
\cN(g) = e^{-(0\vee[\bar M(-\bb)-\th])},\\
&\bar h_c^{\rm ann}(\bb) = \bar M(-\bb).
\end{aligned}
\end{equation}
For more details on these special cases, see Giacomin~\cite{Gi07} and den 
Hollander~\cite{dHo09}, and references therein.  
 

\subsection{Main results}
\label{S1.4}

Our variational characterization of the excess free energies and the critical curves
is contained in the following theorem. For technical reasons, in the sequel we 
\emph{exclude} the case $\hb>0$, $\hh=0$ for the quenched version.

\begin{theorem}
\label{freeenegvar} 
Assume  {\rm (\ref{rhocond})} and {\rm (\ref{mgffin})}.\\ 
(i) For every $\alpha \geq 1$ and $\hb,\hh,\bb\geq 0$, there are lower semi-continuous, 
convex and non-increasing functions
\begin{equation}
\label{Sannounce}
\begin{aligned}
&g\mapsto S^\mathrm{que}(\hb,\hh,\bb;g),\\
&g\mapsto S^\mathrm{ann}(\hb,\hh,\bb;g),
\end{aligned}
\end{equation}
given by explicit variational formulas such that, for every $\th\in\R$,
\begin{equation}
\label{gquegannvarfor}
\begin{aligned}
g^\mathrm{que}(\hb,\hh,\bb,\th) &= \inf\{g\in\R\colon\,S^\mathrm{que}(\hb,\hh,\bb;g)-\th<0\},\\
g^\mathrm{ann}(\hb,\hh,\bb,\th) &= \inf\{g\in\R\colon\,S^\mathrm{ann}(\hb,\hh,\bb;g)-\th<0\}.
\end{aligned}	
\end{equation} 
(ii) For every $\alpha \geq 1$, $\hb>0$, $\bb\geq 0$ and $\th\in\R$,
\begin{equation}
\label{ccquccannvarfor}
\begin{aligned}
h_c^\mathrm{que}(\hb,\bb,\th) 
&= \inf\left\{\hh>0\colon\,S^\mathrm{que}(\hb,\hh,\bb;0)-\th\leq 0\right\},\\
h_c^\mathrm{ann}(\hb,\bb,\th) 
&= \inf\left\{\hh\geq 0\colon\,S^\mathrm{ann}(\hb,\hh,\bb;0)-\th\leq0\right\}.
\end{aligned}	  
\end{equation}
\end{theorem}

\noindent
The variational formulas for $S^\mathrm{que}(\hb,\hh,\bb;g)$ and $S^\mathrm{ann}(\hb,\hh,
\bb;g)$ are given in Theorems~\ref{varfloc}--\ref{varflocann} in Section~\ref{S3}. 
Figs.~\ref{fig-varfe}--\ref{fig-copadvarhs} in Sections~\ref{S3} and \ref{Sec5} show 
how these functions depend on $\hb$, $\hh$, $\bb$ and $g$, which is crucial for our analysis.

Next, we state seven corollaries that are consequences of the variational formulas. 
The content of these corollaries will be discussed in Section~\ref{S1.6}. The 
first corollary looks at the excess free energies. Put
\begin{equation}
\label{lann12}
\begin{split}
\bar h_\ast(\hb,\hh,\bb)&=\bar M(-\bb)+\log\left(\tfrac12\left[1+\cN\big(|\hM(2\hb)-2\hb\hh|\big)
\right]\right),\cr
\cL^{\rm ann}_1
&=\left\{(\hb,\hh,\bb,\bar h)\in[0,\infty)^3\times\R\colon\,(\hb,\hh)\in\hat 
\cL^{\rm ann}\right\},\cr
\cL^{\rm ann}_2
&=\left\{(\hb,\hh,\bb,\bar h)\in[0,\infty)^3\times\R\colon\,(\hb,\hh)\in\hat 
\cD^{\rm ann},\bar h<\bar h_\ast(\hb,\hh,\bb)\right\}.
\end{split}
\end{equation}

\begin{corollary}
\label{freeeneggap} 
(i) For every $\alpha\geq 1$, $\hb>0$ and $\bb\geq0$, $g^\mathrm{que}(\hb,\hh,\bb,\th)$ 
and $g^\mathrm{ann}(\hb,\hh,\bb,\th)$ are the unique $g$-values that solve the equations
\begin{equation}
\begin{split}
S^{\rm que}(\hb,\hh,\bb;g)
&= \bar h, \quad \text{if }  \hh>0 , \,\bar h\leq 
S^{\rm que}(\hb,\hh,\bb;0), 
\cr
S^{\rm ann}(\hb,\hh,\bb;g)
&=\bar h, \quad \text{if } \hh\geq 0,\, 
\bar h\leq \bar h_{\ast}(\hb,\hh,\bb).
\end{split}
\end{equation}
(ii) The annealed localized phase $\cL^{\rm ann}$ admits the decomposition $\cL^{\rm ann}
=\cL^{\rm ann}_1\cup \cL^{\rm ann}_2$.\\
\noindent
(iii) On $\cL^{\rm ann}$,
\begin{equation}
g^\mathrm{que}(\hb,\hh,\bb,\th) < g^\mathrm{ann}(\hb,\hh,\bb,\th),
\end{equation}
with the possible exception of the case where $m_\rho=\sum_{n\in\N} n\rho(n)=\infty$ and 
$\th=\bar h_\ast(\hb,\hh,\bb)$.\\
(iv) For every $\alpha\geq 1$ and $\hb,\hh,\bb\geq0$, 
\begin{equation}
g^\mathrm{ann}(\hb,\hh,\bb,\th)
\left\{\begin{array}{ll}
=\hat g^{\rm ann}(\hb,\hh), &\mbox{if } \th\geq \bar h_\ast(\hb,\hh,\bb),\\
>\hat g^{\rm ann}(\hb,\hh), &\mbox{otherwise.}
\end{array}
\right.
\end{equation}
\end{corollary} 

The next four corollaries look at the critical curves.

\begin{corollary}
\label{hquehannvarfor} 
For every $\alpha\geq 1$, $\hb>0$ and $\bb\geq0~$, the maps 
\begin{equation}
\label{Sannounce*}
\begin{aligned}
&\hh\mapsto S^\mathrm{que}(\hb,\hh,\bb;0),\\
&\hh \mapsto S^\mathrm{ann}(\hb,\hh,\bb;0),
\end{aligned}
\end{equation}
are convex and non-increasing on $(0,\infty)$. Both critical curves are continuous and 
non-increasing in $\th$. Moreover (see Figs.~{\rm \ref{fig-copadvarhc}--\ref{fig-copadvarhannc}}), 
\begin{equation}
\label{ccqueflatp}
h_c^\mathrm{que}(\hb,\bb,\th)
=\left\{\begin{array}{ll}
\infty, &\mbox{if } \th\leq\th_c^{\rm que}(\bb)-\log 2,\\
\hat h_c^{\rm ann}(\hb/\a), &\mbox{if } \th> s^{\ast}(\hb,\bb,\alpha),\\
h_\ast^{\rm que}(\hb,\bb,\th), &\mbox{otherwise},
\end{array}
\right.
\end{equation}
and 
\begin{equation}
\label{ccannflatp}
h_c^\mathrm{ann}(\hb,\bb,\th)
=\left\{\begin{array}{ll}
\infty, &\mbox{if } \th\leq\th_c^{\rm ann}(\bb)-\log 2,\\
\hat h_c^{\rm ann}(\hb), &\mbox{if } \th> \th_c^{\rm ann}(\bb),\\
h_\ast^{\rm ann}(\hb,\bb,\th), &\mbox{otherwise}.
\end{array}
\right.
\end{equation}
where 
\begin{equation}
\label{sast}
s^{\ast}(\hb,\bb,\alpha) = S^\mathrm{que}\left(\hb,\hat h_c^{\rm ann}
(\hb/\a),\bb;0\right) \in \left([\th_c^{\rm que}(\bb)-\log 2]\vee0,\infty\right],
\end{equation}
and  $h_\ast^{\rm que}(\hb,\bb,\th)$ and $h_\ast^{\rm ann}
(\hb,\bb,\th)$ are the unique $\hh$-values that solve the equations 
\begin{equation}
\label{qccth}
\begin{split}
S^\mathrm{que}(\hb,\hh,\bb;0) &= \th,\cr
S^\mathrm{ann}(\hb,\hh,\bb;0) &= \th.
\end{split}
\end{equation}
In particular, both $h_\ast^{\rm que}(\hb,\bb,\th)$ and $h_\ast^{\rm ann}
(\hb,\bb,\th)$ are convex and strictly decreasing functions of $\th$.
\end{corollary}

\begin{figure}[htbp]
\vspace{.8cm}
\begin{minipage}[hbt]{7.6cm}
\centering
\setlength{\unitlength}{0.50cm}
\begin{picture}(12,6)(0,0)
\put(0,0){\line(12,0){12}}
\put(4,0){\line(0,6){6}}
{\thicklines
\qbezier(2,5)(5,1.6)(9,1)
\qbezier(9,1)(11,1)(12,1)
\qbezier(2,5)(1,5)(0,5)
}
\qbezier[30](2,5)(2,3)(2,0)
\qbezier[30](2,1)(5,1)(9,1)
\qbezier[10](8.9,1)(8.9,.5)(8.9,0)
\put(1,5.2){$\infty$}
\put(7.6,-1){$s^{\ast}(\hb,\bb,\alpha)$}
\put(2,6.3){$h^\mathrm{que}_c(\hb,\bb,\th)$}
\put(12.2,-.3){$\th$}
\put(-1,.9){$\hat h^{\rm ann}_c(\frac{\hb}{\a})$}
\put(-.5,-1){$\th_c^{\rm que}(\bb)-\log 2$} 
\end{picture}
\vspace{.6cm}
\begin{center}
(a)
\end{center}
\end{minipage}
\begin{minipage}[hbt]{8.6cm}
\centering
\setlength{\unitlength}{0.50cm}
\begin{picture}(12,6)(0,0)
\put(0,0){\line(12,0){12}}
\put(4,0){\line(0,6){6}}
{\thicklines
\qbezier(2,5)(5,1.3)(12,1.2)
\qbezier(2,5)(1,5)(0,5)
}
\qbezier[30](2,5)(2,3)(2,0)
\qbezier[30](2,1)(5,1)(12,1)
\put(1,5.2){$\infty$}
\put(2,6.3){$h^\mathrm{que}_c(\hb,\bb,\th)$}
\put(12.2,-.3){$\th$}
\put(-1,.9){$\hat h^{\rm ann}_c(\frac{\hb}{\a})$}
\put(-.5,-1){$\th_c^{\rm que}(\bb)-\log 2$} 
\end{picture}
\vspace{.6cm}
\begin{center}
(b)
\end{center}
\end{minipage}
\vspace{-0cm}
\caption{\small Qualitative picture of the map $\th \mapsto h^\mathrm{que}_c(\hb,\bb,\th)$ 
for $\hb>0$ and $\bb\geq0$ when: (a) $s^{\ast}(\hb,\bb,\alpha)<\infty$; (b)  
$s^{\ast}(\hb,\bb,\alpha)=\infty$.} 
\label{fig-copadvarhc}
\end{figure}

\begin{figure}[htbp]
\vspace{1.5cm}
\begin{center}
\setlength{\unitlength}{0.50cm}
\begin{picture}(22,9)(0,0)
\put(-1,0){\line(22,0){22}}
\put(8,0){\line(0,9){9}}
{\thicklines
\qbezier(3,8)(10,2.6)(17,2)
\qbezier(17,2)(20,2)(21,2)
\qbezier(3,8)(1,8)(-1,8)
}
\qbezier[40](3,8)(3,3)(3,0)
\qbezier[20](16.8,0)(16.8,1)(16.8,2)
\qbezier[60](3,2)(12,2)(17,2)
\put(1.3,8.2){$\infty$}
\put(15.9,-1){$\th^{\rm ann}_c(\bb)$}
\put(5.6,9.5){$h^\mathrm{ann}_c(\hb,\bb,\th)$}
\put(0.2,1.9){$\hat h^\mathrm{ann}_c(\hb)$}
\put(21.5,-.3){$\th$}
\put(.81,-1){$\th^{\rm ann}_c(\bb)-\log 2$} 
\end{picture}
\end{center}
\vspace{0.2cm}
\caption{\small Qualitative picture of the map $\th \mapsto h^\mathrm{ann}_c
(\hb,\bb,\th)$ for $\hb,\bb\geq0$.} 
\label{fig-copadvarhannc}
\end{figure}

\begin{corollary}
\label{hcubstrict} 
For every $\alpha>1$, $\hb>0$ and $\bb\geq0$,  
\begin{equation}
h^\mathrm{que}_c(\hb,\bb,\th) 
\left\{\begin{array}{ll} 
< h^\mathrm{ann}_c(\hb,\bb,\th)\leq\infty, &\mbox{if } \th>\bar h^{\rm que}_c(\bb)-\log 2,\\
= h^\mathrm{ann}_c(\hb,\bb,\th)=\infty, &\mbox{otherwise.}
\end{array}
\right.
\end{equation}
\end{corollary}

\begin{corollary}
\label{hclbstrict}
For every $\alpha>1$, $\hb>0$ and $\bb\geq0$, 
\begin{equation}
h^\mathrm{que}_c(\hb,\bb,\th)
\left\{\begin{array}{ll}
> \hat h^\mathrm{ann}_c(\hb/\alpha), &\mbox{ if } \th< s^{\ast}(\hb,\bb,\alpha),\\
= \hat h^\mathrm{ann}_c(\hb/\alpha), &\mbox{ otherwise.}
\end{array}
\right.
\end{equation} 
\end{corollary}

\begin{corollary}
\label{critqueann}
(i) For every $\alpha\geq 1$ and $\hb,\bb\geq 0$,
\begin{equation}
\begin{aligned}
\inf\left\{\th\in\R\colon\,g^\mathrm{ann}(\hb,\hat h_c^{\rm ann}(\hb),\bb,\th)=0 \right\}
&= \bar h^{\rm ann}_c(\bb),\\
\inf\left\{\hat h\geq 0\colon\,g^\mathrm{ann}(\hb,\hat h,\bb,\th_c^{\rm ann}(\bb))=0 \right\}
&= \hat h^{\rm ann}_c(\hb).
\end{aligned}
\end{equation}
(ii) For every $\alpha\geq 1$, $\hb>0$ and $\bb=0$, 
\begin{equation}
\inf\left\{\th\in\R\colon g^\mathrm{que}(\hb,\hat h_c^{\rm que}(\hb),\bb,\th)=0 \right\}
=0.
\end{equation}
\end{corollary}

The last two corollaries concern the typical path behavior. Let $P^{\hb,\hh,\bb,\th,\o}_n$ 
denote the path measure associated with the Hamiltonian $H_n^{\hb,\hh,\bb,\th,\o}$ defined 
in \eqref{copadmodhamil}. Write $\cM_n = \cM_n(\pi) = |\{1\leq i\leq n\colon\,\pi_i=0\}|$ 
to denote the number of times the polymer returns to the interface up to time $n$. Define 
\begin{equation}
\label{dquedecomp}
\begin{split}
\cD^{\rm que}_1 = \left\{(\hb,\hh,\bb,\th)\in\cD^{\rm que}\colon\, 
\th\leq s^\ast(\hb,\bb,\alpha)\right\}.
\end{split}
\end{equation}

\begin{corollary}
\label{delocpathprop}
For every $(\hb,\hh,\bb,\th)\in \mathrm{int}(\cD^\mathrm{que}_1) \cup (\cD^{\rm que}
\setminus\cD^{\rm que}_1)$ and $c>\alpha/[-(S^\mathrm{que}(\hb,\hh,\bb;0)-\th)] \in 
(0,\infty)$,
\begin{equation}
\label{ubdint}
\lim_{n\to\infty} P^{\hb,\hh,\bb,\th,\o}_n\left(\cM_n \geq c\log n\right) 
= 0 \qquad \o-a.s.
\end{equation}
\end{corollary}

\begin{corollary}
\label{locpathprop}
For every $(\hb,\hh,\bb,\th) \in \cL^\mathrm{que}$, 
\begin{equation}
\label{LLNpath}
\lim_{n\to\infty} 
P^{\hb,\hh,\bb,\th,\o}_n\left(|\tfrac1n\cM_n-C| \leq \varepsilon \right) = 1
\qquad \o-a.s.\quad \forall\,\varepsilon>0,
\end{equation}
where 
\begin{equation}
\label{densid}
-\frac{1}{C} = \frac{\partial}{\partial g}\,
S^\mathrm{que}\big(\hb,\hh,\bb;g^\mathrm{que}(\hb,\hh,\bb,\th)\big)
\in (-\infty,0),
\end{equation}
provided this derivative exists. (By convexity, at least the left-derivative and the 
right-derivative exist.)
\end{corollary}


\subsection{Discussion}
\label{S1.6}

{\bf 1.}
The copolymer and pinning versions of Theorem \ref{freeenegvar} are obtained by putting 
$\bb=\th=0$ and $\hb=\hh=0$, respectively. The copolymer version of Theorem \ref{freeenegvar} 
was proved in Bolthausen, den Hollander and Opoku \cite{BodHoOp11}.

\medskip\noindent 
{\bf 2.} 
Corollary~\ref{freeeneggap}(i) identifies the range of parameters for which the free energies
given by \eqref{gquegannvarfor} are the $g$-values where the variational formulas equal $\th$.   
Corollary~\ref{freeeneggap}(ii) shows that the annealed combined model is localized when 
the annealed copolymer model is localized. On the other hand, if the annealed copolymer 
model is delocalized, then a sufficiently attractive pinning interaction is needed for 
the annealed combined model to become localized, namely, $\th<\th_\ast(\hb,\hh,\bb)$. It 
is an open problem to identify a similar threshold for the quenched combined model.

\medskip\noindent
{\bf 3.}
In Bolthausen, den Hollander and Opoku~\cite{BodHoOp11} it was shown with the help of 
the variational approach that for the copolymer model there is a gap between the quenched 
and the annealed excess free energy in the localized phase of the annealed copolymer model. 
It was argued that this gap can also be deduced with the help of a result in Giacomin 
and Toninelli~\cite{GiTo06a,GiTo06b}, namely, the fact that the map $\hh\mapsto\hat 
g^\mathrm{que}(\hb,\hh)$ drops below a quadratic as $\hh \uparrow \hat h_c^\mathrm{que}(\hb)$ 
(i.e., the phase transition is ``at least of second order''). Indeed, $g^\mathrm{que}
\leq g^\mathrm{ann}$, $\hh\mapsto \hat g^\mathrm{que}(\hb,\hh)$ is convex and strictly decreasing 
on $(0,\hat h_c^\mathrm{que}(\hb)]$, and $\hh\mapsto\hat g^\mathrm{ann}(\hb,\hh)$ 
is linear and strictly decreasing on $(0,\hat h_c^\mathrm{ann}(\hb)]$. The quadratic bound
implies that the gap is present for $\hh$ slightly below $\hat h_c^\mathrm{ann}(\hb)$, and 
therefore it must be present for all $\hh$ below $\hat h_c^\mathrm{ann}(\hb)$. Now, the 
same arguments as in \cite{GiTo06a,GiTo06b} show that also $\hh\mapsto  g^\mathrm{que}
(\hb,\hh,\bb,\th)$ drops below a quadratic as $\hh \uparrow h_c^\mathrm{que}(\hb,\bb,\th)$. However, $\hh\mapsto g^\mathrm{ann}(\hb,\hh,\bb,\th)$ is {\em not} linear on 
$(0,h_c^\mathrm{ann}(\hb,\bb,\th)]$ (see \eqref{ganncomb}), and so there is no similar 
proof of Corollary~\ref{freeeneggap}(iii). Our proof underscores the \emph{robustness} 
of the variational approach. We expect the gap to be present also when $m_\rho=\infty$ 
and $\th=\th_*(\hb,\hh,\bb)$, but this remains open.

\medskip\noindent
{\bf 4.} Corollary~\ref{freeeneggap}(iv) gives a natural interpretation for $\bar h_\ast
(\hb,\hh,\bb)$, namely, this is the critical value below which the pinning interaction 
has an effect in the annealed model and above which it has not.

\medskip\noindent
{\bf 5.}
The precise shape of the quenched critical curve for the combined model was not well 
understood (see e.g.\ Giacomin~\cite{Gi07}, Section 6.3.2, and Caravenna, Giacomin and 
Toninelli~\cite{CaGiTo10}, last paragraph of Section 1.5). In particular, in \cite{Gi07}
two possible shapes were suggested for $\bb=0$, as shown in Fig.~\ref{fig-giaco}. 
Corollary~\ref{hquehannvarfor} rules out line 2, while it proves line 1 in the following 
sense: (1) this line holds for all $\bb\geq0$; (2) for $\th<\th^{\rm que}_c(\bb)-\log 2$, 
the combined model is \emph{fully localized}; (3) conditionally on $s^\ast(\hb,\bb,
\alpha)<\infty$, for $\th \geq s^\ast(\hb,\bb,\alpha)$ the quenched critical curve concides 
with $\hat h_c^{\rm ann}(\hb/\alpha)$ (see Fig.~\ref{fig-copadvarhc}). In the literature 
$\hh\mapsto \hat h_c^{\rm ann}(\hb/\alpha)$ is called the \emph{Monthus-line}. Thus, when we sit at 
the far ends of the $\th$-axis, the critical behavior of the quenched combined model is
determined either by the copolymer interaction (on the far right) or by the pinning 
interaction (on the far left). Only in between is there a non-trivial competition between 
the two interactions.   

\begin{figure}[htbp]
\vspace{.5cm}
\begin{center}
\setlength{\unitlength}{.68cm}
\begin{picture}(14,7)(0,0)
\put(0,0){\line(14,0){14}}
\put(7.2,0){\line(0,7){7}}
{\thicklines
\qbezier(-.61,6.8)(1.8,2.6)(14,1.7)
\qbezier(-.61,5.2)(3,2.9)(5,2.6)
\qbezier(5,2.6)(12,2.6)(14,2.6)
}
\qbezier[25](7.2,1.5)(12,1.5)(14,1.5)
\put(6,7.5){$h^\mathrm{que}_c(\hb,0,\th)$}
\put(7.3,2.8){$\hat h^\mathrm{que}_c(\hb)$}
\put(4.5,1.3){$\hat h^\mathrm{ann}_c(\hb/\alpha)$}
\put(14.5,-.3){$\th$}
\put(-1,6.6){$1$}
\put(-1,5){$2$}
\end{picture}
\end{center}
\vspace{-0.4cm}
\caption{\small Possible qualitative pictures of the map $\th \mapsto h^\mathrm{que}_c
(\hb,0,\th)$ for $\hb>0$.} 
\label{fig-giaco}
\end{figure}

\medskip\noindent
{\bf 6.}
The threshold values $\th=\th^{\rm que}(\bb)-\log 2$ and $\th=\th^{\rm ann}(\bb)-\log 2$ 
(see Figs.~\ref{fig-copadvarhc}--\ref{fig-copadvarhannc}) are the critical points for 
the quenched and the annealed pinning model when the polymer is allowed to stay in the 
upper halfplane only. In the literature this restricted pinning model is called the 
\emph{wetting model} (see Giacomin~\cite{Gi07}, den Hollander \cite{dHo09}). 
These values of $\th$ are the transition points at which the quenched and the annealed 
critical curves of the combined model change from being finite to being infinite. Thus, 
we recover the critical curves for the wetting model from those of the combined model 
by putting $\hh=\infty$.

\medskip\noindent
{\bf 7.}
It is known from the literature that the pinning model undergoes a transition between
\emph{disorder relevance} and \emph{disorder irrelevance}. In the former regime, there 
is a gap between the quenched and the annealed critical curve, in the latter there is 
not. (For  some authors disorder relevance also incorperates a difference in behaviour 
of the annealed and the quenched free energies as the critical curves are approached.) 
The transition depends on $\alpha$, $\bb$ and $\bm$ (the pinning disorder law). 
In particular, if $\alpha>
\tfrac32$, then the disorder is relevant for all $\bb>0$, while if $\alpha\in (1,\tfrac32)$, 
then there is a critical threshold $\bb_c\in (0,\infty]$ such that the disorder is 
irrelevant for $\bb\leq \bb_c$ and relevant for $\bb>\bb_c$.  The transition is absent 
in the copolymer model (at least when the copolymer disorder law $\hat\mu$ has finite exponential 
moments): the disorder is relevant for all $\alpha>1$. However, 
Corollary~\ref{hcubstrict} shows that in the combined model the transition occurs for 
all $\alpha>1$, $\hb>0$ and $\bb\geq 0$. Indeed, the disorder is relevant for $\th>
\th^{\rm que}(\bb)-\log 2$ and is irrelevant for $\th\leq \th^{\rm que}(\bb)-\log 2$.  

\medskip\noindent
{\bf 8.}
The quenched critical curve is bounded from below by the Monthus-line (as the critical 
curve moves closer to the Monthus-line, the copolymer interaction more and more dominates 
the pinning interaction). Corollary~\ref{hclbstrict} and Fig.~\ref{fig-copadvarhc} show
that the critical curve stays above the Monthus-line as long as $\th<s^\ast(\hb,\bb,
\alpha)$. If $s^\ast(\hb,\bb,\alpha)=\infty$, then the quenched critical curve is 
everywhere above the Monthus-line (see Fig.~\ref{fig-copadvarhc}(b)). A \emph{sufficient} 
condition for $s^\ast(\hb,\bb,\alpha)<\infty$ is 
\begin{equation}
\label{neccond}
\sum_{n\in\N}\rho(n)^{\frac1\alpha}<\infty.
\end{equation}
We do not know whether $s^\ast(\hb,\bb,\alpha)<\infty$ always. For $\bb=0$, 
Toninelli~\cite{To07} proved that, under condition \eqref{neccond}, the quenched 
critical curve coincides with the Monthus-line for $\th$ large enough. 
 
\medskip\noindent
{\bf 9.} 
  As an anonymous referee pointed out, line 2 of Fig. \ref{fig-giaco}  can be  disproved 
  by combining   results for the copolymer model proved in Bolthausen, den Hollander and  Opoku
  \cite{BodHoOp11} and Giacomin \cite{Gi07}, Section 6.3.2, with  a fractional moment estimate.
   Let us present the argument  for the case $\bar\beta=0$. The key is the observation that 
  \begin{enumerate}
  \item[(1)] $h_c^{\rm que}(\hat \beta,0,0)=\hat h_c^{\rm que}(\hat\beta)>\hat h_c^{\rm ann}
  (\hat\beta/\alpha)$ (proved in \cite{BodHoOp11}). 
  \item[(2)] $\lim_{\bar h\rightarrow\infty} h_c^{\rm que}(\hat \beta,0,\bar h)
  =\hat h_c^{\rm ann}(\hat\beta/\alpha).$
  \end{enumerate}
  The proof for (2) goes as follows:
Note from Giacomin \cite{Gi07}, Section 6.3.2, that $h_c^{\rm que}(\hat \beta,0,\bar h)
  \geq\hat h_c^{\rm ann}(\hat\beta/\alpha)$. The reverse of this inequality as $\bar h\rightarrow\infty$ 
  follows from the  fractional moment estimate
 \begin{equation}\label{simproof}
 \begin{aligned}
 \E\left[\left(Z_n^{\hb,\hh,0,\bar h,\o}\right)^\gamma\right]
& \leq \sum_{N=1}^n\sum_{0=k_0<k_1<\ldots<k_N=n}\prod_{i=1}^N\rho(k_i-k_{i-1})^\gamma e^{-\bar h\gamma} 
 2^{1-\gamma},
 \end{aligned}
 \end{equation}
 valid for $\hh=\hh^{\rm ann}_c(\gamma\hb)$ and $\gamma\in(0,1)$, for the combined 
 partition sum where the path starts and ends at the interface. 
 Indeed, for any $\gamma>\frac1\alpha$, if $\bar h>0$ is large enough so that 
 $\sum_{n\in\N}\rho(n)^\gamma e^{-h\gamma} 
 2^{1-\gamma}\leq 1$, then  the right-hand side is  the partition function for a homogeneous 
 pinning model with a defective excursion length distribution and therefore has zero free 
 energy (see e.g. \cite{Gi07}, Section 2.2). Hence $h_c^{\rm que}(\hat \beta,0,\bar h)
  \leq\hat h_c^{\rm ann}(\gamma\hat\beta)$. Let $\bar h\rightarrow \infty$ followed by 
  $\gamma\downarrow\frac1\alpha$ to get (2). But (1) and (2) together with convexity of 
  $\bar h\mapsto h_c^{\rm que}(\hat \beta,0,\bar h)$ disprove line 2 of  Fig. \ref{fig-giaco}.
 
 Although the above fractional moment estimate extends to the case $\bar\beta>0$, it is not 
 clear to us how the rare stretch strategy used in \cite{Gi07}, Section 6.3.2, can be 
 used to arrive at the lower bound $h^{\rm que}_c(\hb,\bar\beta,\bar h)\geq \hat 
 h^{\rm ann}_c(\hb/\alpha)$, for the case $\bar\beta>0$, since the polymer may hit pinning 
 disorder with  very large absolute value upon visiting or exiting  a rare stretch in the copolymer
 disorder   making the pinning contribution to the energy  of order greater than $O(1)$. 
 Moreover, it is not clear to us how to arrive at the lower bound  $\,h^{\rm que}_c(\hb,
 \bar\beta,\bar h)\geq \hat h^{\rm que}_c(\hat\beta),$ for the case $\bar\beta>0$, 
 based on the argument in \cite{BodHoOp11} that gave rise to the inequality in (1).

\medskip\noindent
{\bf 10.}
Corollary~\ref{critqueann}(i) shows that the critical curve for the annealed
combined model taken at the $\th$-value where the annealed copolymer model is 
critical coincides with the annealed critical curve of the pinning model, and
vice versa. For the quenched combined model a similar result is expected, but
this remains open. One of the questions that was posed in Giacomin \cite{Gi07}, 
Section 6.3.2, for the quenched combined model is whether an arbitrary small 
pinning bias $-\th>0$ can lead to localization for $\bb=0$, $\hb>0$ and $\hh
=\hat h^{\rm que}_c(\hb)$. This question is answered in the affirmative by 
Corollary~\ref{critqueann}(ii). 

\medskip\noindent
{\bf 11.}
Giacomin and Toninelli~\cite{GiTo06} showed that in $\cL^{\rm que}$ the longest 
excursion under the quenched path measure $P_n^{\hb,\hh,\bb,\th,\o}$ is of order 
$\log n$. No information was obtained about the path behavior in $\cD^{\rm que}$. 
Corollary~\ref{delocpathprop} says that in $\cD^{\rm que}$ (which is the region 
on or above the critial curve in Fig.~\ref{fig-copadvarhc}), with the exception 
of the piece of the critical curve over the interval $(-\infty,s_*(\hb,\bb,\alpha))$, 
the total number of visits to the interface up to time $n$ is at most of order $\log n$. 
On this piece, the number may very well be of larger order. Corollary~\ref{locpathprop} 
says that in $\cL^{\rm que}$ this number is proportional to $n$, with a variational 
formula for the proportionality constant. Since on the piece of the critical curve 
over the interval $[s_*(\hb,\bb,\alpha),\infty)$ the number is of order $\log n$, 
the phase transition is expected to be first order on this piece. 

\medskip\noindent
{\bf 12.}
Smoothness of the free energy in the localized phase, finite-size corrections, and 
a central limit theorem for the free energy can be found in \cite{GiTo06}. 
P\'etr\'elis~\cite{Pert09} studies the weak interaction limit of the combined model.


\subsection{Outline}
\label{S1.7}

The present paper uses ideas from Cheliotis and den Hollander~\cite{ChdHo10} and Bolthausen, 
den Hollander and Opoku~\cite{BodHoOp11}. The proof of Theorem \ref{freeenegvar} uses 
large deviation principles derived in Birkner~\cite{Bi08} and Birkner, Greven and den 
Hollander~\cite{BiGrdHo10}. The quenched variational formula and its proof are given in 
Section \ref{S3.1}, the annealed variational formula and its proof in Section~\ref{S3.2}. 
Section~\ref{S3.3} contains the proof of Theorem~\ref{freeenegvar}. The proofs of Corollaries 
\ref{freeeneggap}--\ref{locpathprop} are given in Sections~\ref{Sec4}--\ref{Sec6}. 
The latter require certain technical results, which are proved in 
Appendices~\ref{appA}--\ref{appc}.


\section{Large Deviation Principle (LDP)}
\label{S2}

Let $E$ be a Polish space, playing the role of an alphabet, i.e., a set of \emph{letters}. 
Let $\widetilde{E} = \cup_{k\in\N} E^k$ be the set of \emph{finite words} drawn from
$E$, which can be metrized to become a Polish space.

Fix $\nu \in \cP(E)$, and $\rho\in\cP(\N)$ satisfying (\ref{rhocond}). Let $X=(X_k)_{k\in\N}$
be i.i.d.\ $E$-valued random variables with marginal law $\nu$, and $\tau=(\tau_i)_{i\in\N}$
i.i.d.\ $\N$-valued random variables with marginal law $\rho$. Assume that $X$ and
$\tau$ are independent, and write $\P\otimes P^\ast$ to denote their joint law. Cut words 
out of the letter sequence $X$ according to $\tau$ (see Fig.~\ref{fig-cutting}), i.e., put
\begin{equation}
\label{Tdefs}
T_0=0 \quad \mbox{ and } \quad T_i=T_{i-1}+\tau_i,\quad i\in\N,
\end{equation}
and let
\begin{equation}
\label{eqndefYi}
Y^{(i)} = \bigl( X_{T_{i-1}+1}, X_{T_{i-1}+2},\dots, X_{T_{i}}\bigr),
\quad i \in \N.
\end{equation}
Under the law $\P\otimes P^\ast$, $Y = (Y^{(i)})_{i\in\N}$ is an i.i.d.\ sequence of words
with marginal distribution $q_{\rho,\nu}$ on $\widetilde{E}$ given by
\begin{equation}
\label{q0def}
\begin{aligned}
\P\otimes P^\ast\big(Y^{(1)}\in(dx_1,\dots,dx_n)\big)
&=q_{\rho,\bm}\big((dx_1,\dots,dx_n)\big) \\[0.2cm]
& = \rho(n) \,\nu(dx_1) \times \dots \times \nu(dx_n), \qquad
n\in\N,\,x_1,\dots,x_n\in E.
\end{aligned}
\end{equation}

\begin{figure}[htbp]
\vspace{-.7cm}
\begin{center}
\includegraphics[scale = .7]{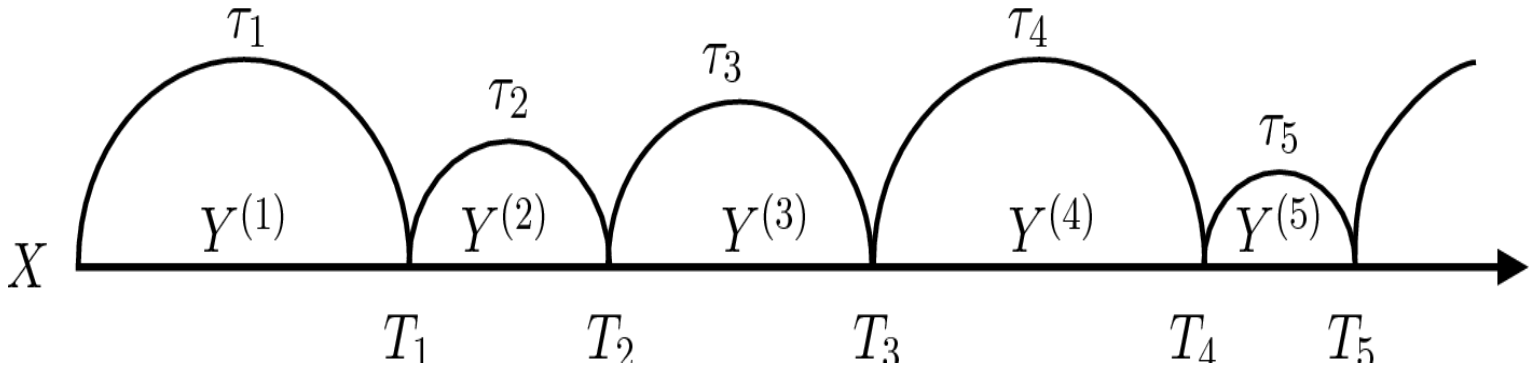}
\vspace{-17.1cm}
\caption{\small Cutting words out of a sequence of letters according to renewal times.}
\label{fig-cutting}
\end{center}
\vspace{-.2cm}
\end{figure}


The reverse operation of \emph{cutting} words out of a sequence of letters is 
\emph{glueing} words together into a sequence of letters. Formally, this is done 
by defining a \emph{concatenation} map $\kappa$ from $\widetilde{E}^\N$ to $E^{\N}$. 
This map induces in a natural way a map from $\cP(\widetilde{E}^\N)$ to $\cP(E^{\N})$,
the sets of probability measures on $\widetilde{E}^\N$ and $E^{\N}$ (endowed with
the topology of weak convergence). The concatenation $q_{\rho,\nu}^{\otimes\N}\circ
\kappa^{-1}$ of $q_{\rho,\nu}^{\otimes\N}$ equals $\nu^{\otimes\N}$, as is evident from 
(\ref{q0def}).


\subsection{Annealed LDP}
\label{S2.1}

Let $\cP^{\mathrm{inv}}(\widetilde{E}^\N)$ be the set of probability measures
on $\widetilde{E}^\N$ that are invariant under the left-shift $\widetilde{\theta}$
acting on $\widetilde{E}^\N$. For $N\in\N$, let $(Y^{(1)},\dots,Y^{(N)})^\mathrm{per}$ 
be the periodic extension of the $N$-tuple $(Y^{(1)},\dots,Y^{(N)})\in\widetilde{E}^N$ 
to an element of $\widetilde{E}^\N$. The \emph{empirical process of $N$-tuples of words} is 
 defined as
\begin{equation}
\label{eqndefRN}
R_N^X = \frac{1}{N} \sum_{i=0}^{N-1}
\delta_{\widetilde{\theta}^i (Y^{(1)},\dots,Y^{(N)})^\mathrm{per}}
\in \mathcal{P}^{\mathrm{inv}}(\widetilde{E}^\N),
\end{equation}
where the supercript $X$ indicates that the words $Y^{(1)},\dots,Y^{(N)}$ are cut from 
the latter sequence $X$. For $Q\in\cP^\mathrm{inv}(\widetilde{E}^\N)$, let $H(Q \mid
q_{\rho,\nu}^{\otimes\N})$ be the \emph{specific relative entropy of $Q$ w.r.t.\ 
$q_{\rho,\nu}^{\otimes\N}$} defined by
\begin{equation}
\label{spentrdef}
H(Q \mid q_{\rho,\nu}^{\otimes\N}) = \lim_{N\rightarrow\infty} \frac{1}{N}\,
h(\pi_N Q \mid q_{\rho,\nu}^N),
\end{equation}
where $\pi_N Q \in \cP(\widetilde{E}^N)$ denotes the projection of $Q$ onto the first
$N$ words, $h(\,\cdot\mid\cdot\,)$ denotes relative entropy, and the limit is
non-decreasing.

For the applications below we will need the following tilted version of $\rho$:  
\begin{equation}
\label{rhoz}
\rho_g(n)=\frac{e^{-gn}~\rho(n)}{\cN(g)} \qquad \text{with}
\qquad \cN(g)=\sum_{n\in\N} e^{-gn}\,\rho(n), \quad g\geq 0.
\end{equation}
Note that, for $g>0$, $\rho_g$ has a tail that is exponentially bounded. The following 
result relates the relative entropies with $q_{\rho_g,\nu}^{\otimes\N}$ and 
$q_{\rho,\nu}^{\otimes\N}$ as reference measures.

\begin{lemma} {\rm \cite{BodHoOp11}}
\label{rhozre}
For $Q\in\cP^\mathrm{inv}(\widetilde{E}^\N)$ and $g\geq0$,
\begin{equation}
\label{rhozrho}
H(Q\mid q_{\rho_g,\nu}^{\otimes\N})
= H(Q\mid q_{\rho,\nu}^{\otimes\N})+\log \cN(g)+g~\E_Q(\tau_1).
\end{equation}
\end{lemma}

\noindent
This result shows that, for $g\geq0$, $m_Q=\E_Q(\tau_1)<\infty$ whenever $H(Q\mid
q_{\rho_g,\nu}^{\otimes\N})<\infty$, which is a special case of \cite{Bi08}, Lemma 7.

The following \emph{annealed LDP} is standard (see e.g.\ Dembo and Zeitouni~\cite{DeZe98}, 
Section 6.5).
\begin{theorem}
\label{aLDP}
For every $g \geq 0$, the family $(\P\otimes P^\ast_g)(R_N^\cdot \in \cdot\,)$, 
$N\in\N$, satisfies the LDP on $\cP^{\mathrm{inv}}(\widetilde{E}^\N)$ with rate $N$ and 
with rate function $I_g^{\mathrm{ann}}$ given by
\begin{equation}
\label{Ianndef}
I_g^{\mathrm{ann}}(Q)= H\big(Q \mid q_{\rho_g,\nu}^{\otimes\N}\big), 
\qquad Q \in \cP^{\mathrm{inv}}(\widetilde{E}^\N).
\end{equation}
This rate function is lower semi-continuous, has compact level sets, has a unique
zero at $q_{\rho_g,\nu}^{\otimes\N}$, and is affine.
\end{theorem}

It follows from Lemma~\ref{rhozre} that 
\begin{equation}\label{Iannz}
I_g^{\rm ann}(Q)=I^{\rm ann}(Q)+\log\cN(g)+gm_Q,
\end{equation}
where $I^{\rm ann}(Q)=H(Q\mid q_{\rho,\nu}^{\otimes \N})$, the annealed rate function 
for $g=0$.


\subsection{Quenched LDP}
\label{S2.2}

To formulate the quenched analogue of Theorem~\ref{aLDP}, we need some more notation. 
Let $\cP^{\mathrm{inv}}(E^{\N})$ be the set of probability measures on $E^{\N}$ 
that are invariant under the left-shift $\theta$ acting on $E^{\N}$. For $Q\in
\cP^{\mathrm{inv}}(\widetilde{E}^\N)$ such that $m_Q < \infty$, define
\begin{equation}
\label{PsiQdef}
\Psi_Q = \frac{1}{m_Q} E_Q\left(\sum_{k=0}^{\tau_1-1}
\delta_{\theta^k\kappa(Y)}\right) \in \cP^{\mathrm{inv}}(E^{\N}).
\end{equation}
Think of $\Psi_Q$ as the shift-invariant version of $Q\circ\kappa^{-1}$ obtained
after \emph{randomizing} the location of the origin. This randomization is necessary
because a shift-invariant $Q$ in general does not give rise to a shift-invariant
$Q\circ\kappa^{-1}$.

For $\tr \in \N$, let $[\cdot]_\tr \colon\,\widetilde{E} \to [\widetilde{E}]_\tr
= \cup_{n=1}^\tr E^n$ denote the \emph{truncation map} on words defined by
\begin{equation}
\label{trunword}
y = (x_1,\dots,x_n) \mapsto [y]_\tr = (x_1,\dots,x_{n \wedge \tr}),
\qquad  n\in\N,\,x_1,\dots,x_n\in E,
\end{equation}
i.e., $[y]_\tr$ is the word of length $\leq\tr$ obtained from the word $y$ by dropping
all the letters with label $>\tr$. This map induces in a natural way a map from 
$\widetilde E^\N$ to $[\widetilde{E}]_\tr^\N$, and from $\cP^{\mathrm{inv}}(\widetilde{E}^\N)$ 
to $\cP^{\mathrm{inv}}([\widetilde{E}]_\tr^\N)$. Note that if $Q\in\cP^{\mathrm{inv}}
(\widetilde{E}^\N)$, then $[Q]_\tr$ is an element of the set
\begin{equation}
\label{Pfin} 
\cP^{\mathrm{inv,fin}}(\widetilde{E}^\N) = \{Q\in\cP^{\mathrm{inv}}(\widetilde{E}^\N)
\colon\,m_Q<\infty\}.  
\end{equation}

Define (w-lim means weak limit)
\begin{equation}
\label{Rdef}
\cR = \left\{Q\in\cP^{\mathrm{inv}}(\widetilde{E}^\N)\colon\,
{\rm w}-\lim_{N\to\infty} \frac{1}{N}
\sum_{k=0}^{N-1} \delta_{\theta^k\kappa(Y)}=\nu^{\otimes\N} \quad Q-a.s.\right\},
\end{equation}
i.e., the set of probability measures in $\cP^{\mathrm{inv}}(\widetilde{E}^\N)$ under
which the concatenation of words almost surely has the same asymptotic statistics as 
a typical realization of $X$.

\begin{theorem}
\label{qLDP}
{\rm (Birkner~\cite{Bi08}; Birkner, Greven and den Hollander~\cite{BiGrdHo10})}
Assume {\rm (\ref{rhocond}--\ref{mgffin})}. Then, for $\nu^{\otimes \N}$--a.s.\ 
all $X$ and all $g \in [0,\infty)$, the family of (regular) conditional probability 
distributions $P^\ast_g(R_N^X \in \cdot \mid X)$, $N\in\N$, satisfies the LDP on 
$\cP^{\mathrm{inv}}(\widetilde{E}^\N)$ with rate $N$ and with deterministic rate 
function $I^{\mathrm{que}}_g$ given by
\begin{equation}
\label{eqgndefinitionIalgz}
I^\mathrm{que}_g(Q) = \left\{\begin{array}{ll}
I^{\rm ann}_g(Q),
&\mbox{if } Q\in\cR,\\
\infty,
&\mbox{otherwise},
\end{array}
\right.
\text{ when } g>0,
\end{equation}
and
\begin{equation}
\label{eqgndefinitionIalg}
I^\mathrm{que}(Q) = \left\{\begin{array}{ll}
I^\mathrm{fin}(Q),
&\mbox{if } Q\in\cP^{\mathrm{inv,fin}}(\widetilde{E}^\N),\\
\lim_{\tr \to \infty} I^\mathrm{fin}\big([Q]_\tr\big),
&\mbox{otherwise},
\end{array}
\right.
\text{ when } g=0,
\end{equation}
where
\begin{equation}
\label{eqnratefctexplicitalg}
I^\mathrm{fin}(Q) = H(Q \mid q_{\rho,\nu}^{\otimes\N}) 
+ (\alpha-1) \, m_Q \,H\big(\Psi_{Q} \mid \nu^{\otimes\N}\big).
\end{equation}
This rate function is lower semi-continuous, has compact level sets, has a unique zero
at $q_{\rho_g,\nu}^{\otimes\N}$, and is affine.
\end{theorem}

It was shown in \cite{Bi08}, Lemma 2, that
\begin{equation}
\label{Requiv}
\Psi_{Q} =\nu^{\otimes\N} \quad \Longleftrightarrow \quad  
Q\in\cR \qquad \mbox{ on } \cP^{\mathrm{inv,fin}}(\widetilde{E}^\N),
\end{equation}
which explains why the restriction $Q\in\cR$ appears in 
\eqref{eqgndefinitionIalgz}. For more background, see \cite{BiGrdHo10}. 

Note that $I^\mathrm{que}(Q)$ requires a truncation approximation when 
$m_Q=\infty$, for which case there is no closed form expression like in
\eqref{eqnratefctexplicitalg}. As we will see later on, the cases 
$m_Q<\infty$ and $m_Q=\infty$ need to be separated. For later reference 
we remark that, for all $Q\in\cP^{\mathrm{inv}}(\widetilde{E}^\N)$,
\begin{equation}
\label{truncapproxcont}
\begin{aligned}
I^\mathrm{ann}(Q)
&= \lim_{\tr \to \infty} I^\mathrm{ann}([Q]_\tr),\\
I^\mathrm{que}(Q)
&= \lim_{\tr \to \infty} I^\mathrm{que}([Q]_\tr),
\end{aligned}
\end{equation}
as shown in \cite{BiGrdHo10}, Lemma A.1. 


\section{Variational formulas for excess free energies}
\label{S3}

This section uses the LDP of Section~\ref{S2} to derive variational formulas for the 
excess free energy of the quenched and the annealed version of the combined model. 
The quenched version is treated in Section~\ref{S3.1}, the  annealed version 
in Section~\ref{S3.2}. The results in Sections~\ref{S3.1}--\ref{S3.2} are used 
in Section~\ref{S3.3} to prove Theorem~\ref{freeenegvar}. 

In the combined model words are made up of letters from the alphabet $E=\hE\times\bE$, 
where $\hE$ and $\bE$ are subsets of $\R$, and are cut from the letter sequence $\o
=((\ho_i,\bo_i))_{i\in\N}$, where $\ho=(\ho_i)_{i\in\N}$ and $\bo=(\bo_i)_{i\in\N}$ 
are i.i.d.\ sequences of $\hE$-valued and $\bE$-valued random variables with joint
common law $\nu=\hm\otimes\bm$. Let $\hat\pi$ and $\bar\pi$ be the projection maps 
from $E$ onto $\hE$ and $\bE$, respectively, i.e $\hat\pi((\ho_1,\bo_1))=\ho_1$ and  
$\bar\pi((\ho_1,\bo_1))=\bo_1$ for $(\ho_1,\bo_1)\in E$. These maps extend naturally 
to $E^\N$, $\wE$, $\wE^\N$, $\cP\big(\wE\big)$ and $\cP\big(\wE^\N\big)$. For instance, 
if $\xi\in E^\N$, i.e., $\xi=((\ho_i,\bo_i))_{i\in\N}$, then $\hat\pi\xi=\ho=(\ho_i)_{i
\in\N}$ and $\bar\pi\xi=\bo=(\bo_i)_{i\in\N}$.

As before, we will write $k$, $\hat k$ and $\bar k$ for a quantity $k$ associated with 
the copolymer with pinning model, the copolymer model, respectively, the pinning model. 
For instance, if $Q\in\cP^{\rm inv}\big(\wE^\N\big)$, $\hQ\in\cP^{\rm inv}\big(\whE^\N
\big)$ and $\bQ\in\cP^{\rm inv}\big(\wbE^\N\big)$, then the rate functions $I^{\rm ann}
(Q)=H(Q|q_{\rho,\hm\otimes\bm}^{\otimes\N})$, $\hat I^{\rm ann}(\hQ)=H(\hQ|q_{\rho,
\hm}^{\otimes\N})$, $\bar I^{\rm ann}(\bQ)=H(\bQ|q_{\rho,\bm}^{\otimes\N})$ and the 
sets $\cR$, $\hcR$, $\bar\cR$ are defined as in \eqref{Rdef}.
 
The LDPs of the laws of the empirical processes $R^{\ho}_N=\hat\pi R^{\o}_N$ and 
$R^{\bo}_N=\bar\pi R^{\o}_N$ can be derived from those of $R^{\o}_N$ via the 
contraction principle (see e.g.\ Dembo and Zeutouni~\cite{DeZe98}, Theorem 4.2.1), 
because the projection maps $\hat\pi$ and $\bar\pi$ are continuous. In particular, 
for any $\hQ\in\cP^{\rm inv}\big(\whE^\N\big)$ and  $\bQ\in\cP^{\rm inv}\big(\wbE^\N
\big)$
\begin{equation}
\label{conprin}
\hat I^{\rm que}(\hQ)
= \inf_{Q\in\cP^{\rm inv}\big(\wE^\N\big)\colon\,\atop\hat\pi Q=\hQ } \Iq(Q), 
\qquad
\bar I^{\rm que}(\bQ) 
= \inf_{Q\in\cP^{\rm inv}\big(\wE^\N\big)\colon\,\atop\bar\pi Q=\bQ } 
I^{\rm que}(Q),
\end{equation}
where $\hat\pi Q=Q\circ(\hat\pi)^{-1}$ and $\bar\pi Q=Q\circ(\bar\pi)^{-1}$. Similarly, 
we may express $\hat I^{\rm ann}$ and $\bar I^{\rm ann}$ in terms of $I^{\rm ann}$.


\subsection{Quenched excess free energy}
\label{S3.1}

Abbreviate
\begin{equation}
\label{Csetdef}
\cC^{\rm fin} = \left\{Q\in\cP^\mathrm{inv}(\widetilde{E}^\N)\colon\,
I^\mathrm{ann}(Q)<\infty, \, m_Q<\infty\right\}. 
\end{equation}

\begin{theorem}
\label{varfloc}
Assume {\rm (\ref{rhocond})} and {\rm (\ref{mgffin})}. Fix $\hb$, $\hh>0$, $\bb\geq0$ 
and $\th\in\R$.\\
(i) The quenched excess free energy is given by
\begin{equation}
\label{gvarexp}
g^\mathrm{que}(\hb,\hh,\bb,\th) 
= \inf\left\{g\in\R\colon\,S^\mathrm{que}(\hb,\hh,\bb;g)-\th<0\right\},
\end{equation}
where
\begin{equation}
\label{Sdef}
S^\mathrm{que}(\hb,\hh,\bb;g) = \sup_{Q\in\cC^\mathrm{fin}\cap\cR}
\left[\bb \Phi(Q)+\Phi_{\hb, \hh}(Q)-g m_Q-I^\mathrm{ann}(Q)\right]
\end{equation}
with
\begin{eqnarray}
\label{pinPhidef}
\Phi(Q) &=& \int_{\bE} \bo_1\, (\bar\pi_{1,1}Q)(d\bo_1)\\
\label{coPhidef}
\Phi_{\hb,\hh}(Q) 
&=& \int_{\widetilde{\hE}} (\hat\pi_1 Q)(d\ho)\,\log \phi_{\hb, \hh}(\ho),\\
\label{phidef}
\phi_{\hb, \hh}(\ho) 
&=& \tfrac12\left(1+\exp\left[-2\hb \hh\,\tau_1-2\hb\,\sum_{k=1}^{\tau_1}
\ho_k\right]\right).
\end{eqnarray}
Here, the map $\bar\pi_{1,1}\colon\,\wE^\N\to \bE$ is the projection onto the first 
letter of the first word in the sentence consisting of words cut out from $\bo$, i.e., 
$\bar\pi_{1,1}Q=Q\circ(\bar\pi_{1,1})^{-1}$, while the map $\hat\pi_1\colon\,\wE^\N\to\whE$ 
is the projection onto the first word in the sentence consisting of words cut out 
from $\ho$, i.e., $\hat\pi_1 Q=Q\circ(\hat\pi_1)^{-1}$, and $\tau_1$ is the length
of the first  word.\\ 
(ii) An alternative variational formula at $g=0$ is $S^\mathrm{que}(\hb,\hh,\bb;0)
=S_{*}^\mathrm{que}(\hb,\hh,\bb)$ with
\begin{equation}
\label{Sdefalt}
S_*^\mathrm{que}(\hb,\hh,\bb) = \sup_{Q\in\cC^\mathrm{fin}} 
\left[\bb \Phi(Q)+\Phi_{\hb, \hh}(Q)-I^\mathrm{que}(Q)\right].
\end{equation}
(iii) The map $g \mapsto S^\mathrm{que}(\hb,\hh,\bb;g)$ is lower semi-continuous, 
convex and non-increasing on $\R$, is infinite on $(-\infty,0)$, and is finite, continuous 
and strictly decreasing on $(0,\infty)$. 
\end{theorem}

\begin{proof}
The proof is an adaptation of the proof of Theorem 3.1 in \cite{BodHoOp11} and comes
in 3 steps. 

\medskip\noindent
{\bf 1.} 
Suppose that $\pi\in\Pi_n$ has $t_n=t_n(\pi)$ excursions away from the interface. If 
$k_i$ denote the times at which $\pi$ visits the interface, then the Hamiltonian reads  
\begin{equation}
\label{copadHamexcsplit}
\begin{split}
H_n^{\hb, \hh,\bb,\th,\o}(\pi)
&=\hb\sum_{k=1}^n(\ho_k+\hh)\left[{\rm sign}(\pi_{k-1},\pi_k)-1
\right]+\sum_{k=1}^{n} (\bb \bo_k-\th)1_{\{\pi_k=0\}}\cr
&= \sum_{i=1}^{t_n}\left[\bb \bo_{k_{i}}-\th-2\,\hb\,1_{A^-_i}
\sum_{k\in I_i}(\ho_k+\hh) \right], 
\end{split}
\end{equation} 
where $A^-_i$ is the event that the $i$-th excursion is below the interface and $I_i
=(k_{i-1},k_i]\cap\N$. Since each excursion has equal probability to lie below or 
above the interface, the $i$-th excursion contributes 
\begin{equation}
\label{ithexcontri}
\phi_{\hb, \hh}(\ho_{I_i})\,e^{\bb \bo_{k_{i}}-\th}
=\tfrac12\left(1+\exp\left[-2\,\hb\sum_{k\in I_i}
(\ho_k+\hh)\right]\right)\,e^{\bb \bo_{k_{i}}-\th} 
\end{equation} 
to the partition sum $Z_{n}^{\hb, \hh,\bb,\th,\o}$, where $\ho_{I_i}$ is the word in 
$\whE$ cut out from $\ho$ by the $i$-th excursion interval $I_i$. Consequently, we have 
\begin{equation}
Z_{n}^{\hb, \hh,\bb,\th,\o} = \sum_{N\in\N}\,\,\sum_{0=k_0<k_1<\cdots<k_N=n}\,\, 
\prod_{i=1}^N \rho(k_i-k_{i-1})\,e^{(\bb \bo_{k_{i}}-\th)}
e^{\log \phi_{\hb, \hh}(\ho_{I_i})}.
\end{equation}
Therefore, summing over $n$, we get
\begin{equation}
\label{copadZFNrel}
\sum_{n\in\N}  Z_{n}^{\hb, \hh,\bb,\th,\o} e^{-gn} 
= \sum_{N\in\N} F_N^{\hb, \hh,\bb,\th,\o}(g),
\qquad g\geq0,
\end{equation}
with
\begin{equation}
\label{FNdef}
\begin{aligned}
F_N^{\hb, \hh,\bb,\th,\o}(g)&=
e^{-N\th}
\sum_{0=k_0<k_1<\cdots<k_N<\infty}\,\, 
\left(\prod_{i=1}^N \rho(k_i-k_{i-1})\right)\,e^{-g(k_i-k_{i-1})}\\
&\qquad\qquad\times
\exp\left[\sum_{i=1}^N \left( \log \phi_{\hb, \hh}(\ho_{I_i})
+\bb~\bo_{k_{i}}\right)\right]\\ 
&= \left(\mathcal{N}(g)\,e^{-\th}\right)^N
\sum_{0=k_0<k_1<\cdots<k_N<\infty}\,\, 
\left(\prod_{i=1}^N \rho_g(k_i-k_{i-1})\right)\\
&\qquad\qquad\times
\exp\left[\sum_{i=1}^N \left( \log \phi_{\hb, \hh}(\ho_{I_i})
+\bb~\bo_{k_{i}}\right)\right]\\
&= \left(\mathcal{N}(g)\,e^{-\th}\right)^N
\,E^*_g\left(\exp\left [N\left(\Phi_{\hb, \hh}(R_N^{\o})+\bb\,
\Phi(R_N^{\o})\right)\right]\right),
\end{aligned}
\end{equation}
where
\begin{equation}
\label{copadempprocomega}
R_N^{\o}\big((k_i)_{i=0}^N\big) = \frac{1}{N} \sum_{i=1}^N 
\delta_{\widetilde{\theta}^i(\o_{I_1},\ldots,\o_{I_N})^{\mathrm{per}}}
\end{equation}
denotes the \emph{empirical process of $N$-tuples of words} cut out from $\o$ by 
the $N$ successive excursions, and $\Phi_{\hb,\hh},\Phi$ are defined in 
(\ref{pinPhidef}--\ref{phidef}).

\medskip\noindent
{\bf 2.}
The left-hand side of (\ref{copadZFNrel}) is a power series with radius of convergence 
$g^{\rm que}(\hb,\hh,\bb,\th)$ (recall \eqref{copadefe}). Define
\begin{equation}
\label{Slimdef}
\begin{aligned}
s^\mathrm{que}(\hb,\hh,\bb;g)  
&= \log \mathcal{N}(g) + \limsup_{N\to\infty} \frac{1}{N} \log 
E^*_g\left(\exp\left [N\left(\Phi_{\hb, \hh}(R_N^{\o})+\bb\,
\Phi(R_N^{\o})\right)\right]\right)
\end{aligned}
\end{equation}
and note that the limsup exists and is constant (possibly infinity) $\omega$-a.s.\ 
because it is measurable w.r.t.\ the tail sigma-algebra of $\omega$ (which is trivial). 
Note from \eqref{FNdef} and \eqref{Slimdef} that
\begin{equation}
\label{copadSlim}
\begin{aligned}
s^\mathrm{que}(\hb,\hh,\bb;g)-\th 
&= \limsup_{N\to\infty} \frac{1}{N} \log F_N^{\hb, \hh,\bb,\th,\o}(g).
\end{aligned}
\end{equation}
By \eqref{copadefe}, the left-hand side of (\ref{copadZFNrel}) is a power series that 
converges for $g>g^\mathrm{que}(\hb,\hh,\bb,\th)$ and diverges for $g<g^\mathrm{que}
(\hb,\hh,\bb,\th)$. Further, it follows from the first equality of \eqref{FNdef} and 
\eqref{copadSlim}, that the map $g\mapsto s^\mathrm{que}(\hb,\hh,\bb;g)$ is non-increasing. 
In particular, it is strictly decreasing when finite. This we show in Step 3 below. 
Hence we have
\begin{equation}
\label{roc}
g^\mathrm{que}(\hb,\hh,\bb,\th) 
= \inf\left\{g\in\R\colon\,s^\mathrm{que}(\hb,\hh,\bb;g)-\th<0\right\}.
\end{equation}

\medskip\noindent
{\bf 3.}
We claim that, for any $\hb, \hh>0$ and  $\bb\geq0$, the map $g\mapsto \bar 
S^\mathrm{que}(\hb,\hh,\bb;g)$ is finite on $(0,\infty)$ and infinite on 
$(-\infty,0)$ (see Fig.~\ref{fig-varfe}), and 
\begin{equation} 
\label{seqsbar}
s^\mathrm{que}(\hb,\hh,\bb;g)= S^\mathrm{que}(\hb,\hh,\bb;g)\quad \forall\, g\in\R.
\end{equation}
Note from the contraction principle in \eqref{conprin} that $\hat I^{\rm ann}(\hat\pi Q)$ 
and $\bar I^{\rm ann}(\bar\pi Q)$ are finite whenever $I^{\rm ann}(Q)<\infty$. Therefore, 
for any $\hb>0$, $\bb\geq0$ and $\hh>0$, it follows from Lemmas~\ref{pinphifin} and 
\ref{finitephicop} in Appendix~\ref{appA} that $\bb\Phi(Q)+\Phi_{\hb, \hh}(Q)<\infty$ 
whenever $I^{\rm ann}(Q)<\infty$. This implies that the map $g\mapsto S^\mathrm{que}
(\hb,\hh,\bb;g)$ is convex and lower semi-continuous, since, by \eqref{Sdef}, 
$S^\mathrm{que}(\hb,\hh,\bb;g)$ is the supremum of a family of functions that are 
finite and linear (and hence continuous) in $g$. Now the fact that $g\mapsto 
S^\mathrm{que}(\hb,\hh,\bb;g)$ is strictly 
decreasing when finite follows as follows: Suppose that $g_1<g_2$, and 
$S^\mathrm{que}(\hb,\hh,\bb;g_1)<\infty$. Then it follows from the fact $m_Q\geq 1$ 
and \eqref{Sdef} that 
\begin{equation}
S^\mathrm{que}(\hb,\hh,\bb;g_2)- S^\mathrm{que}(\hb,\hh,\bb;g_1)\leq-(g_2-g_1)<0.
\end{equation}
Further, we will show in Section \ref{Sec6.1} that $s^\mathrm{que}(\hb,\hh,\bb;g)<\infty$, 
for all $g>0$. This and convexity imply continuity on $(0,\infty).$ These prove 
(iii) of the theorem.

The rest of the proof follows from the claim in \eqref{seqsbar}, whose 
proof we defer to Appendix~\ref{appB}.

\end{proof}

\begin{figure}[htbp]
\vspace{1.5cm}
\begin{minipage}[hbt]{6cm}
\centering
\setlength{\unitlength}{0.3cm}
\begin{picture}(8,8)(-7,-2.2)
\put(-2,0){\line(10,0){10}}
\put(0,-6){\line(0,10){14}}
{\thicklines
\qbezier(2,-2.5)(0,1.5)(0,4)
\qbezier(-3,7)(-1.5,7)(-.2,7)
}
\put(8.5,-0.2){$g$}
\put(-2,9.5){$S^\mathrm{que}(\hb,\hh,\bb;g)$}
\put(-2,7.5){$\infty$}
\put(0,4){\circle*{.5}}
\put(0,7){\circle{.5}}
\put(-3,-.2){$\th$}
\end{picture}  
\vspace{1.2cm}
\begin{center}
\qquad (1) $\hat h_c^{\rm ann}(\frac{\hb}{\a})<\hh < h^\mathrm{que}_c(\hb,\bb,\th)$
\end{center}
\end{minipage}
\begin{minipage}[hbtb]{4.3cm}
\centering
\setlength{\unitlength}{0.3cm}
\begin{picture}(8,8)(-7,-2.2)
\put(-2,0){\line(8,0){10.5}}
\put(0,-6){\line(0,12){14}}
{\thicklines
\qbezier(2.5,-4)(1,-2.5)(0,0)
\qbezier(-3,7)(-1.5,7)(-.2,7)
}
\put(9,-0.2){$g$}
\put(-2,9.5){$S^\mathrm{que}(\hb,\hh,\bb;g)$}
\put(-2,7.5){$\infty$}
\put(0,0){\circle*{.5}}
\put(0,7){\circle{.5}}
\put(-3,-.2){$\th$}
\end{picture}
\vspace{1.2cm}
\begin{flushright}
\qquad (2) $\hh = h^\mathrm{que}_c(\hb,\bb,\th)$
\end{flushright}
\end{minipage}
\begin{minipage}[hbt]{4.3cm}
\centering
\setlength{\unitlength}{0.3cm}
\begin{picture}(8,8)(-7,-2.2)
\put(-2,0){\line(8,0){10}}
\put(0,-6){\line(0,14){14}}
{\thicklines
\qbezier(2.5,-6.5)(1,-5.5)(0,-2)
\qbezier(-3,7)(-1.5,7)(-.2,7)
}
\put(8.5,-0.2){$g$}
\put(-2,9.5){$S^\mathrm{que}(\hb,\hh,\bb;g)$}
\put(-2,7.5){$\infty$}
\put(0,-2){\circle*{.5}}
\put(0,7){\circle{.5}}
\put(-3,-.2){$\th$}
\end{picture}
\vspace{1.2cm}
\begin{flushright}
\qquad (3)  $\hh > h^\mathrm{que}_c(\hb,\bb,\th)$
\end{flushright}
\end{minipage}
\vspace{-.4cm}
\begin{center}
\caption{\small Qualitative picture of the map $g \mapsto S^\mathrm{que}(\hb,\hh,\bb;g)$ 
for $\hb,\hh>0$,  $\bb\geq0$. The horizontal axis is located at the  point $\bar h\leq  
S^\mathrm{que}(\hb,\hh,\bb;0)$ on 
the vertical axis. The map $g\mapsto S^\mathrm{que}(\hb,\hh,\bb;g)$ is strictly decreasing when finite.  For 
$\hh = h^\mathrm{que}_c(\hb,\bb,\th)$ and $\th> S^\mathrm{que}(\hb,\hh,\bb;0),$  
$S^\mathrm{que}(\hb,\hh,\bb;g)$ jumps at $g=0$ from a value below $\th$ to 
infinity. Since the map $g\mapsto S^\mathrm{que}(\hb,\hh,\bb;g)$ is lower semi-continuous 
and decreasing,  it follows that $\lim_{g\downarrow 0}S^\mathrm{que}(\hb,\hh,\bb;g)=S^\mathrm{que}(\hb,\hh,\bb;0)$, 
which is finite if $\hh>\hat h_c^{\rm ann}(\frac{\hb}{\a})$ and infinite if 
$\hh<\hat h_c^{\rm ann}(\frac{\hb}{\a})$ (see Lemma 5.2).  
For $\hh<\hat h_c^{\rm ann}(\frac{\hb}{\a})$, $g \mapsto S^\mathrm{que}(\hb,\hh,\bb;g)$ has the shape 
as in (1) but  tends to infinity as $g\downarrow0$ . For $\hh=\hat h_c^{\rm ann}(\frac{\hb}{\a})$ the picture 
is the same as for $\hh<\hat h_c^{\rm ann}(\frac{\hb}{\a})$ if $s^\ast(\hb,\hh,\alpha)=\infty$,  and takes the 
shape in (1) if $s^\ast(\hb,\hh,\alpha)<\infty$.}
\label{fig-varfe}
\end{center}
\vspace{-.5cm}
\end{figure}

Analogues of Theorem~\ref{varfloc} also hold for the copolymer model and the pinning 
model. The copolymer analogue is obtained by putting $\bb=\th=0$, which leads to
analogous variational formulas for $\hat S^{\rm que}(\hb,\hh;g)$ and $\hat g^{\rm que}
(\hb,\hh)$. In the variational formula for $\hat S^{\rm que}(\hb,\hh;g)$ we replace 
$\cC^{\rm fin}\cap\cR$ by $\hat\cC^{\rm fin}\cap\hat\cR$ in \eqref{Sdef}. This 
replacement is a consequence of the contraction principle in \eqref{conprin}. Although 
the contraction principle holds on $\cP^{\rm inv}(\wE^\N)$, it turns out that the $Q\notin
\cC^{\rm fin}\cap\cR$ play no role in \eqref{Sdef}. Similarly, Theorem~\ref{varfloc} 
reduces to the pinning model upon putting $\hb=\hh=0$. The variational formula for 
$\bar S^{\rm que}(\bb;g)$ is the same as that in \eqref{Sdef}, with  $\cC^{\rm fin}
\cap\cR$ replaced by $\bar\cC^{\rm fin}\cap\bar\cR$.


\subsection{Annealed excess free energy}
\label{S3.2}

We next present the variational formula for the annealed excess free energy. This will 
serve as an \emph{object of comparison} in our study of the quenched model. Define
\begin{equation}
\label{cndef}
\cN(\hb, \hh,\bb;g) = \tfrac12e^{\tM(-\bb)}\left
(\sum_{n\in\N} \rho(n)\,e^{-ng}+\sum_{n\in\N}\rho(n)\,
e^{-n\left(g-[\hM(2\,\hb)-2\,\hb\hh]\right)}\right)
\end{equation} 
(recall \eqref{mgffin}).

\begin{theorem}
\label{varflocann}  
Assume {\rm (\ref{rhocond})} and {\rm (\ref{mgffin})}. Fix $\hb,\hh,\bb\geq0$ 
and $\th\in\R$.\\
(i) The annealed excess free energy is given by
\begin{equation}
\label{ganvarexp}
g^\mathrm{ann}(\hb,\hh,\bb,\th) 
= \inf\left\{g\in\R\colon\,S^\mathrm{ann}(\hb,\hh,\bb;g)-\th<0\right\},
\end{equation}
where
\begin{equation}
\label{Sandef}
S^\mathrm{ann}(\hb,\hh,\bb;g) 
= \sup_{Q\in\cC^\mathrm{fin}}
\left[\bb \Phi(Q)+\Phi_{\hb, \hh}(Q)-g m_Q-I^\mathrm{ann}(Q)\right].
\end{equation}
(ii) The map $g \mapsto S^\mathrm{ann}(\hb,\hh,\bb;g)$ is lower semi-continuous, 
convex and non-increasing on $\R$. Furthermore, it is infinite on $(-\infty,
\hat{g}^{\rm ann}(\hb,\hh))$, and finite, continuous and strictly decreasing 
on $[\hat{g}^{\rm ann}(\hb,\hh),\infty)$ (recall \eqref{copannfeg}). 
\end{theorem}

\begin{proof}
The proof comes in 3 steps.

\medskip\noindent
Replacing $Z^{\hb, \hh,\bb,\th,\o}_n$ by $ Z^{\hb, \hh,\bb,\th}_n=\E(Z^{\hb, \hh,\bb,
\th,\o}_n)$ in \eqref{copadZFNrel}, we obtain from \eqref{FNdef} that 
\begin{equation}
\label{FNdefann}
F_N^{\hb, \hh,\bb,\th}(g) = \E\left(F_N^{\hb, \hh,\bb,\th,\o}(g)\right)
= \cN(\hb, \hh,\bb;g)^N e^{-\th N}.
\end{equation}
It therefore follows from \eqref{copadSlim} and \eqref{FNdefann} that 
\begin{equation}
\begin{split}
\label{copadsa1}
s^{\rm ann}(\hb,\hh,\bb;g)-\th
&=\limsup_{N\rightarrow\infty}\frac1N\log F_N^{\hb, \hh,\bb,\th}(g)
=\log \cN(\hb, \hh,\bb;g)-\th,
\end{split}
\end{equation}
where 
\begin{equation}
\begin{split}\label{copadsdef}
s^{\rm ann}(\hb,\hh,\bb;g)
=\limsup_{N\rightarrow\infty}\frac1N\log \left(e^{N\th}F_N^{\hb, \hh,\bb,\th}(g)\right)
=\log \cN(\hb, \hh,\bb;g).
\end{split}
\end{equation}
Note from \eqref{cndef} and \eqref{copadsdef} that the map $g\mapsto s^{\rm ann}(\hb,\hh,
\bb;g)$ is non-increasing. Moreover, for any $\hb,\hh,\bb\geq 0$ and $\th\in\R$, we see
from \eqref{copadZFNrel} after replacing $ Z^{\hb, \hh,\bb,\th,\o}_n$ by $ Z^{\hb,\hh,\bb,
\th}_n$ that $g^{\rm ann}(\hb,\hh,\bb,\th)$ is the smallest $g$-value at which $s^{\rm ann}
(\hb,\hh,\bb;g)-\th$ changes sign, i.e.,
\begin{equation}
g^{\rm ann}(\hb,\hh,\bb,\th)
=\inf\left\{g\in\R\colon s^{\rm ann}(\hb,\hh,\bb;g)-\th<0 \right\}.
\end{equation}
The proof of (i) and (ii) will follow once we show that
\begin{equation}
\label{anneqvfm}
S^{\rm ann}(\hb,\hh,\bb;g)=s^{\rm ann}(\hb,\hh,\bb;g) \quad \forall\,g\in\R,
\end{equation}
since \eqref{cndef}, \eqref{copadsdef} and \eqref{anneqvfm} show that the map $g\mapsto 
S^{\rm ann}(\hb,\hh,\bb;g)$ is infinite whenever $g<\hat{g}^{\rm\, ann}(\hb,\hh) =
0\vee[M(2\hb)-2\hb \hh]$, and is finite otherwise. Lower semi-continuity and convexity 
of this map follow from \eqref{Sandef}, because the function under the supremum is 
linear and finite in $g$, while convexity and finiteness imply continuity. The proof 
of \eqref{anneqvfm} follows from the arguments in \cite{BodHoOp11}, Theorem 3.2, as we
show in steps 2--3. 

\medskip\noindent
{\bf 2.} 
For the case $g<\hat{g}^{\rm ann}(\hb,\hh)$, note from \eqref{cndef} that $\cN(\hb,
\hh,\bb;g)=\infty$ for all $\hb,\hh,\bb\geq0$ and $\th\in\R$. To show that $S^{\rm ann}
(\hb,\hh,\bb;g)=\infty$ for this case, we proceed as in steps (II) and (III) of the proof
of \cite{BodHoOp11}, Theorem 3.2, by evaluating the functional under the supremum 
in \eqref{Sandef} at $Q^L_{\hb}=(q_{\hb}^L)^{\otimes\N}$ with
\begin{equation}
q_{\hb}^L(d(\ho_1,\bo_1),\ldots,d(\ho_n,\bo_n))
=\delta_{Ln}\left[\hm_{\hb}(d\ho_1)\times\cdots\times\hm_{\hb}(d\ho_n)\right]
\times\left[\bm(d\bo_1)\times\cdots\times\bm(d\bo_n)\right],
\end{equation}
where $L,n\in\N$, $\ho_1,\ldots,\ho_n\in\hE$, $\bo_1,\ldots,\bo_n\in\bE$, and (recall 
\eqref{mgffin})
\begin{equation}
\hm_{\hb}(d\ho_1)=e^{-2\hb\ho_1- \hM(2\hb)}\hm(d\ho_1).
\end{equation}
Note from \eqref{pinPhidef} that $\Phi(Q^L_{\hb})=0$ because $\bm$ has zero mean. This 
leads to a lower bound on $S^{\rm ann}(\hb,\hh,\bb;g)$ that tends to infinity as $L\to
\infty$. To get the desired lower bound, we have to distinguish between the cases 
$\hat g^{\rm ann}(\hb,\hh)=0$ and $\hat g^{\rm ann}(\hb,\hh)>0$. For $\hat g^{\rm ann}
(\hb,\hh)=0$ use $Q_0^L$, for $\hat g^{\rm ann}(\hb,\hh)>0$ with $\hb>0$ use $Q_{\hb}^L$.

\medskip\noindent
{\bf 3.}
For the case $g\geq\hat{g}^{\rm ann}(\hb,\hh)$, we proceed as in step 1 and 2 of the 
proof of Theorem 3.2 of \cite{BodHoOp11}. Note that $\Phi_{\hb,\hh}(Q)$ and $\Phi(Q)$ 
defined in (\ref{pinPhidef}--\ref{phidef}) are functionals of $\pi_1 Q$, where $\pi_1 Q$ 
is the first-word marginal of $Q$. Moreover, by (\ref{spentrdef}),
\begin{equation}
\label{Hiidred}
\inf_{ {Q\in\cP^{\mathrm{inv}}(\widetilde{E}^\N)} \atop {\pi_1 Q=q} } 
H(Q \mid q_{\rho,\hm\otimes\bm}^{\otimes\N}) 
= h(q \mid q_{\rho,\hm\otimes\bm}) \qquad \forall\,q \in \cP(\widetilde{E})
\end{equation}  
with the infimum \emph{uniquely} attained at $Q = q^{\otimes\N}$, where the right-hand 
side denotes the relative entropy of $q$ w.r.t.\ $q_{\rho,\hm\otimes\bm}$. (The uniqueness 
of the minimum is easily deduced from the strict convexity of relative entropy on 
finite cylinders.) Consequently, the variational formula in \eqref{Sandef} becomes
\begin{equation}
\label{varredann}
\begin{split}
S^\mathrm{ann}(\hb,\hh,\bb;g)
&= \sup_{ {q \in \cP(\widetilde{E})} \atop 
{m_q<\infty,\,h(q \mid q_{\rho,\hm\otimes\bm})<\infty} } 
\Big\{\int_{\widetilde{E}} q(d\o)\,[\bb \bo_1+\log\phi_{\hb, \hh}(\ho)] 
-gm_Q -h(q \mid q_{\rho,\hm\otimes\bm})\Big\}\cr
&=\sup_{ {q \in \cP(\widetilde{E})} \atop 
{m_q<\infty,\,h(q \mid q_{\rho,\hm\otimes\bm})<\infty} } 
\Big\{ \int_{\widetilde{E}} q(d\o)\,[\bb \bo_1+\log\phi_{\hb, \hh}(\ho)-g\tau(\o)]\\ 
&\qquad\qquad\qquad\qquad\qquad - \int_{\widetilde{E}} q(d\o) 
\log \left(\frac{q(d\o)}{q_{\rho,\hm\otimes\bm}(d\o)}\right)\Big\} \cr
&=\bar M(-\bb)+ \log\hat\cN(\hb, \hh;g)-\inf_{ {q \in \cP(\widetilde{E})} 
\atop {m_q<\infty,\,h(q \mid q_{\rho,\hm\otimes\bm})<\infty} } 
h(q \mid q_{\hb, \hh,\bb;g}),
\end{split}
\end{equation}
where (by an abuse of notation) $\o=((\ho_i,\bo_i))_{i=1}^{\tau(\o)}$ is the disorder in 
the first word, $\phi_{\hb, \hh}(\ho)$ is defined in (\ref{phidef}), $m_q=\int_{\widetilde
{E}} q(d\o)\tau(\o)$, $\tau(\o)$ is the length of the word $\o$, and
\begin{equation}
\label{uniqmax}
\begin{split}
q_{\hb, \hh,\bb;g}(d(\ho_1,\bo_1),\cdots,d(\ho_n,\bo_n))
&=\frac{\rho(n)\phi_{\hb, \hh}(\ho)e^{\bb \bo_1-ng}}
{\hat\cN(\hb, \hh;g)e^{\bar M(-\bb)}}(\hm\otimes\bm)^n
(d(\ho_1,\bo_1),\cdots,d(\ho_n,\bo_n)),\cr
\hat\cN(\hb, \hh;g)
&=\tfrac12\left[\sum_{n\in\N}\rho(n)e^{-ng}
+\sum_{n\in\N}\rho(n)e^{-n(g-[M(2\hb)-2\hb \hh])}\right].
\end{split}
\end{equation}
Note from \eqref{cndef} that $\cN(\hb, \hh,\bb;g)=\hat\cN(\hb, \hh;g)e^{\bar M(-\bb)}$.
The infimum in the last equality of \eqref{varredann} is uniquely attained at $q=q_{\hb,
\hh,\bb;g}$. Therefore the variational problem in \eqref{Sandef} for $g\geq\hat{g}^{\rm
ann}(\hb,\hh)$ takes the form
\begin{equation}
\label{scopadf}
\begin{split}
S^{\rm ann}(\hb,\hh,\bb;g)
&=\log\left(\tfrac12\left(\sum_{n\in\N}\rho(n)
e^{-gn}+\sum_{n\in\N}\rho(n)e^{-n\left(g-[\hM(2\,\hb)-2\,\hb\hh]\right)}\right) 
e^{\bar M(-\bb)}\right)\cr
&=\log \cN(\hb, \hh,\bb;g) = s^{\rm ann}(\hb,\hh,\bb;g).
\end{split}
\end{equation}
The last formula proves \eqref{ganncomb}. 
\end{proof}

As in the quenched model, there are analogous versions of Theorem~\ref{varflocann} 
for the annealed copolymer model and the annealed pinning model. These are obtained 
by putting either $\bb=\th=0$ or $\hb=\hh=0$, replacing $\cC^\mathrm{fin}$ 
by $\hat\cC^\mathrm{fin}$ and $\bar\cC^\mathrm{fin}$, respectively. The copolymer 
version of Theorem~\ref{varflocann} was derived in \cite{BodHoOp11}, Theorem 3.2, 
and the pinning version (for $g=0$ only) in \cite{ChdHo10}, Theorem 1.3. 

Putting $\bb=\th=0$, we get the copolymer analogue of \eqref{scopadf}:
\begin{equation}
\label{Scopann}
\hat S^{\rm ann}(\hb,\hh;g)
=\log\left(\tfrac12\left[\sum_{n\in\N}\rho(n)\,e^{-n g}
+ \sum_{n\in\N}\rho(n)\,e^{-n(g-[\hM(2\,\hb)-2\,\hb\,\hh])}\right]\right).
\end{equation}
This expression, which was obtained in \cite{BodHoOp11}, is plotted in 
Fig.~\ref{fig-varfe1}. Putting $\hb=\hh=0$, we get the pinning analogue:
\begin{equation}
\label{Sanpin}
\begin{split}
\bar S^{\rm ann}(\bb;g)
&= \bar M(-\bb)+\log\left(\sum_{n\in\N} \rho(n)\,e^{-ng}\right).
\end{split}
\end{equation}

\begin{figure}[htbp]
\vspace{2cm}
\begin{minipage}[hbt]{5cm}
\centering
\setlength{\unitlength}{0.3cm}
\begin{picture}(8,8)(-7,-2.2)
\put(0,0){\line(8,0){8}}
\put(0,-4){\line(0,12){12}}
{\thicklines
\qbezier(3,-5)(2,-3.5)(1,-1)
\qbezier(-3,7)(-1.5,7)(.8,7)
}
\qbezier[40](1,-1)(1,-3)(1,7)
\put(8.3,-0.2){$g$}
\put(-3,9){$\hat S^\mathrm{ann}(\hb, \hh;g)$}
\put(-2,7.5){$\infty$}
\put(1,-1){\circle*{.5}}
\put(-1,-0.2){$0$}
\put(1,7){\circle{.5}}
\end{picture}
\vspace{1.2cm}
\begin{center}
{\qquad (1)  $\hh <\hh^\mathrm{ann}_c(\hb)$}
\end{center}
\end{minipage}
\begin{minipage}[hbtb]{5cm}
\centering
\setlength{\unitlength}{0.3cm}
\begin{picture}(8,8)(-7,-2.2)
\put(0,0){\line(8,0){8.5}}
\put(0,-4){\line(0,12){12}}
{\thicklines
\qbezier(3,-4.5)(1,-2.5)(0,0)
\qbezier(-3,7)(-1.5,7)(-.2,7)
}
\qbezier[40](4,0)(4,0)(0,0)
\qbezier[40](4,0)(4,0)(0,0)  
\put(8.7,-0.2){$g$}
\put(-3,9){$ \hat S^\mathrm{ann}(\hb, \hh;g)$}
\put(-2,7.5){$\infty$}
\put(0,0){\circle*{.5}}
\put(-1,-0.2){$0$}
\put(0,7){\circle{.5}}
\end{picture}
\vspace{1.3cm}
\begin{center}
\qquad (2)  $\hh = \hh^\mathrm{ann}_c(\hb)$
\end{center}
\end{minipage}
\begin{minipage}[hbt]{5cm}
\centering
\setlength{\unitlength}{0.3cm}
\begin{picture}(8,8)(-7,-2.2)
\put(0,0){\line(8,0){8}}
\put(0,-4){\line(0,12){12}}
{\thicklines
\qbezier(4,-5)(1,-3.5)(0,-1)
\qbezier(-3,7)(-1.5,7)(-.2,7)
}
\put(8.3,-0.2){$g$}
\put(-3,9){$\hat  S^\mathrm{ann}(\hb, \hh;g)$}
\put(-2,7.5){$\infty$}
\put(0,-1){\circle*{.5}}
\put(-1,-0.2){$0$}
\put(0,7){\circle{.5}}
\end{picture}
\vspace{1.2cm}
\begin{center}
\qquad (3)  $\hh >  \hh^\mathrm{ann}_c(\hb)$
\end{center}
\end{minipage}
\vspace{-.4cm}
\begin{center}
\caption{\small Qualitative picture of the map $g \mapsto \hat S^\mathrm{ann}(\hb,\hh;g)$ 
for $\hb,\hh\geq0$. 
Compare with Fig.~\ref{fig-varfe}.}
\label{fig-varfe1}
\end{center}
\vspace{-.6cm}
\end{figure}

The map $g \mapsto S^\mathrm{ann}(\hb,\hh,\bb;g)$ has the same qualitative picture 
as in Fig.~\ref{fig-varfe1}, with the following changes: the horizontal axis is 
located at $\th$ instead of zero, and $\hh_c^{\rm ann}(\hb)$ is replaced by 
$h_c^{\rm ann}(\hb,\bb,\th)$. 

Subtracting $\th$ from \eqref{Scopann} and \eqref{Sanpin}, we get from \eqref{ganvarexp} 
that the excess free energies $\hat g^{\rm ann}(\hb,\hh)$ and $\bar g^{\rm ann}(\bb,\th)$
take the form given in \eqref{copannfeg} and \eqref{pinannfeg}, respectively. The 
following lemma summarizes their relationship.

\begin{lemma}
\label{grela}
For every $\bb,\hh,\hb\geq 0$ and $\th\in\R$ (recall \eqref{lann12})
\begin{equation}
\label{anneferel}
\begin{split}
g^{\rm ann}(\hb,\hh,\bb,\th)\left\{\begin{array}{ll}
=\hat g^{\rm ann}(\hb,\hh), &\mbox{ if }\quad \th\geq \th_\ast(\hb, \hh,\bb),\\
> \hat g^{\rm ann}(\hb,\hh), &\mbox{ if } \quad \th< \th_\ast(\hb, \hh,\bb),\\
\leq\bar g^{\rm ann}(\bb,\th), &\mbox{ if } \quad \hh>\hat h_c^{\rm ann}(\hb),\\
=\bar g^{\rm ann}(\bb,\th), &\mbox{ if } \quad \hh=\hat h_c^{\rm ann}(\hb),\\
\geq\bar g^{\rm ann}(\bb,\th), &\mbox{ if } \quad \hh<\hat h_c^{\rm ann}(\hb).
\end{array}
\right.
\end{split}
\end{equation}
\end{lemma}

\begin{proof}
Note from \eqref{Sandef} and (\ref{scopadf}--\ref{Scopann}) that $S^{\rm ann}(\hb,\hh,
\bb;g)-\th$ is $\hat S^{\rm ann}(\hb,\hh;g)$ shifted by $\bar M(-\bb)-\th$.  We see 
from Fig.~\ref{fig-varfe1} that if $\th\geq \th_\ast(\hb,\hh,\bb)$, then the map 
$g\mapsto S^{\rm ann}(\hb,\hh,\bb;g)-\th$ changes sign at the same value of $g$ as 
the map $g\mapsto\hat S^{\rm ann}(\hb,\hh;g)$ does. Hence $g^{\rm ann}(\hb,\hh,\bb,\th)
=\hat g^{\rm ann}(\hb,\hh)$ whenever $\th\geq\th_\ast(\hb,\hh,\bb)$. On the other 
hand, if $\th<\th_\ast(\hb, \hh,\bb)$, then the map $g\mapsto\hat S^{\rm ann}(\hb,\hh;g)$
changes sign before the map $g\mapsto S^{\rm ann}(\hb,\hh,\bb;g)-\th$ does, i.e., 
$S^{\rm ann}(\hb,\hh,\bb;\hat g^{\rm ann}(\hb,\hh))-\th>0$, and hence $g^{\rm ann}
(\hb,\hh,\bb,\th)>\hat g^{\rm ann}(\hb,\hh)$.
 
The rest of the proof follows from a comparison of \eqref{scopadf} and \eqref{Sanpin}.
Note that, for $\hh>\hat h_c^{\rm ann}(\hb)$, we have $S^{\rm ann}(\hb,\hh,\bb;g)-\th
<\bar S^{\rm ann}(\bb;g)-\th$, which implies that $g^{\rm ann}(\hb,\hh,\bb,\th)\leq
\bar g^{\rm ann}(\bb,\th)$. For $\hh=\hat h_c^{\rm ann}(\hb)$, we have $S^{\rm ann}
(\hb,\hh,\bb;g)-\th = \bar S^{\rm ann}(\bb;g)-\th$,  which implies that $g^{\rm ann}
(\hb,\hh,\bb,\th)=\bar g^{\rm ann}(\bb,\th)$. Finally, for $\hh<\hat h_c^{\rm ann}(\hb)$ 
we  have $S^{\rm ann}(\hb,\hh,\bb;g)-\th > \bar S^{\rm ann}(\bb;g)-\th$, which implies
that $g^{\rm ann}(\hb,\hh,\bb,\th)\geq \bar g^{\rm ann}(\bb,\th)$.
\end{proof}

 
\subsection{Proof of Theorem~\ref{freeenegvar}}  
\label{S3.3}

\begin{proof}  
Throughout the proof $\hb>0$, $\bb\geq0$ and $\th\in\R$ are fixed.\\
(i) Use Theorems~\ref{varfloc}(i,iii).\\
(ii) Recall from \eqref{copadcc} and \eqref{gvarexp} that
\begin{equation}
\label{scc}
\begin{split}
h_c^{\rm que}(\hb,\bb,\th)
&=\inf\left\{\hh>0\colon\, g^{\rm que}(\hb,\hh,\bb,\th)=0\right\}\cr
&=\inf\left\{\hh>0\colon\, S^{\rm que}(\hb,\hh,\bb;0)-\th\leq0\right\}.
\end{split}
\end{equation}
Indeed, it follows from \eqref{gvarexp} that $g^{\rm que}(\hb,\hh,\bb,\th)=0$ is equivalent 
to saying that the map $g\mapsto S^{\rm que}(\hb,\hh,\bb;g)-\th$ changes sign at zero. 
This sign change can happen while $S^{\rm que}(\hb,\hh,\bb;0)-\th$ is either zero or 
negative (see Fig.~\ref{fig-varfe}(2--3)). The corresponding expression for $h_c^{\rm ann}
(\hb,\bb,\th)$ is obtained in a similar way.
\end{proof}
 
 
\section{Key lemma and proof of Corollary~\ref{freeeneggap}}
\label{Sec4}

The following lemma will be used in the proof of Corollary~\ref{freeeneggap}.  

\begin{lemma}
\label{auxlemgap}
Fix $\alpha\geq 1$, $\hb,\hh>0$ and $\bb\geq0$. Then, for $g>0$,  
\begin{equation}
\label{squesanninq}
S^{\rm que}(\hb,\hh,\bb;g)<S^{\rm ann}(\hb,\hh,\bb;g)
\left\{\begin{array}{ll}
\text{if } (\hb,\hh)\in \hat\cD^{\rm ann},\\
\text{if } (\hb,\hh)\in \hat\cL^{\rm ann} \text{ and } m_\rho<\infty,\\
\text{if } g\neq \hat g^{\rm ann}(\hb,\hh),(\hb,\hh)\in \hat\cL^{\rm ann}
\text{ and } m_\rho=\infty.
\end{array}
\right.
\end{equation}
\end{lemma}
Lemma~\ref{auxlemgap} is proved in Section~\ref{Sec4.2.2}. In Section~\ref{Sec4.2.1} we 
use Lemma~\ref{auxlemgap} to prove Corollary~\ref{freeeneggap}.

 
\subsection{Proof of Corollary~\ref{freeeneggap}}
\label{Sec4.2.1}
 
\begin{proof} 
(ii) Throughout the proof, $\alpha\geq 1$, $\hb,\hh>0$ and $\bb\geq0$. 
It follows from \eqref{anneferel} that $\cL^{\rm ann}_1\subset\cL^{\rm ann}$.

Note that, 
for $(\hb,\hh)\in\hat\cL^{\rm ann}$, the map $g\mapsto S^{\rm ann}(\hb,\hh,\bb;g)-\th$ 
changes sign at some $g\geq\hat g^{\rm ann}(\hb,\hh)>0$, i.e., $g^{\rm ann}
(\hb,\hh,\bb,\th)\geq\hat g^{\rm ann}(\hb,\hh)>0$ for all $\bb\geq0$ and $\th\in\R$ (see \eqref{anneferel}). Hence 
$\cL^{\rm ann}_1\subset\cL^{\rm ann}$.
 
Note from \eqref{scopadf} and \eqref{Scopann} that 
\begin{equation}
S^{\rm ann}(\hb,\hh,\bb;g)-\th = \hat S^{\rm ann}(\hb, \hh;g)+\bar M(-\bb)-\th.
\end{equation}
Furthermore, note from Fig.~\ref{fig-varfe1}(2--3) that, for $(\hb,\hh)\in\hat\cD^{\rm ann}$, 
the map $g\mapsto \hat S^{\rm ann}(\hb,\hh;g)$ changes sign at $g=0$ while $\hat 
S^{\rm ann}(\hb,\hh;0)$ is either negative or zero. In either case, we need 
\begin{equation}
\th<\bar M(-\bb)+\log \left(\tfrac12\left[1+\sum_{n\in\N}\rho(n)
e^{n[\hM(2\hb)-2\hb \hh]}\right]\right)=\th_\ast(\hb, \hh,\bb)
\end{equation}
to ensure that the map $g\mapsto S^{\rm ann}(\hb,\hh,\bb;g)-\th$ changes sign at a positive 
$g$-value. This concludes the proof that $\cL^{\rm ann}=\cL^{\rm ann}_1\cup \cL^{\rm ann}_2$.
 
\medskip\noindent
(i) As we saw in the proof of (ii), for the map $g\mapsto S^{\rm ann}(\hb,\hh,\bb;g)-\th$
to reach zero we need that $\th\leq \th_\ast(\hb, \hh,\bb)$. Thus, for this range of 
$\th$-values, we know that the map $g\mapsto S^{\rm ann}(\hb,\hh,\bb;g)-\th~$ changes 
sign when it is zero. The proof for $g^{\rm que}(\hb,\hh,\bb,\th)$ follows from 
Fig.~\ref{fig-varfe}.
 
\medskip\noindent
(iii) We first consider the cases: (a) $(\hb,\hh,\bb,\th)\in\cL^{\rm ann}_2$; (b) 
$(\hb,\hh,\bb,\th)\in\cL^{\rm ann}_1$ and $m_\rho<\infty$. In these cases we have that 
$\hh<h_c^\mathrm{ann}(\hb,\bb,\th)$ by (ii). It follows from (\ref{scopadf}--\ref{Scopann}) 
and Fig.~\ref{fig-varfe1} that the map $g\mapsto S^{\rm ann}(\hb,\hh,\bb;g)-\th$ changes 
sign at some $g>0$ while it is either zero or negative. In either case the finiteness of 
the map $g\mapsto S^{\rm que}(\hb,\hh,\bb;g)-\th$ on $(0,\infty)$ and \eqref{squesanninq} 
imply that $g\mapsto S^{\rm que}(\hb,\hh,\bb;g)-\th$ changes sign at a smaller value of 
$g$ than $g\mapsto S^{\rm ann}(\hb,\hh,\bb;g)-\th$ does. This concludes the proof for 
cases (a--b). 
 
We next consider the case: (c) $(\hb,\hh,\bb,\th)\in\cL^{\rm ann}$ with $(\hb,\hh)\in\hat
\cL^{\rm ann}$, $\th\neq\th_\ast(\hb,\hh,\bb)$ and $m_\rho=\infty$. We know from 
\eqref{squesanninq} that $S^{\rm que}(\hb,\hh,\bb;g)<S^{\rm ann}(\hb,\hh,\bb;g)$ for 
$g>0$ and $g\neq\hat g^{\rm ann}(\hb,\hh)$. If $\th>\th_\ast(\hb, \hh,\bb)$, then the 
map $g\mapsto S^{\rm ann}(\hb,\hh,\bb;g)-\th$ changes sign at $\hat g^{\rm ann}(\hb,\hh)$ 
while jumping from $<0$ to $\infty$. By the continuity of the map $g\mapsto S^{\rm que}
(\hb,\hh,\bb;g)$ on $(0,\infty)$, this implies that the map $g\mapsto S^{\rm que}(\hb,\hh,
\bb;g)-\th$ changes sign at a $g$-value smaller than $\hat g^{\rm ann}(\hb,\hh)$. Furthermore, if $\th<\th_\ast(\hb,\hh,\bb)$, then the map $g\mapsto S^{\rm ann}(\hb,\hh,\bb;g)-\th$ 
changes sign at a $g$-value larger than $\hat g^{\rm ann}(\hb,\hh)$, while it is zero. 
Since $S^{\rm que}(\hb,\hh,\bb;g)<S^{\rm ann}(\hb,\hh,\bb;g)$ for $g>\hat g^{\rm ann}
(\hb,\hh)$, we have that $g^{\rm que}(\hb,\hh,\bb,\th)<g^{\rm ann}(\hb,\hh,\bb,\th)$.

\medskip\noindent
(iv) The proof follows from Lemma \ref{grela}.
\end{proof}
 
 
\subsection{Proof of Lemma~\ref{auxlemgap}}
\label{Sec4.2.2}
 
\begin{proof}
The proof comes in five steps. Step 1 proves the strict inequality in \eqref{squesanninq}, 
using a claim about the finiteness of $I^{\rm ann}$ at some specific $Q$ in combination
with arguments from Birker~\cite{Bi08}. Steps 2-5 are used to prove the claim about the finiteness of $I^{\rm ann}$. Note that for $0<g<\hat g^\mathrm{ann}(\hb,\hh)$ the claim trivially follows from Theorems~\ref{varfloc}(iii) and \ref{varflocann}(ii), since 
$S^\mathrm{que}(\hb,\hh,\bb;g)<\infty$ and $S^\mathrm{ann}(\hb,\hh,\bb;g)=\infty$ 
for this range of $g$-values. Thus, what remains to be considered is the case 
$g\geq \hat g^\mathrm{ann}(\hb,\hh)$.

\medskip\noindent
{\bf 1.}
For $g\geq \hat g^\mathrm{ann}(\hb,\hh)$, note from \eqref{uniqmax} and the remark below 
it that there is a unique maximizer $Q_{\hb, \hh,\bb;g}=(q_{\hb,\hh,\bb;g})^{\otimes\N}$ 
for the  variational formula for $S^\mathrm{ann}(\hb,\hh,\bb;g)$ in \eqref{Sandef}, where 
\begin{equation}
\label{qbhbb}
\begin{split}
&q_{\hb, \hh,\bb;g}(d(\ho_1,\bo_1),\ldots,d(\ho_n,\bo_n))\cr
&=\dfrac{\frac12\;\rho(n)e^{-gn}(1+e^{-2\hb\big[n\hh
+\sum_{i=1}^n\ho_i\big]})e^{\bb\bo_1}}{ \cN(\hb, \hh,\bb;g)}
(\hm\otimes\bm)^{\otimes n}(d(\ho_1,\bo_1),\ldots,d(\ho_n,\bo_n)),
\end{split}
\end{equation}
where
\begin{equation}
\label{cnhat}
\cN(\hb, \hh,\bb;g)
=\tfrac12e^{\bar M(-\bb)}\left[\sum_{n\in\N}\rho(n)e^{-gn}
+\sum_{n\in\N}\rho(n) e^{-n(g-[\hat M(2\hb)-2\hb \hh])}\right]
=e^{\bar M(-\bb)} \hat\cN(\hb, \hh;g).
\end{equation}
Note further that $Q_{\hb, \hh,\bb;g}\notin\cR$. We claim that, for $g\geq\hat g^{\rm ann}
(\hb,\hh)$ and under the conditions in \eqref{squesanninq}, 
\begin{equation}
\label{iannfclaim}
H(Q_{\hb, \hh,\bb;g} \mid q_{\rho,\hm\otimes\bm}^{\otimes\N})
=h(q_{\hb, \hh,\bb;g} \mid q_{\rho,\hm\otimes\bm})<\infty.
\end{equation}
This will be proved in Step 2. Let $M<\infty$ be such that $h(q_{\hb, \hh,\bb;g} \mid 
q_{\rho,\hm\otimes\bm})<M$. Then the set 
\begin{equation}
\mathcal{A}_M = \left\{Q\in\cP^{\rm inv}(\widetilde{E}^{\rm \N})
\colon\, H(Q \mid q_{\rho,\hm\otimes\bm}^{\otimes\N})\leq M \right\}
\end{equation}
is compact in the weak topology, and contains $Q_{\hb,\hh,\bb;g}$ in its interior. It 
follows from Birkner \cite{Bi08}, Remark 8, that $\mathcal{A}_M\cap\cR$ is a closed 
subset of $\cP^{\rm inv}(\widetilde{E}^{\rm \N})$. This in turn implies that there exists
a $\delta>0$ such that $B_\delta(Q_{\hb, \hh,\bb;g})$ (the $\delta$-ball around $Q_{\hb,
\hh,\bb;g}$) satisfies $B_\delta(Q_{\hb, \hh,\bb;g})\cap \mathcal{A}_M\subset\cR^c$. 
Let 
\begin{equation}
\bar\delta = \sup\left\{0\leq \delta'\leq \delta\colon\, 
B_{\delta'}(Q_{\hb, \hh,\bb;g})\cap \mathcal{A}_M
= B_{\delta'}(Q_{\hb, \hh,\bb;g})\right\}.
\end{equation}
Then $\cR\subset B_{\bar\delta}(Q_{\hb, \hh,\bb;g})^c$. Therefore, for $g\geq\hat 
g^{\rm ann}(\hb,\hh)$ and under the conditions in \eqref{squesanninq}, we get that 
\begin{equation}
\begin{split}
S^{\rm que}(\hb,\hh,\bb;g)
&=\sup_{Q\in\cC^{\rm fin}\cap\cR}
\left[\bb \Phi(Q)+\Phi_{\hb, \hh}(Q)-g m_Q-I^{\rm ann}(Q)\right]\cr
&\leq \sup_{Q\in\cC^{\rm fin}\cap B_{\bar\delta}(Q_{\hb, \hh,\bb;g})^c}
\left[\bb \Phi(Q)+\Phi_{\hb, \hh}(Q)-g m_Q-I^{\rm ann}(Q)\right]\cr
&< \sup_{Q\in\cC^{\rm fin}}
\left[\bb \Phi(Q)+\Phi_{\hb, \hh}(Q)-g m_Q-I^{\rm ann}(Q)\right]\cr
&= S^{\rm ann}(\hb,\hh,\bb;g) = \log \cN(\hb,\hh,\bb;g).
\end{split}
\end{equation}
The strict inequality follows because no maximizing sequence in $\cC^{\rm fin}\cap 
B_{\bar\delta}(Q_{\hb,\hh,\bb;g})^c$ can have $Q_{\hb, \hh,\bb;g}$ as its limit
($Q_{\hb, \hh,\bb;g}$ being the unique maximizer of the variational problem in 
the second inequality).

\medskip\noindent
{\bf 2.} Let us now turn to the proof of the claim in \eqref{iannfclaim}. For $g\geq\hat
g^{\rm ann}(\hb,\hh)$, it follows from \eqref{qbhbb} and \eqref{cnhat} that 
\begin{equation}
\label{Hfinatqmax}
\begin{split}
h(q_{\hb, \hh\bb;g} \mid q_{\rho,\hm\otimes\bm})\leq I+II,
\end{split}
\end{equation}
where 
\begin{equation}
\label{anexp}
\begin{split}
I &= \bb\int_{\bE} \bo_1 e^{\bb \bo_1-\bar M(-\bb)}\bm(d\bo_1)
- \log\cN(\hb, \hh,\bb;g),\cr
II &= \frac{1}{\hat\cN(\hb, \hh;g)} \sum_{n\in\N} \rho(n)e^{-ng} A(n),\cr
A(n) &= \tfrac12\int_{\hE^n}\left[1+e^{-2\hb\sum_{i=1}^n(\ho_i+\hh)}\right]
\log\left(\tfrac12 \left[1+e^{-2\hb\sum_{i=1}^n(\ho_i+\hh)}\right]\right)
\hm^{\otimes n}(d\ho).
\end{split}
\end{equation}
The inequality in \eqref{Hfinatqmax} follows from \eqref{qbhbb} after replacing $e^{-gn}$ 
by 1. It is easy to see that $I<\infty$, because for $g\geq \hat g^{\rm ann}(\hb,\hh)$ 
we have that $\cN(\hb, \hh,\bb;g)<\infty$. Furthermore, since $\bm$ has a finite moment
generating function, it follows from the H\"older inequality that $\int_{\R} \bo_1 
e^{\bb \bo_1-\bar M(-\bb)}\bm(d\bo_1)<\infty$. We proceed to show that $II<\infty$. 

\medskip\noindent
{\bf 3.} 
We first estimate  $A(n)$. Note that
\begin{equation}
\begin{split}
 A(n) &= \tfrac12\int_{\hE^n}\left[1+e^{-2\hb\sum_{i=1}^n(\ho_i+\hh)}\right]
\log\left(\tfrac12 \left[1+e^{-2\hb\sum_{i=1}^n(\ho_i+\hh)}\right]\right)
\hm^{\otimes n}(d\ho)\cr
&=\tfrac12\int_{\hE^n}\left[1+e^{-2\hb\sum_{i=1}^n(\ho_i+\hh)}\right]
\log\left(\frac{e^{-2\hb\sum_{i=1}^n(\ho_i+\hh)}}{2}
\left[1+e^{2\hb\sum_{i=1}^n(\ho_i+\hh)}\right]\right)
\hm^{\otimes n}(d\ho)\cr
&= A_1(n)+A_2(n),
\end{split}
\end{equation}
where
\begin{equation}
\begin{split}
A_1(n) &= -\hb\sum_{k=1}^n\int_{\hE^n}(\ho_k+\hh)
\left[1+e^{-2\hb\sum_{i=1}^n(\ho_i+\hh)}\right]\hm^{\otimes n}(d\ho),\cr
A_2(n) &= \tfrac12\int_{\hE^n}\left[1+e^{-2\hb\sum_{i=1}^n(\ho_i+\hh)}\right]
\log\left( \tfrac12 \left[1+e^{2\hb\sum_{i=1}^n(\ho_i+\hh)}\right]\right)
\hm^{\otimes n}(d\ho).
\end{split}
\end{equation}
The finiteness of $II$ will follow once we show that 
\begin{equation}
\sum_{n\in\N}\rho(n)e^{-gn}[A_1(n)+A_2(n)]<\infty.
\end{equation} 

\medskip\noindent
{\bf 4.}
We start with the estimation of $A_2(n)$. Put $u_n(\ho)=-2\hb\sum_{i=1}^n(\ho_i+h)$ and, 
for $n\in\N$ and $m\in \N_0$, define
\begin{equation}
B_{m,n} = \left\{\ho\in\hE^n\colon -(m+1)< u_n(\ho)\leq -m\right\},
\qquad m_n=m_n(\hb,\hh)=\lceil 4\hb \hh n \rceil.
\end{equation}
Then note that 
\begin{equation}\label{A2n}
\begin{split}
A_2(n) &=\tfrac12\int_{\hE^n}\left[1+e^{u_n(\ho)}\right]
\log\left( \tfrac12 \left[1+e^{-u_n(\ho)}\right]\right)\hm^{\otimes n}(d\ho)\cr
&\leq \int_{\hE^n}\left[1\vee e^{u_n(\ho)}\right]
\log\left(1\vee e^{-u_n(\ho)}\right)\hm^{\otimes n}(d\ho)\cr
&=-\int_{u_n\leq0} u_n(\ho)\hm^{\otimes n}(d\ho)\cr
&=-\sum_{m\in\N_0} \int_{B_{m,n}} u_n(\ho)\hm^{\otimes n}(d\ho)\cr
&\leq \sum_{m=0}^{m_n} (m+1) + \sum_{m>m_n} (m+1)\P_{\ho}(B_{m,n})\cr
&\leq 4m_n^2 + \sum_{v\in\N} (m_n+v+1)\P_{\ho}(B_{m_n+v,n}).
\end{split}
\end{equation}
The second inequality uses that $-(m+1)< u_n\leq -m$ on $B_{m,n}$ and $\P_{\ho}(B_{m_n+v,n})
\leq1$. Estimate
\begin{equation}
\label{A2npre}
\begin{split}
\P_{\ho}(B_{m_n+v,n})
&=\P_{\ho}\left(\frac{m_n+v}{2\hb} \leq \sum_{k=1}^n(\ho_k+\hh)
< \frac{m_n+v+1}{2\hb}\right)\cr
&\leq \P_{\ho}\left(\sum_{k=1}^n\ho_k\geq \frac{m_n+v}{2\hb}-n\hh\right)\cr
&\leq \P_{\ho}\left(\sum_{k=1}^n\ho_k\geq \frac{4\hb n \hh+v}{2\hb}-n\hh\right)\cr
& =\P_{\ho}\left(\sum_{k=1}^n\ho_k\geq \frac{v}{2\hb}+n\hh\right)\cr
&\leq e^{-C\left(\frac{v}{2\hb}+n\right)}.
\end{split}
\end{equation}
The last inequality uses \cite{BodHoOp11}, Lemma D.1, where $C$ is a positive constant 
depending on $\hh$ only. Inserting \eqref{A2npre} into \eqref{A2n}, 
we get
\begin{equation}
A_2(n)\leq 4 m_n^2+(m_n+1)e^{-Cn}\frac{e^{-1/2\hb}}{1-e^{-1/2\hb}}
+e^{-Cn}\sum_{v\in\N}v\, e^{-v/2\hb}.
\end{equation}
Furthermore, using that $g>0$, we get
\begin{equation}
\begin{split}
\sum_{n\in\N}\rho(n) e^{-ng} A_2(n)
&\leq 4\sum_{n\in\N}\rho(n) e^{-ng} m_n^2
+ \frac{e^{-1/2\hb}}{1-e^{-1/2\hb}}\sum_{n\in\N}\rho(n) (m_n+1)e^{-n[g+C]}\cr
&\qquad + \sum_{n\in\N}\rho(n) e^{-n[g+C]} \sum_{v\in\N}v\, e^{-v/2\hb}<\infty.
\end{split}
\end{equation} 

\medskip\noindent
{\bf 5.} We proceed with the estimation of $A_1(n)$: 
\begin{equation}
\begin{split}
A_1(n) &= -\hb\sum_{k=1}^n\int_{\hE^n} (\ho_k+\hh)
\left[1+e^{-2\hb\sum_{i=1}^n(\ho_i+\hh)}\right]\hm^{\otimes n}(d\ho)\cr
&\leq -\hb\sum_{k=1}^n\int_{\hE^n}\ho_k
\left[1+e^{-2\hb\sum_{i=1}^n(\ho_i+\hh)}\right]\hm^{\otimes n}(d\ho)\cr
&= -n\hb\; e^{n[\hM(2\hb)-2\hb \hh]}\;\E_{\hm_{\hb}}(\ho_1), 
\end{split}
\end{equation}
where $\hm_{\hb}(d\ho_1)=e^{-2\hb\ho_1-M(2\hb)}\hm(d\ho_1)$. The right-hand side is 
non-negative because $\E_{\hm_{\hb}}(\ho_1)\leq0$, and so
\begin{equation}
\sum_{n\in\N}\rho(n) e^{-ng} A_1(n) \leq -\hb \E_{\hm_{\hb}}(\ho_1)
\sum_{n\in\N} n\rho(n)\,e^{-n(g-[M(2\hb)-2\hb \hh])}.
\end{equation}
This bound is finite if 
\begin{enumerate}
\item $g>\hat g^{\rm ann}(\hb,\hh)=\hM(2\hb)-2\hb \hh$;
\item $g=\hat g^{\rm ann}(\hb,\hh)$ and  $m_\rho<\infty$.
\end{enumerate}
This concludes the proof since, if $(\hb,\hh)\in\hat\cD^{\rm ann}$, then $\hat g^{\rm ann}
(\hb,\hh)=0$ and we only want the finiteness for $g>0$.
\end{proof}

For the pinning model, the associated unique maximizer $\bQ_{\bb;g}$ for the variational 
formula for $\bar S^{\rm ann}(\bb;g)$ satisfies $ H(\bQ_{\bb;g}|q^{\otimes\N}_{\rho,\bm})
<\infty$ for $g \geq 0$. However, this does not imply separation between $\bar S^{\rm que}
(\bb;0)$ and $\bar S^{\rm ann}(\bb;0)$, since we may have $\bQ_{\bb;0}\in\bar\cR$ for 
$m_\rho=\infty$. The separation occurs at $g=0$ as soon as $m_\rho<\infty$, since this will 
imply that $\bQ_{\bb;0}\notin\bar\cR$.


\section{Proof of Corollary \ref{hquehannvarfor}} 
\label{Sec5} 

To prove Corollary~\ref{hquehannvarfor} we need some further preparation, formulated as 
Lemmas~\ref{asyptann}--\ref{hizlimts} below. These lemmas, together with the proof of 
Corollary~\ref{hquehannvarfor},  are  given in Section~\ref{Sec5.1}. Section~\ref{Sec5.2} 
contains the proof of the first two lemmas, and Appendix~\ref{appc} the proof of the third
lemma.


\subsection{Key lemmas and proof of Corollary~\ref{hquehannvarfor}}
\label{Sec5.1}

\begin{lemma}
\label{asyptann}
For $\hb,\bb\geq0$, 
\begin{equation}
\label{Stildeann}
\begin{split}
S^{\rm ann}(\hb,\hh,\bb;0)=\left\{\begin{array}{ll}
\bar M(-\bb), &\mbox{ if }\quad \hh=\hat h_c^{\rm ann}(\hb),\\
\infty, &\mbox{ if }\quad \hh<\hat h_c^{\rm ann}(\hb),\\
\bar M(-\bb)-\log 2,  &\mbox{ if }\quad \hh=\infty.
\end{array}
\right.
\end{split}
\end{equation}
Furthermore, the map $\hh\mapsto S^{\rm ann}(\hb,\hh,\bb;0)$ is strictly convex and 
strictly decreasing on $[\hat h_c^{\rm ann}(\hb),\infty)$.
\end{lemma}


\begin{figure}[htbp]
\vspace{3.5cm}
\begin{center}
\setlength{\unitlength}{0.48cm}
\begin{picture}(10,4)(0,-4)
\put(0,0){\line(14,0){14}}
\put(0,-4.5){\line(0,5){10}}
{\thicklines
\qbezier(4,3.5)(7.5,-2.8)(14,-3.5)
\qbezier(3.8,5)(3,5)(0,5)
}
\qbezier[30](0,-4)(4,-4)(14,-4)
\qbezier[30](4,0)(4,2)(4,3.5)
\qbezier[30](0,3.5)(2,3.5)(4,3.5)
\put(14.5,-0.2){$\hh$}
\put(2.8,-1){$\hat h_c^{\rm ann}(\hb)$}
\put(-3,6.3){$S^\mathrm{ann}(\hb,\hh,\bb;0)$}
\put(-3,3.2){$\th^{\rm ann}_c(\bb)$}
\put(7,.5){$h^\mathrm{ann}_c(\hb,\bb,\th)$}
\put(-1,-.4){$\th$}
\put(2,5.3){$\infty$}
\put(4,5){\circle{.3}}
\put(4,3.5){\circle*{.3}}
\put(-6,-4.2){$\th^{\rm ann}_c(\bb)-\log 2$}
\end{picture}
\end{center}
\vspace{-0.2cm}
\caption{\small Qualitative picture of  $\hh \mapsto S^\mathrm{ann}(\hb,\hh,\bb;0)$ 
for $\hb>0$ and  $\bb\geq0$. } 
\label{fig-copadannvarhs}
\vspace{0cm}
\end{figure}

\begin{lemma}
\label{queasympt}
For every $\hb>0$ and $\bb\geq0$ (see Fig.~{\rm \ref{fig-copadannvarhs}}), 
\begin{equation}
S^{\rm que}(\hb,\hh,\bb;0)=S_*^\mathrm{que}(\hb,\hh,\bb) 
\left\{\begin{array}{ll}
=\infty, &\mbox{ for }\quad \hh<\hat h_c^{\rm ann}(\hb/\alpha)\\
>0,      &\mbox{ for }\quad \hh=\hat h_c^{\rm ann}(\hb/\alpha)\\
<\infty, &\mbox{ for }\quad \hh>\hat h_c^{\rm ann}(\hb/\alpha).
\end{array}
\right.
\end{equation}
\end{lemma}

\begin{figure}[htbp]
\vspace{3.5cm}
\begin{minipage}[hbt]{8.1cm}
\centering
\setlength{\unitlength}{0.48cm}
\begin{picture}(10,4)(0,-4)
\put(2,0){\line(10,0){10}}
\put(2,-4.5){\line(0,9){9}}
{\thicklines
\qbezier(6,2.5)(7.5,-2.8)(11.5,-3.5)
\qbezier(2,3.5)(4.5,3.5)(5.8,3.5)
}
\qbezier[30](2,-4)(6,-4)(11.5,-4)
\qbezier[30](2,2.5)(4,2.5)(6,2.5)
\qbezier[30](6,0)(6,1.5)(6,2.5)
\put(12.5,-0.2){$\hh$}
\put(-.9,5.5){$S^\mathrm{que}(\hb,\hh,\bb;0)$}
\put(-1.8,2.4){\small $s_*(\hb,\bb,\a)$}
\put(7.5,.3){$h^\mathrm{que}_c(\hb,\bb,\th)$}
\put(4,-1){$\hat h_c^\mathrm{ann}(\frac{\hb}{\a})$}
\put(1,-.4){$\th$}
\put(4,3.64){$\infty$}
\put(6,2.5){\circle*{.4}}
\put(6,3.5){\circle{.4}}
\put(-3.8,-4){$\th^{\rm que}_c(\bb)-\log 2$}
\end{picture}
\begin{center}
(a) 
\end{center}
\end{minipage}
\begin{minipage}[hbt]{8.1cm}
\centering
\setlength{\unitlength}{0.48cm}
\begin{picture}(10,4)(0,-4)
\put(2,0){\line(10,0){10}}
\put(2,-4.5){\line(0,9){9}}
{\thicklines
\qbezier(6,3.5)(7.5,-2.8)(11.5,-3.5)
\qbezier(2,3.5)(4.5,3.5)(6,3.5)
}
\qbezier[30](2,-4)(6,-4)(11.5,-4)
\qbezier[30](6,0)(6,1.5)(6,3.5)
\put(12.5,-0.2){$\hh$}
\put(-.9,5.5){$S^\mathrm{que}(\hb,\hh,\bb;0)$}
\put(-1.8,3.4){\small $s_*(\hb,\bb,\a)$}
\put(7.5,.3){$h^\mathrm{que}_c(\hb,\bb,\th)$}
\put(4,-1){$\hat h_c^\mathrm{ann}(\frac{\hb}{\a})$}
\put(1,-.4){$\th$}
\put(4,3.64){$\infty$}
\put(-3.8,-4){$\th^{\rm que}_c(\bb)-\log 2$}
\end{picture}
\begin{center}
(b) 
\end{center}
\end{minipage}
\vspace{-0.2cm}
\caption{\small Qualitative picture of  $\hh \mapsto S^\mathrm{que}(\hb,\bb,\hh;0)$ for 
$\hb>0$ and $\bb\geq0$: (a) $s_*(\hb,\bb,\a)<\infty$; (b) $s_*(\hb,\bb,\a)=\infty$.} 
\label{fig-copadvarhs}
\vspace{-.20cm}
\end{figure}

\begin{lemma}
\label{hizlimts}
For every $\hb>0$ and $\bb\geq0$ (see Fig.~{\rm \ref{fig-copadvarhs}}),
\begin{equation}\label{hzilim1}
\begin{split}
S^{\rm que}(\hb,\infty-,\bb;0)
& = \lim_{\hh\to\infty} S^{\rm que}(\hb,\hh,\bb;0)
=S^{\rm que}(\hb,\infty,\bb;0) = \th_c^{\rm que}(\bb)-\log 2.
\end{split}
\end{equation}
\end{lemma}

We now give the proof of Corollary~\ref{hquehannvarfor}.

\begin{proof}
Throughout the proof $\hb>0$, $\bb\geq0$ and $\th\in\R$ are fixed. Note from \eqref{phidef} 
that the map $\hh\mapsto\log\phi_{\hb,\hh}(\ho)$ is strictly decreasing and convex for 
all $\ho\in\widetilde{\hE}$. It therefore follows from  \eqref{Sdef} and \eqref{Sandef} 
that the maps $\hh \mapsto S^\mathrm{que}(\hb,\hh,\bb;0)$ and  $\hh \mapsto S^\mathrm{ann}
(\hb,\hh,\bb;0)$ are strictly decreasing when finite (because $\tau(\o)\geq 1$) and 
convex (because sums and suprema of convex functions are convex). 

Recall from \eqref{copadacc} and \eqref{gvarexp} that
\begin{equation}
\label{sancc}
\begin{split}
h_c^{\rm ann}(\hb,\bb,\th)
&=\inf\left\{\hh\geq0\colon\,g^{\rm ann}(\hb,\hh,\bb,\th)=0\right\}\cr
&=\inf\left\{\hh\geq0\colon\,S^{\rm ann}(\hb,\hh,\bb;0)-\th\leq0\right\}.
\end{split}
\end{equation}
Indeed, it follows from \eqref{gvarexp} that $g^{\rm ann}(\hb,\hh,\bb,\th)=0$ is equivalent 
to saying that the map $g\mapsto S^{\rm ann}(\hb,\hh,\bb;g)-\th$ changes sign at zero. 
This change of sign can happen while  $S^{\rm ann}(\hb,\hh,\bb;g)-\th$ is either zero or 
negative (see e.g.\ Fig.~\ref{fig-varfe}(2--3)).

For $\th\geq\th^{\rm ann}_c(\bb)$, it follows from Lemma~\ref{asyptann} and  
Fig.~\ref{fig-copadannvarhs} that $\hh=\hat h^{\rm ann}_c(\hb)$ is the smallest 
value of $\hh$ at which $S^{\rm ann}(\hb,\hh,\bb;0)-\th\leq 0$ and hence $h^{\rm ann}_c(\hb,
\bb,\th)=\hat h^{\rm ann}_c(\hb)$. Furthermore, note from Fig.~\ref{fig-copadannvarhs} 
that the map $\hh\mapsto S^{\rm ann}(\hb,\hh,\bb;0)$ is strictly decreasing and convex 
on $[\hat h^{\rm ann}_c(\hb),\infty)$ and has the interval $\left(\th^{\rm ann}_c(\bb)
-\log 2,\th^{\rm ann}_c(\bb)\right]$ as its range. In particular, $S^{\rm ann}(\hb,\hat
h^{\rm ann}_c(\hb),\bb;0)=\th^{\rm ann}_c(\bb)$ and $S^{\rm ann}(\hb,\infty,\bb;0)
=\th^{\rm ann}_c(\bb)-\log 2$. Therefore, for $\th\in \left(\th^{\rm ann}_c(\bb)-\log 2,
\th^{\rm ann}_c(\bb)\right]$, the map $\hh\mapsto S^{\rm ann}(\hb,\hh,\bb;0)-\th$ changes 
sign at the unique value of $\hh$ at which $S^{\rm ann}(\hb,\hh,\bb;0)=\th$.

For $\th\leq \th^{\rm ann}_c(\bb)-\log 2$, it follows from Fig.~\ref{fig-copadannvarhs} 
that $S^{\rm ann}(\hb,\hh,\bb;0)-\th>0$ for all $\hh\in[0,\infty)$. It therefore follows 
from \eqref{sancc} that
\begin{equation}
\begin{split}
h_c^{\rm ann}(\hb,\bb,\th)
&=\inf\left\{\hh\geq0\colon\,S^{\rm ann}(\hb,\hh,\bb;0)-\th\leq0\right\}=\infty.
\end{split}
\end{equation}

The proof for $h_c^{\rm que}(\hb,\bb,\th)$ follows from that of $h_c^{\rm ann}(\hb,\bb,\th)$ 
after replacing $S^\mathrm{ann}(\hb,\hh,\bb;0)$, $\th_c^{\rm ann}(\bb)-\log 2$ and 
$\th_c^{\rm ann}(\bb)$ by $S^\mathrm{que}(\hb,\hh,\bb;0)$, $\th_c^{\rm que}(\bb)-\log 2$ 
and  $s^{\ast}(\hb,\bb,\alpha)$, respectively. 

Since both the quenched and the annealed free energies are convex  functions of $\bar h$ 
(by H\"older inequality), $\bar h \mapsto h_\ast^{\rm que}(\hb,\bb,\th)$ and  $\bar h \mapsto 
h_\ast^{\rm ann}(\hb,\bb,\th)$ are also convex and strictly decreasing.

\end{proof}


\subsection{Proof of Lemmas~\ref{asyptann}--\ref{queasympt}}
\label{Sec5.2}

{\bf Proof of Lemma~\ref{asyptann}:}\\
\begin{proof}
Note from \eqref{scopadf}  that 
\begin{equation}
S^{\rm ann}(\hb,\hh,\bb;0)
=\bar M(-\bb)+\log\left(\tfrac12\left[1+\sum_{n\in\N}\rho(n)
e^{n[\hM(2\hb)-2\hb \hh]}\right]\right),
\end{equation}
which implies the claim.
\end{proof}

\noindent
{\bf Proof of Lemma~\ref{queasympt}:}\\
\begin{proof}
Throughout the proof $\hb,\hh>0$ and $\bb\geq0$ are fixed. The proof uses arguments 
from \cite{BodHoOp11}, Theorem 3.3 and Section 6.  Note from \eqref{copadSlim}, 
\eqref{seqsbar} and Lemma~\ref{varlem} that  
\begin{equation}
\label{copadSlimref}
S^\mathrm{que}(\hb,\hh,\bb;g) 
= \th + \limsup_{N\to\infty} \frac{1}{N} \log F_N^{\hb, \hh,\bb,\th,\o}(g)
= \log \mathcal{N}(g) + \limsup_{N\to\infty} \frac{1}{N}\log  S^{\o}_N(g),
\end{equation}
where
\begin{equation} 
S_N^\o (g) =E^*_g\left(\exp\left [N\left(\Phi_{\hb, \hh}(R_N^{\o})+\bb\,
\Phi(R_N^{\o})\right)\right]\right).
\end{equation}
It follows from the fractional-moment argument in \cite{BodHoOp11}, Eq.\ (6.4), that
\begin{equation}
\label{braceterm}
\E([S_N^\o (g)]^t)\leq \left(2^{1-t}\sum_{n\in\N}\rho_g(n)^{t}
e^{\tM(-\bb t)}\right)^{N}<\infty, \quad g\geq 0,
\end{equation}
where $t\in[0,1]$ is chosen such that $\hM(2\hb t)-2\hb \hh t\leq 0$. Abbreviate the 
term inside the brackets of \eqref{braceterm} by $K_t$ and note that 
\begin{equation}
\begin{split}
\P\left(\frac{t}{N}\log S_N^\o(g)\geq\log K_t+\epsilon\right)
&=\P\left( [S_N^\o(g)]^t\geq K_t^N e^{N\epsilon}\right)\cr
&\leq \E([S_N^\o(g)]^t)K_t^{-N} e^{-N\epsilon}\cr
&\leq e^{-N\epsilon}, \qquad \epsilon>0.
\end{split}
\end{equation}
Therefore, for $g>0$, this estimate together with the Borel-Cantelli lemma shows that 
$\o$-a.s.\ (recall \eqref{rhoz})
\begin{equation}
\begin{split}
S^\mathrm{que}(\hb,\hh,\bb;g)
&= \log\cN(g)+\limsup_{N\rightarrow\infty}\frac1N\log S_N^\o(g)
\leq \frac1t \log K_t+\log\cN(g)\cr 
&= \frac{1-t}{t}\log 2 + \frac1t \log\left(\sum_{n\in\N}
\rho_g(n)^{t}\right)+\frac1t \tM(-\bb t)+\log\cN(g)<\infty.
\end{split}
\end{equation}
This estimate also holds for $g=0$ when $\sum_{n\in\N}\rho(n)^{t}<\infty$. This is the 
case for any pair $t\in(1/\a,1]$ and $\hh>\hat h_c^{\rm ann}(\hb/\a)$ satisfying 
$\hM(2\hb t)-2\hb \hh t\leq 0$ (recall \eqref{rhocond}). Therefore we conclude that  
$S^{\rm que}(\hb,\hh,\bb;0)<\infty$ whenever $\hh>\hh_c^{\rm ann}(\hb/\a)$. 

To prove that $S^{\rm que}(\hb,\hh,\bb;0)=\infty$ for $\hh<\hat h_c^{\rm ann}(\hb/\alpha)$, 
we  replace $q_{\bb}^L$ in \cite{BodHoOp11}, Eq.\ (6.8), by
\begin{equation}
q_{\hb}^L(d(\ho_1,\bo_1),\ldots,d(\ho_n,\bo_n))
= \delta_{n,L} \left[\hm_{\hb/\a}(\ho_1)\times\cdots\times\hm_{\hb/\a}(d\ho_n)\right]
\times\left[\bm(d\bo_1)\times\cdots\times\bm(d\bo_n)\right],
\end{equation}
where
\begin{equation}
\hm_{\hb/\a}(d\ho_1)=e^{-(2\hb/\alpha)\ho_1-\hM(2\hb/\alpha)}\hm(d\ho_1).
\end{equation} 
With this choice the rest of the argument in \cite{BodHoOp11}, Section 6.2, goes through 
easily. 

Finally, to prove that $S^{\rm que}(\hb,\hh,\bb;0)>0$ at $\hh=\hat h_c^{\rm ann}(\hb/\alpha)$ 
we proceed as follows. Adding, respectively, $\bb \Phi(Q)$ and $\bb\sum_{n\in\N}\int_{\bE^n}
\bo_1 q(d(\ho_1,\bo_1),\ldots,d(\ho_n,\bo_n))$ to the functionals being optimized in 
\cite{BodHoOp11}, Eqs.\ (6.19--6.20), we get the following analogue of \cite{BodHoOp11},
Eq.\ (6.21),
\begin{equation}
\begin{split}
&q_{\hb, \hh,\bb}(d(\ho_1,\bo_1),\ldots,d(\ho_n,\bo_n))\cr
&=\frac{1}{\hat\cN(\hb, \hh)\,e^{\tM(-\bb/\a)}}
\left[\phi_{\hb, \hh}((\ho_1,\ldots\ho_n))\,e^{\bb \bo_1}\right]^{1/\a}
q_{\rho, \hm\otimes\bm}(d(\ho_1,\bo_1),\ldots,d(\ho_n,\bo_n)),
\end{split}
\end{equation}
where 
\begin{equation}
\hat\cN(\hb, \hh)=\sum_{n\in\N}\rho(n)\int_{\hE^n}\hm(d\ho_1)\times\cdots\times\hm(\ho_n)
\left\{\tfrac12\left(1+e^{-2\hb\sum_{k=1}^n(\ho_k+\hh)}\right)\right\}^{1/\a}.
\end{equation}
Note from \cite{BodHoOp11}, Eqs.\ (6.23--6.29), that $1<\hat\cN(\hb,\hh_c^{\rm ann}(\hb/\a))
\leq 2^{1-(1/\a)}$. Therefore it follows from \cite{BodHoOp11}, Steps 1 and 2 in Section 6.3, that 
\begin{equation}
S^{\rm que}(\hb,\hh_c^{\rm ann}(\hb/\a),\bb;0)= S_*^{\rm que}(\hb,\hh_c^{\rm ann}
(\hb/\a),\bb)\geq \a\log \hat\cN(\hb,\hh_c^{\rm ann}(\hb/\a))+\a \tM(-\bb/\a)>0.
\end{equation}
\end{proof}


\section{Proofs of Corollaries~\ref{hcubstrict}--\ref{locpathprop}}
\label{Sec6}


\subsection{Proof of  Corollary \ref{hcubstrict}}
\label{Sec6.1}
\begin{proof}
Throughout the proof, $\alpha>1$, $\hb>0$, $\bb\geq0$ and $\th\in\R$ are fixed.
The proof for $\bar h^{\rm que}_c(\bb)-\log 2<\th\leq\bar h^{\rm ann}_c(\bb)-
\log 2$ is trivial, since $h_c^{\rm ann}(\hb,\bb,\th)=\infty$ and $h_c^{\rm que}
(\hb,\bb,\th)<\infty$. The rest of the proof will follow once we show that 
$h_c^{\rm que}(\hb,\bb,\th)<h_c^{\rm ann}(\hb,\bb,\th)$ for $\bar h^{\rm ann }_c(\bb)
-\log 2<\th\leq\bar h^{\rm ann}_c(\bb)$. This is because, for $\th\geq\bar h^{\rm ann}_c(\bb)$, 
$h_c^{\rm ann}(\hb,\bb,\th)=\hat h^{\rm ann}_c(\hb)$ and $\th\mapsto h^{\rm que}
(\hb,\bb,\th)$ is non-increasing. Furthermore, the map $\th\mapsto h^{\rm ann}
(\hb,\bb,\th)$ is also non-increasing, and so $h_c^{\rm ann}(\hb,\bb,\th)\geq
\hat h^{\rm ann}_c(\hb)$.

For $\bar h^{\rm ann }_c(\bb)-\log 2<\th\leq\bar h^{\rm ann}_c(\bb)$, it follows 
from Corollary~\ref{hquehannvarfor} that $h_c^{\rm ann}(\hb,\bb,\th)$ is the unique
$\hh$-value that solves the equation $S^{\rm ann}(\hb,\hh,\bb;0)=\th$. Note that 
for $\hh\geq \hat h_c^{\rm ann}(\hb)=\hM(2\hb)/2\hb$, which is the range of 
$\hh$-values attainable by $h_c^{\rm ann}(\hb,\bb,\th)$, the measure 
$q_{\hb,\hh,\bb;0}$ (recall \eqref{uniqmax}) is well-defined and is the unique 
minimizer of the last variational formula in \eqref{varredann}, for $g=0$. 
Hence, for $\bar h^{\rm ann}_c(\bb)-\log 2<\th\leq\bar h^{\rm ann}_c(\bb)$, 
$h_c^{\rm ann}(\hb,\bb,\th)$ is the unique $\hh$-value that solves the equation
\begin{equation}
S^{\rm ann}(\hb,\hh,\bb;0)-\th=\bar M(-\bb)+\log\hat\cN(\hb, \hh;0)-\th=0.
\end{equation}
Again, it follows from \eqref{truncapproxcont} that, for any $Q\in\cP^{\rm inv}
(\widetilde{E}^\N)$,
\begin{equation}
\begin{split}
I^{\rm que}(Q) 
&=\sup_{{\rm tr}\in\N}I^{\rm que}([Q]_{{\rm tr}})=\sup_{{\rm tr}\in\N}
\left[H\left([Q]_{{\rm tr}}|q_{\rho,\hm\otimes\bm}^{\otimes\N}\right)
+(\alpha-1)m_{[Q]_{{\rm tr}}}
H\left(\Psi_{[Q]_{{\rm tr}}}|(\hm\otimes\bm)^{\otimes\N}\right)\right]\cr
&\geq H\left(Q \mid q_{\rho,\hm\otimes\bm}^{\otimes\N}\right)
+(\alpha-1)m_{[Q]_{{\rm tr}}}
H\left(\Psi_{[Q]_{{\rm tr}}}|(\hm\otimes\bm)^{\otimes\N}\right), 
\quad {\rm tr}\in\N.
\end{split}
\end{equation}
Furthermore, it follows from \eqref{spentrdef} and the remark below it that
\begin{equation}
\label{red1}
H(Q \mid q_{\rho,\hm\otimes\bm}^{\otimes\N}) 
\geq h(\pi_1Q \mid q_{\rho,\hm\otimes\bm}),
\qquad H(\Psi_{[Q]_{{\rm tr}}} \mid (\hm\otimes\bm)^{\otimes\N}) 
\geq h(\widetilde\pi_{1} 
\Psi_{[Q]_{{\rm tr}}} \mid \hm\otimes\bm),
\end{equation}
where $\widetilde\pi_1$ is the projection onto the first \emph{letter} and ${\rm tr}\in\N$. 
Moreover, it follows from \eqref{PsiQdef} that
\begin{equation}
\label{red2}
\widetilde\pi_1\Psi_Q = \widetilde\pi_1\Psi_{(\pi_1 Q)^{\otimes\N}}.
\end{equation} 
Since $m_Q = m_{(\pi_1 Q)^{\otimes\N}}= m_{\pi_1 Q}$, (\ref{red1}--\ref{red2}) combine 
with \eqref{Sdefalt} to give
\begin{equation}
\label{upbound}
\begin{split}
&S_*^\mathrm{que}(\hb,\hh,\bb)\cr
&\leq \sup_{ {q\in\cP(\widetilde{E})} \atop h(q \mid q_{\rho,\hm\otimes\bm})<\infty;
\,{m_q<\infty} } 
\left[\int_{\widetilde{E}} q(d\o)[\bb \bo_1+\log\phi_{\hb,\hh}(\ho)] 
- h(q \mid q_{\rho,\hm\otimes\bm})\right.\\
&\qquad\qquad\qquad \left. - (\alpha-1) m_{[q]_{{\rm tr}}} 
h(\widetilde\pi_1\Psi_{[q]_{{\rm tr}}} \mid \hm\otimes\bm) \right],
\end{split}
\end{equation}
where
\begin{equation}
\phi_{\hb, \hh}(\ho) 
= \tfrac12 \left(1 + e^{-2\hb \hh m-2\hb [\ho_1+\dots+\ho_m]}\right)
\end{equation}
and
\begin{equation}
(\widetilde\pi_1\Psi_{[q]_{{\rm tr}}})(d(\ho_1,\bo_1)) 
= \frac{1}{m_{[q]_{{\rm tr}}}} \sum_{m\in\N} [r]_{{\rm tr}}(m) 
\sum_{k=1}^m q_m(E^{k-1},d(\ho_1,\bo_1),E^{m-k}) 
\end{equation}
with the notation
\begin{equation}
q(d\o) = r(m)q_m(d(\ho_1,\bo_1),\dots,d(\ho_m,\bo_m)), 
\qquad \o=((\ho_1,\bo_1),\dots,(\ho_m,\bo_m)),
\end{equation}
and 
\begin{equation}
[r]_{{\rm tr}}(m)=\left\{\begin{array}{ll}
r(m) &\mbox{ if } 1\leq m<{\rm tr}\\
\sum_{n={\rm tr}}^\infty r(n) & \mbox{ if } m={\rm tr}\\
0 & \mbox{otherwise.}
\end{array}
\right.
\end{equation}
Therefore, for $\hh\geq\hat h^{\rm ann}_c(\hb)$, after combining the first two terms in the 
supremum in \eqref{upbound}, as in \eqref{varredann}, we obtain
\begin{equation}
\begin{split}
S_*^\mathrm{que}(\hb,\hh,\bb) 
&\leq \bar M(-\bb)+ \log\hat\cN(\hb, \hh;0) \cr
&\qquad - \inf_{ {q\in\cP(\widetilde{E})} \atop h(q\mid q_{\rho,\hm\otimes\bm})<\infty;
\,{m_q<\infty} } \left[h(q \mid q_{\hb, \hh,\bb})
+ (\alpha-1) m_{[q]_{{\rm tr}}} 
h(\widetilde\pi_1\Psi_{[q]_{{\rm tr}}} \mid \hm\otimes \bm) \right].
\end{split}
\end{equation}
Hence, for $\hh\geq\hat h^{\rm ann}_c(\hb)$, it follows from \eqref{varredann} that 
\begin{equation}
\begin{split}
&S_*^\mathrm{que}(\hb,\hh,\bb) 
\leq S^\mathrm{ann}(\hb,\hh,\bb;0)\cr 
&\qquad - \inf_{ {q\in\cP(\widetilde{E})} \atop 
h(q \mid q_{\rho,\hm\otimes\bm})<\infty;~ {m_q<\infty} }
\left[h(q \mid q_{\hb, \hh,\bb})
+ (\alpha-1) m_{[q]_{{\rm tr}}} 
h(\widetilde\pi_1\Psi_{[q]_{{\rm tr}}} \mid\hm\otimes \bm) \right].
\end{split}
\end{equation}
The first term in the variational formula achieves its minimal value zero at 
$q=q_{\hb,\hh,\bb}$ (or along a minimizing sequence converging to $q_{\hb, \hh,\bb}$). 
However, via some simple computations we obtain 
\begin{equation}
\begin{split}
\widetilde\pi_1\Psi_{[q_{\hb, \hh, \bb}]_{{\rm tr}}}(d\ho_1,d\bo_1)
&= \dfrac{C_{{\rm tr}}(\ho_1,\hb, \hh)}
{A_{{\rm tr}}(\hb,\hh)}\hm(d\ho_1)\bm_{\bb}(d\bo_1)
+ \dfrac{B_{{\rm tr}}(\ho_1,\hb, \hh)}
{A_{{\rm tr}}(\hb,\hh)}\hm(d\ho_1)\bm(d\bo_1),
\end{split}
\end{equation}
where
\begin{equation}
\begin{split}
A_{{\rm tr}}(\hb,\hh)
&= \tfrac12\left(\sum_{n=1}^{{\rm tr-1}}n[1+e^{n[\hM(2\hb)-2\hb \hh]}]\rho(n)
+ {\rm tr}\sum_{n={\rm tr}}^{\infty}[1+e^{n[\hM(2\hb)-2\hb \hh]}]\rho(n)\right),\cr
B_{{\rm tr}}(\ho_1,\hb, \hh)&=\sum_{m=1}^{{\rm tr}-1}(m-1)\rho(m)
\left[1+e^{(m-1)[\hM(2\hb)-2\hb \hh]-2\hb(\ho_1+\hh)}\right]\cr
&\qquad +({\rm tr}-1)\dfrac{1+e^{({\rm tr}-1)[\hM(2\hb)-2\hb \hh]
-2\hb(\ho_1+\hh)}}{1+e^{{\rm tr}[\hM(2\hb)-2\hb \hh]}}
\sum_{m={\rm tr}}^{\infty}\rho(m)
\left[1+e^{m[\hM(2\hb)-2\hb \hh]}\right],\cr
C_{{\rm tr}}(\ho_1,\hb, \hh)
&=\sum_{m=1}^{{\rm tr}-1}\rho(m)\left[1+e^{(m-1)[\hM(2\hb)-2\hb \hh]
-2\hb(\ho_1+\hh)}\right]\cr
&\qquad +\dfrac{1+e^{({\rm tr}-1)[\hM(2\hb)-2\hb \hh]
-2\hb(\ho_1+\hh)}}{1+e^{{\rm tr}[\hM(2\hb)-2\hb \hh]}}
\sum_{m={\rm tr}}^{\infty}\rho(m)\left[1+e^{m[\hM(2\hb)-2\hb \hh]}\right],
\end{split}
\end{equation}
and $\bm_{\bb}(d\bo_1)=e^{\bb \bo_1-\bar M(-\bb)}\bm(d\bo_1)$. Here we use that 
\begin{equation}
q_{\hb, \hh, \bb}(d(\ho_1,\bo_1),\ldots,d(\ho_m,\bo_m))
=r(m)q_{m}(d(\ho_1,\bo_1),\ldots,d(\ho_m,\bo_m))
\end{equation} 
with 
\begin{equation}
\begin{split}
r(m) &=\dfrac{1+e^{m[\hM(2\hb)-2\hb \hh]}}{2\hat\cN(\hb,\hh;0)}\rho(m),\cr
q_m(d(\ho_1,\bo_1),\ldots,d(\ho_m,\bo_m))
&=\dfrac{\left(1+e^{-2\hb\sum_{i=1}^m(\ho_i+\hh)}\right)e^{\bb \bo_1}}
{e^{\bar M(-\bb)}\left(1+e^{m[\hM(2\hb)-2\hb \hh]}\right)}
\prod_{i=1}^m\hm(d\ho_i)\bm(d\bo_i).
\end{split}
\end{equation}
Note that $\widetilde\pi_1\Psi_{[q_{\hb,\hh_c^{\rm ann}(\hb),0}]_{{\rm tr}}}
=(\tfrac12[\hm+\hm_{\hb}])\otimes\bm$, where $\hm_{\hb}(d\ho_1)=e^{-2\hb\ho_1-\hM(2\hb)}
\hm(d\ho_1)$. Thus $\widetilde\pi_1\Psi_{[q_{\hb,\hh,\bb}]_{{\rm tr}}}\neq\hm\otimes \bm$ 
for $\hh\geq\hat h_c^{\rm ann}(\hb)$, and so we have
\begin{equation}
\begin{split}
&S_*^\mathrm{que}(\hb,h_c^\mathrm{ann}(\hb,\bb,\th),\bb)-\th
\leq S^\mathrm{ann}(\hb,h_c^\mathrm{ann}(\hb,\bb,\th),\bb;0)-\th\cr
&-\inf_{ {q\in\cP(\widetilde{E})} \atop h(q\mid q_{\rho,\hm\otimes\bm})<\infty;
\,{m_q<\infty}}\left[h(q \mid q_{\hb,h_c^\mathrm{ann}(\hb,\bb,\th),\bb})
+ (\alpha-1) m_{[q]_{{\rm tr}}} h(\widetilde\pi_1
\Psi_{[q]_{{\rm tr}} }\mid\hm\otimes \bm) \right]\cr
&=\qquad - \inf_{ {q\in\cP(\widetilde{E})} \atop h(q\mid q_{\rho,\hm\otimes\bm})<\infty;
\,{m_q<\infty}}\left[h(q \mid q_{\hb,h_c^\mathrm{ann}(\hb,\bb,\th),\bb})
+ (\alpha-1) m_{[q]_{{\rm tr}}} h(\widetilde\pi_1
\Psi_{[q]_{{\rm tr}}} \mid\hm\otimes \bm) \right]<0.
\end{split}
\end{equation} 
Since $S_*^\mathrm{que}(\hb,h_c^\mathrm{que}(\hb,\bb,\th),\bb)-\th = 0$ and since
$\hh\mapsto S_*^\mathrm{que}(\hb,\hh,\bb)$ is strictly decreasing on $(\hat h_c^\mathrm{ann}
(\hb/\alpha),\infty)$, it follows that $h_c^\mathrm{que}(\hb,\bb,\th)<h_c^\mathrm{ann}
(\hb,\bb,\th)$ for $\th^{\rm ann}_c(\bb)-\log 2<\th\leq\th^{\rm ann}_c(\bb)$.
\end{proof}


\subsection{Proof of Corollary~\ref{hclbstrict}}

\begin{proof}
The map $\hh\mapsto S^{\rm que}(\hb,\hh,\bb;0)$ is strictly decreasing and convex on 
$(\hat h^{\rm ann}_c(\hb/\alpha),\infty)$ (recall Fig.~\ref{fig-copadvarhs}). Therefore, 
for $\th<s^\ast(\hb,\bb,\alpha)$, the $\hh$-value that solves the equation $S^{\rm que}
(\hb,\hh,\bb;0)=\th$ is strictly greater than $\hat h^{\rm ann}_c(\hb/\alpha)$, which 
proves that $h^{\rm que}_c(\hb,\bb,\th)>\hat h^{\rm ann}_c(\hb/\alpha)$. The proof for 
$\th\geq s^\ast(\hb,\bb,\alpha)$ follows from Corollary~\ref{hquehannvarfor} and 
\eqref{ccqueflatp}.  
\end{proof}


\subsection{Proof of Corollary~\ref{critqueann}}

\begin{proof}
(i) Note from \eqref{scopadf} that 
\begin{equation}
\label{s6.37}
S^{\rm ann}(\hb,\hh,\bb;0)-\th=\bar M(-\hb)-\th+\log\left(\tfrac12
\left[1+\sum_{n\in\N}\rho(n)e^{n[\hM(2\hb)-2\hb \hh]}\right]\right).
\end{equation}
Note from \eqref{s6.37} and (\ref{Scopann}--\ref{Sanpin}) that  
$S^{\rm ann}(\hb,\hat h^{\rm ann}_c(\hb),\bb;0)-\th=\bar S^{\rm ann}(\bb;0)-\th$ and 
$S^{\rm ann}(\hb,\hh,\bb;0)$ $-\th^{\rm ann}_c(\bb)=\hat S^{\rm ann}(\hb,\hh;0)$. 
These observations, together with the remark below Theorem~\ref{varflocann}, conclude 
the proof for (i).\\
(ii) Recall from \eqref{Sdefalt} that 
\begin{equation}
\begin{split}
S_*^\mathrm{que}(\hb,\hh,0) &= \sup_{Q\in\cC^\mathrm{fin}} 
\left[\Phi_{\hb, \hh}(Q)-I^\mathrm{que}(Q)\right]\cr
&=\sup_{\hQ\in\hat\cC^\mathrm{fin}} 
\left[\Phi_{\hb, \hh}(\hQ)-\hat I^\mathrm{que}(\hQ)\right]
=\hat S^{\rm que}(\hb,\hh;0).
\end{split}
\end{equation}
The second equality uses the remark below Theorem~\ref{varfloc}. Hence $\tilde
h_c^{\rm que}(\hb,0)=S_\ast^\mathrm{que}(\hb,\hat h^{\rm que}_c(\hb),0)=\hat 
S^\mathrm{que}(\hb,\hat h^{\rm que}_c(\hb);0)=0$ by \cite{BodHoOp11}, Theorem 1.1(ii).
\end{proof}

 
\subsection{Proofs of Corollaries \ref{delocpathprop} and \ref{locpathprop}}

\begin{proof}
The proofs are similar to those of Corollaries 1.6--1.7 in \cite{BodHoOp11}, Section 8. 
For the former, all that is needed is $S^{\rm que}(\hb,\hh,\bb;0)-\th<0$, which holds for 
$(\hb,\hh,\bb,\th)\in {\rm int}(\cD^{\rm que}_1)\cup(\cD^{\rm que}\setminus\cD^{\rm que}_1)$.
\end{proof}


\appendix


\section{Finiteness of $\Phi$}
\label{appA}

\begin{lemma}
\label{pinphifin}
Fix $\delta>0$ and $\bm\in\cP(\bE)$ satisfying \eqref{mgffin}. Then, for all $\bQ\in
\cP^{\rm inv}(\widetilde{\bE}^\N)$ with $h(\pi_{1,1}\bQ\mid\bm)<\infty$, there 
are constants $\gamma \in (\delta^{-1},\infty)$ and $K(\delta,\gamma,\bm)\in (0,\infty)$ 
such that 
\begin{equation}
\label{finphiq}
|\Phi(\bQ)|<\gamma \,h(\pi_{1,1}\bQ\mid\bm)+ K(\delta, \gamma,\bm).
\end{equation} 
\end{lemma}

\begin{proof}
The proof comes in 3 steps.

\medskip\noindent
{\bf 1.}
Abbreviate 
\begin{equation}
f(\bo_1) = \frac{d(\pi_{1,1} \bQ)}{d\bm}( \bo_1), \qquad \bo_1 \in \bE.
\end{equation}
Fix $\gamma\in(\delta^{-1},\infty)$. For $m\in\N$ and $l\in\Z$, define
\begin{equation}
\label{setdef} 
\begin{aligned}
A_{m} &= \{\bo_1\in\bE\colon\, m-1\leq \g\log f( \bo_1)<m\},\\
A_{0} &= \{\bo_1\in\bE\colon\, 0\leq f( \bo_1)<1\},\\
B_{l} &= \{\bo_1\in\bE\colon\, l-1\leq \bo_1<l\}.
\end{aligned}
\end{equation} 
Note that the $A_m$'s and the $B_l$'s are pairwise disjoint, and that 
\begin{equation}
\bE = A_{0} \cup \left[\cup_{m\in\N} A_{m}\right],
\quad \bE_+\cup\{0\}=\cup_{l\in\N} B_{l},
\quad  \bE_-=\cup_{l\in-\N_0} B_{l},
\end{equation}
where $\bE_+$ and $\bE_-$ denote the set of positive and negative real numbers in 
$\bE$. Also note that
\begin{equation}
\begin{aligned}
\Phi( \bQ)
&= \int_{ \bE} \;\bo_1\;(\pi_{1,1} \bQ)(d \bo_1)
\leq  \int_{ \bE } (0\vee  \bo_1)\;f( \bo_1)\;\bm(d \bo_1)\\
&=\sum_{m\in\N_0} \int_{A_m } (0\vee  \bo_1)\;f( \bo_1)\;\bm(d \bo_1)
= I +II+III,
\end{aligned}
\end{equation}
where 
\begin{equation}
\label{Ibds}
\begin{aligned}
I &= \int_{A_0\cap[\cup_{l\in\N} B_l] } \bo_1\;f(\bo_1)\;\bm(d\bo_1)
\leq \sum_{l\in\N}l\, \P_{\bo}(B_l ),\cr
II &=  \sum_{m\in\N} \int_{A_m \cap[\cup_{l=1}^{m-1} B_l]}  
\bo_1\;f( \bo_1)\;\bm(d \bo_1),\cr
III &=  \sum_{m\in\N} \int_{A_m\cap[\cup_{l\in\N_0}B_{m+l}]} 
\bo_1\;f( \bo_1)\;\bm(d \bo_1)
\leq \sum_{m\in\N} e^{m/\gamma}\sum_{l\in\N_0} (m+l) \P_{\bo}(B_{m+l}).
\end{aligned}
\end{equation}
The term $I$ follows from the restriction of the $\bm$-integral to the set $A_0\cap \bE_+$.
The terms $II$ and $III$ follow from the restrictions to the sets $\cup_{m\in\N}[A_m\cap
\cup_{l=1}^{m-1} B_l]$ and $\cup_{m\in\N}[A_m\cap\cup_{l\in\N_0} B_{m+l}]$. The bound on 
$I$ uses that $f<1$ on $A_0$ and $\bo_1<l$ on $B_l$. The bound on $III$ follows from the 
fact that $f<e^{m/\gamma}$ on $A_m$ and $\bo_1<m+l$ on $B_{m+l}$. It follows from 
\eqref{Ibds} that 
\begin{equation}
\begin{aligned}
I +III &\leq 2 \sum_{m\in\N} e^{m/\gamma}\sum_{l\in\N_0} (m+l) \P_{\bo}(B_{m+l})\\
& \leq 2 \sum_{m\in\N} e^{m/\gamma}\sum_{l\in\N_0} (m+l) \P_{\bo}(\bo_1\geq m+l-1)\\
& \leq2 C(\delta) \sum_{m\in\N} e^{m/\gamma}\sum_{l\in\N_0} (m+l) e^{-\delta(m+l-1)}\\
& \leq2 C(\delta)\,e^{\d} \sum_{m\in\N} e^{-m(\delta-1/\gamma)}\sum_{l\in\N_0} (m+l) 
e^{-\delta l}
=k_+(\delta,\gamma,\bm)<\infty,
\end{aligned}
\end{equation}
where the third inequality uses \eqref{mgffin}. Moreover, use that $\bo_1<m-1\leq \g\log f$ 
on $ A_{m} \cap \cup_{l=1}^{m-1} B_{l}$, to estimate
\begin{equation}
\begin{split}
II &\leq \g  \sum_{m\in\N} \int_{\bo_1\in A_{m} \cap [\cup_{l=1}^{m-1} B_{l}]} 
f(\bo_1)\log f(\bo_1)\,\bm(d \bo_1)\cr
&\leq \g  \sum_{m\in\N} \int_{\bo_1\in A_{m}} 
f(\bo_1)\log f(\bo_1)\,\bm(d \bo_1)\cr
&= \g  \int_{ \bE\backslash A_{0}}
f(\bo_1)\log f(\bo_1)\,\bm(d \bo_1) \cr
&=\g h(\pi_{1,1} \bQ\mid\bm)-\g\int_{ A_{0}}
f(\bo_1)\log f(\bo_1)\,\bm(d \bo_1)\cr
&\leq \g h(\pi_{1,1} \bQ\mid\bm)+\g\,e^{-1}<\infty,
\end{split}
\end{equation}
where the third inequality uses that $f\log f\geq -e^{-1}$ on $A_0$, and the second 
equality that
\begin{equation}
h(\pi_{1,1} \bQ|\bm) =  \int_{\bE\backslash A_{0}}
f(\bo_1)\log f(\bo_1)\,\bm(d \bo_1)+ \int_{ A_{0}}
f(\bo_1)\log f(\bo_1)\,\bm(d \bo_1)<\infty.
\end{equation}
Put 
\begin{equation}
\begin{split}
K_+(\delta,\g,\bm)=k_+(\delta,\g,\bm)+\g\,e^{-1}.
\end{split}
\end{equation}

\medskip\noindent
{\bf 2.} Similarly, we have
\begin{equation}
\begin{aligned}
\Phi(\bQ)
&= \int_{ \bE} \;\bo_1\;(\pi_{1,1} \bQ)(d \bo_1)
\geq  \int_{ \bE } (0\wedge  \bo_1)\;f( \bo_1)\;\bm(d \bo_1)\\
&=\sum_{m\in\N_0} \int_{A_m } (0\wedge  \bo_1)\;f( \bo_1)\;\bm(d \bo_1)
= I' +II'+III',
\end{aligned}
\end{equation}
where
\begin{equation} 
\begin{aligned}
I' &= \int_{A_0\cap[\cup_{l\in-\N_0} B_l] } \bo_1\;f(\bo_1)\;\bm(d \bo_1)
\geq \sum_{l\in-\N_0}(l-1)\, \P_{\bo}(B_l ),\cr
II' &= \sum_{m\in\N} \int_{A_m \cap[\cup_{l=-m+1}^{0} B_l]} 
\bo_1\;f(\bo_1)\;\bm(d \bo_1),\cr
III' &=  \sum_{m\in\N} \int_{A_m\cap[\cup_{l\in-\N_0}B_{l-m}] } 
\bo_1\;f(\bo_1)\;\bm(d \bo_1) \geq \sum_{m\in\N} e^{m/\gamma}
\sum_{l\in-\N_0} (l-m-1) \P_{\bo}(B_{l-m}).
\end{aligned}
\end{equation}
The bounds on $I'$  and $III'$ use that $\bo_1\geq l-1$ on $B_l$ and $f<e^{m/\g}$ on 
$A_m$. Note that  
\begin{equation} 
\begin{aligned}
I'+III' &\geq2 \sum_{m\in\N_0} e^{m/\gamma}\sum_{l\in-\N_0} (l-m-1) \P_{\bo}(B_{l-m})\\
& \geq2\sum_{m\in\N_0} e^{m/\gamma}\sum_{l\in-\N_0} (l-m-1) \P_{\bo}(\bo_1\leq l-m)\\
&\geq2  C(\delta)\sum_{m\in\N_0} e^{m/\gamma}\sum_{l\in-\N_0} (l-m-1) e^{\delta(l-m)}
=k_-(\delta,\g,\bm)>-\infty.
\end{aligned}
\end{equation}
Also use that $\bo_1\geq-m\geq -[\g\log f+1]~$ on $~A_{m} \cap \cup_{l=-m+1}^{0} B_{l}$, 
to estimate
\begin{equation}
\begin{split}
II' &\geq   -\sum_{m\in\N} \int_{\bo_1\in A_{m} \cap [\cup_{l=-m+1}^{0} B_{l}]} 
f( \bo_1)[\g\log f( \bo_1)+1]\,\bm(d \bo_1)\cr
&\geq -\g  \sum_{m\in\N} \int_{ A_{m}} 
f( \bo_1)\log f( \bo_1)\,\bm(d \bo_1)-1\cr
&= -\g  \int_{ \bE\backslash A_{0}}
f( \bo_1)\log f( \bo_1)\,\bm(d \bo_1) -1\cr
&=-\g h(\pi_{1,1} \bQ\mid\bm)+\g\int_{ A_{0}}
f( \bo_1)\log f( \bo_1)\,\bm(d \bo_1)-1\cr
&\geq -[\g h(\pi_{1,1} \bQ\mid\bm)+\g e^{-1}+1]>-\infty.
\end{split}
\end{equation}

\medskip\noindent
{\bf 3.} Put $K_-(\delta,\g,\bm)=1+\g e^{-1}-k_-(\delta,\g,\bm)$. Then the claim follows 
with $K(\delta,\g,\bm)=K_+(\delta,\g,\bm)\vee K_-(\delta,\g,\bm)$.
\end{proof}

For the sake of completeness we state the follow finiteness results for $\Phi_{\hb, \hh}$ 
that were proved in \cite{BodHoOp11}, Appendix A.

\begin{lemma}
\label{mainlemmacop}
Fix $\hb, \hh,g>0$. Then $\ho$-a.s.\ there exists a $K(\ho,\hb, \hh,g)<\infty$ such 
that, for all $N\in\N$ and for all sequences $0=k_0 < k_1 < \cdots < k_N < \infty$,
\begin{equation}
\label{sumbd}
-g k_N + \sum_{i=1}^N \log\phi_{\hb, \hh}\big(\ho_{(k_{i-1},k_i]}\big)
\leq K(\ho,\hb, \hh,g)N,
\end{equation}
where $\ho_{(k_{i-1},k_i]}$ is the word cut out from $\ho$ by the $i$th excursion interval 
$(k_{i-1},k_i]$.
\end{lemma}

\begin{lemma}
\label{finitephicop} 
Fix $\hb, \hh>0$, $\rho\in\cP(\N)$ and $\hm\in\cP(\R)$ satisfying \eqref{rhocond} and 
\eqref{mgffin}. Then, for all $\hQ\in\cP^\mathrm{inv}(\widetilde{\hE}^\N)$ 
with $h(\pi_1 \hQ \mid q_{\rho,\hm})<\infty$, there are finite constants $C>0$, 
$\gamma>2\hb/C$ and $K=K(\hb, \hh,\rho,\hm,\gamma)$ such that 
\begin{equation}
\Phi_{\hb, \hh}(\hQ)\leq \gamma\; h(\pi_1 \hQ|q_{\rho,\hm})+K.
\end{equation}
\end{lemma}


\section{Application of Varadhan's lemma}
\label{appB}

In this appendix we prove \eqref{seqsbar} and the claim above it. This was used in
Section~\ref{S3} to complete the proof of Theorem~\ref{varfloc}.

\begin{lemma}
\label{varlem}
For every $\hb,\hh>0$ and $\bb\geq 0$,
\begin{equation}
\label{vareq}
s^\mathrm{que}(\hb,\hh,\bb;g) = S^\mathrm{que}(\hb,\hh,\bb;g) \quad \forall\, g\in\R,
\end{equation}
 where $s^\mathrm{que}(\hb,\hh,\bb;g)$ is the $\omega$-a.s.\ 
constant limit defined in \eqref{copadSlim}, and $S^\mathrm{que}(\hb,\hh,\bb;g)$ is as 
in \eqref{Sdef}. In particular, the map $g\mapsto S^\mathrm{que}(\hb,\hh,\bb;g)$ is finite 
on $(0,\infty)$ and infinite on $(-\infty,0)$.
\end{lemma}

\begin{proof}
Throughout the proof $\hb,\hh>0$, $\bb\geq 0$ and $\th\in\R$ are fixed. The proof comes 
in 3 steps, where we establish the equality in \eqref{varlem} for the cases $g<0$, $g=0$ 
and $g>0$ separately.

\medskip\noindent
{\bf Step 1.} 
For $g<0$ the proof of \eqref{varlem} is given in two steps.

\medskip\noindent
{\bf 1a.} In this step we show that $S^\mathrm{que}(\hb,\hh,\bb;g)=\infty$ when 
$g<0$. Fix $L\in\N$ and let $Q^L=(q^L_{\hm\otimes\bm})^{\otimes\N}$, with 
\begin{equation}
q^L_{\hm\otimes\bm}(d\o_1,\ldots,d\o_n)
=\delta_{Ln}(\hm\otimes\bm)^{\otimes n}(d\o_1,\ldots,d\o_n) 
\end{equation}
and $(\o_1,\ldots,\o_n)=((\ho_1,\bo_1),\ldots,(\ho_n,\bo_n))\in E^n.$ It follows 
from \eqref{Sdef} that 
\begin{equation}
S^\mathrm{que}(\hb,\hh,\bb;g)
\geq \bb \Phi(Q^L)+\Phi_{\hb,\hh}(Q^L)-I^{\rm ann}(Q^L)-gL
\geq -\log 2-gL+\log \rho(L).
\end{equation}
The second inequality uses that $\Phi(Q^L)=0$, $I^{\rm ann}(Q^L)=-\log\rho(L)$ and 
$\Phi_{\hb,\hh}(Q^L)\geq -\log 2$. Letting $L\to\infty$ and using that $\rho$ has a 
polynomial tail by \eqref{rhocond}, we get the claim.

\medskip\noindent
{\bf 1b.} In this step we show that $s^\mathrm{que}(\hb,\hh,\bb;g)=\infty$ when $g<0$. 
The proof follows from a moment estimate. We start by showing that, for each $\bb\in\R$, 
\begin{equation}
\label{mest}
\limsup_{N\to\infty} \frac1N \log E^\ast_0\left(e^{N\bb\Phi(R^{\bo}_N)}\right)
\leq \tM(-\bb)
\end{equation}
(recall \eqref{mgffin}). Indeed, for any $\bb\in\R$, by the Markov inequality,
\begin{equation}
\label{B5}
\begin{split}
\P_{\bo}\left(\frac1N\log E^\ast_0\left(e^{N\bb\Phi(R^{\bo}_N)}\right)
\geq \tM(-\bb)+\epsilon\right)
&= \P_{\bo}\left( E^\ast_0\left(e^{N\bb\Phi(R^{\bo}_N)}\right)
\geq e^{N( \tM(-\bb)+\epsilon)}\right)\cr
&\leq e^{-N\tM(-\bb)}e^{-\epsilon N}
\E_{\bo}\left(E^\ast_0\left(e^{N\bb\Phi(R^{\bo}_N)}\right)\right)\cr
&=e^{-N\tM(-\bb)}e^{-\epsilon N}
E_0^\ast\left[\E_{\bo}\left(e^{\bb\sum_{i=1}^N\bo_{k_{i-1}}}\right)\right]
=e^{-\epsilon N}.
\end{split}
\end{equation}
The claim therefore follows from the Borel-Cantelli lemma. Note from \eqref{B5}  that 
\eqref{mest} holds if we replace $E^\ast_0$ with $E^\ast_g$, $g\geq0.$

Let $\tau_i$ be the length of the $i$-th word, let $L\in\N$, and put
\begin{equation}
k_N=\sum_{i=1}^N\tau_i\quad \text{and}\quad k_N(L)
=\sum_{i=1}^N\left[\tau_i~1_{\{\tau_i<L\}}+ L\,1_{\{\tau_i\geq L\}}\right].
\end{equation}
For any $-\infty<q<0<p<1$ with $p^{-1}+q^{-1}=1$ and $g<0$, it follows from \eqref{FNdef} 
that 
\begin{equation}
\label{FNdef1}
\begin{aligned}
e^{\th N}F_N^{\hb,\hh,\bb,\th,\o}(g) 
&=E^\ast_0\left(\exp\left [-gk_N+N\left(\Phi_{\hb,\hh}(R_N^{\o})+\bb\,
\Phi(R_N^{\o})\right)\right]\right),\\
&\geq \left(\tfrac12\right)^N\,E^\ast_0\left(\exp\left [-gk_N(L)+N\bb\,
\Phi(R_N^{\o})\right]\right)\\
&\geq \left(\tfrac12\right)^N\,E^\ast_0\left(e^{-g pk_N(L)}\right)^{1/p}
E^\ast_0\left(e^{Nq\bb\,\Phi(R_N^{\o})}\right)^{1/q} \\
& =\left(\tfrac12 \right)^N\,\cN_L(pg)^{N/p}
E^\ast_0\left(e^{Nq\bb\,\Phi(R_N^{\o})}\right)^{1/q},  
\end{aligned}
\end{equation}
where 
\begin{equation}
\cN_L(g)=\sum_{n=1}^{L-1}\rho(n)e^{-ng}+e^{-Lg}\sum_{n\geq L}\rho(n).
\end{equation}
The first inequality in \eqref{FNdef1} uses that $\Phi_{\hb,\hh}(Q)\geq-\log 2$, $k_N\geq
k_N(L)$ and $g<0$. The second inequality follows from the reverse H\"older inequality 
with the above choice of $p$ and $q$. Note that $\cN_L(g)$ is finite for $g\in\R$ and 
$\lim_{L\to\infty}\cN_L(g)=\cN(g)$. It therefore follows from \eqref{copadSlim}, 
\eqref{mest} and \eqref{FNdef1} that 
\begin{equation}
\label{Slimlb}
\begin{split}
s^{\rm que}(\hb,\hh,\bb;g)
&=\limsup_{N\to\infty}\frac1N \log \left(e^{N \th}
F_N^{\hb,\hh,\bb,\th,\o}(g)\right)\cr
&\geq -\log 2+\frac1p\log\cN_L(pg)+\frac1q \tM(-\bb q).
\end{split}
\end{equation}
Letting $L\to\infty$, we get from \eqref{rhoz} that $s^{\rm que}(\hb,\hh,\bb;g)
=\infty$, since $\cN(pg)=\infty$ for $g\in(-\infty,0)$.

\medskip\noindent
{\bf Step 2.} In this step, which is divided into 2 substeps, we consider the case 
$g>0$.

\medskip\noindent
{\bf 2a. Lower bound:}
For $M>0$, define 
\begin{equation}
\label{pintrunc}
\begin{split}
\Phi^{-M}(Q)=\int_{\bE}(\bar\pi_{1,1}Q)(d\bo_1)[\bo_1\vee (-M)].
\end{split}
\end{equation}
Note that $\Phi^{-M}$  is lower semi-continuous and that 
\begin{equation}
\label{phiineq}
\begin{split}
\bb\Phi^{-M}(Q)+\Phi_{\hb,\hh}(Q)\leq \bb\Phi(Q)+\Phi_{\hb,\hh}(Q)
-\bb\int_{\bo_1<-M} \bo_1\,(\bar\pi_{1,1} Q)(d\bo_1).
\end{split}
\end{equation}
Therefore, for any $p,q>1$ with $1/p+1/q=1$, it follows from the H\"older inequality 
that 
\begin{equation}
\label{smineq}
\begin{split}
&\frac{1}{N} \log E_g^\ast\left(e^{N\left(\Phi_{\hb,\hh}(R_N^\o)
+\bb\,\Phi^{-M}(R^\o_N)\right)}\right)\cr
&\leq \frac{1}{pN} \log E_g^\ast\left(e^{pN\left[\Phi_{\hb,\hh}(R_N^\o)
+\bb\,\Phi(R^\o_N)\right]}\right)
+ \frac{1}{qN} \log  E_g^\ast\left(e^{-q\,\bb\sum_{i=1}^N \bo_{k_{i}} 
1_{\{\bo_{k_{i}}<-M\}}}\right).
\end{split}
\end{equation}
The rest of the proof consists of taking the appropriate limits and showing that 
the left-hand side of \eqref{smineq} is bounded from below by $S^{\rm que}(\hb,
\hh,\bb;g)$, while the second term in the right-hand side tends to zero and the 
first term tends to $s^{\rm que}(\hb,\hh,\bb;g)$.

Let us start with the second term in the right-hand side of \eqref{smineq}. Note 
from \eqref{rhoz} that
\begin{equation}
\label{trunerr}
\begin{split}
&\limsup_{N\to\infty}\frac{1}{qN}\log  E_g^\ast\left(e^{-q\,
\bb\sum_{i=1}^N \bo_{k_{i}}1_{\{\bo_{k_{i}<-M\}}}}\right)\cr
&=-\frac1q\log\cN(g)+ \limsup_{N\to\infty}\frac{1}{qN}\log 
E_0^\ast\left(e^{-gk_N-q\,
\bb\sum_{i=1}^N \bo_{k_{i}}1_{\{\bo_{k_{i}<-M\}}}}\right)\cr
&\leq -\frac1q\log\cN(g)+ \frac{1}{2q}\log\cN(2g)
+\limsup_{N\to\infty}\frac{1}{2qN}\log E_0^\ast\left(e^{-2q\,
\bb\sum_{i=1}^N \bo_{k_{i}}1_{\{\bo_{k_{i}<-M\}}}}\right)\cr
&\leq-\frac1q\log\cN(g)+\frac{1}{2q}\log\cN(2g)+ \frac{1}{2q} \log 
\int_{\bE}e^{-2q\bb\bo_1 1_{\{\bo_{1}<-M\}}}\bm(d\bo_1).
\end{split}
\end{equation}
The first inequality uses the Cauchy-Schwarz inequality, the second inequality uses 
\eqref{mest}. Note from \eqref{mgffin} that the above bound tends to zero upon when
$M\to\infty$ followed by $q\to\infty$.
 
For the first term in the right-hand side of \eqref{smineq} we proceed as follows. 
Note from Lemma~\ref{mainlemmacop} that $\ho$-a.s. 
\begin{equation}
\label{copcon}
-gk_N+pN\Phi_{\hb,\hh}(R_N^x)\leq NK(\ho,p,\hb,\hh,g),
\end{equation}
where we use that 
\begin{equation}
-gk_N+pN\Phi_{\hb,\hh}(R_N^\o)
=p\left(-\frac{g}{p}k_N+N\Phi_{\hb,\hh}(R_N^\o)\right).
\end{equation}
Therefore, for any $1<p<\infty$, it follows from \eqref{mest} and \eqref{copcon} 
that $\ho$-a.s.
\begin{equation}
\label{fullgr}
\begin{split}
\limsup_{N\to\infty}\frac{1}{pN}\log 
&E_g^\ast\left(e^{Np\,\left[\Phi_{\hb,\hh}(R_N^\o)
+\bb\,\Phi(R^\o_N)\right]}\right)\cr
&=-\frac1p\log\cN(g)
+\limsup_{N\to\infty}\frac{1}{pN}\log 
E_0^\ast\left(e^{-gk_N+Np\,\left[\Phi_{\hb,\hh}(R_N^\o)
+\bb\,\Phi(R^\o_N)\right]}\right)\cr
&\leq \frac1p\left(K(p,\hb,\hh,g)+\tM(-p\bb)-\log\cN(g)\right)<\infty.
\end{split}
\end{equation}

Next, for $-\infty<r<0<s<1$ with $r^{-1}+s^{-1}=1$, it follows from the argument 
leading to \eqref{Slimlb} that 
\begin{equation}
\label{fullgrlb}
\begin{split}
\limsup_{N\to\infty}\frac{1}{pN}\log 
&E_g^\ast\left(e^{Np\,\left[\Phi_{\hb,\hh}(R_N^\o)
+\bb\,\Phi(R^\o_N)\right]}\right)\cr
&\geq \log\frac12+\frac{1}{p}\cN(g) + \frac{1}{sp} \cN(sg)
+\frac{1}{pr}\tM(-pr\bb) >-\infty,
\end{split}
\end{equation}
since $g\in(0,\infty)$. Define
\begin{equation}
\label{spdef}
S(p) = \limsup_{N\to\infty}\frac{1}{N}\log 
E_g^\ast\left(e^{Np\,\left[\Phi_{\hb,\hh}(R_N^\o)
+ \bb\,\Phi(R^\o_N)\right]}\right).
\end{equation}
By (\ref{fullgr}--\ref{fullgrlb}), the map $p\mapsto S(p)$ is convex and finite on 
$(0,\infty)$, and hence continuous on $(0,\infty)$. It therefore follows from 
\eqref{Slimdef} that the left-hand side of \eqref{fullgr} converges to $s^{\rm que}
(\hb,\hh,\bb;g)-\log\cN(g)$ as $p \downarrow 1$. It follows from (\ref{fullgr}--\ref{fullgrlb}) 
that this limit is finite, which proves the finiteness of the map $g\mapsto s^{\rm que}
(\hb,\hh,\bb;g)$ on $(0,\infty)$.

Finally, we turn to the left-hand side of the inequality in \eqref{smineq}. For any 
$\epsilon>0$ and $Q\in\cC^{\rm fin}\cap\cR$, note from the lower semi-continuity of 
the map $Q\mapsto \bb\Phi^{-M}(Q)+\Phi_{\hb,\hh}(Q)$ that the set 
\begin{equation}
A_\epsilon(Q)=\left\{Q'\in\cP^{\rm inv}(\widetilde{E}^\N)\colon\, 
\bb\Phi^{-M}(Q')+\Phi_{\hb,\hh}(Q')
\geq \bb\Phi^{-M}(Q)+\Phi_{\hb,\hh}(Q)-\epsilon\right\}
\end{equation}
is open. This implies that
\begin{equation}
\label{fullMgr}
\begin{split}
\limsup_{N\to\infty}\frac{1}{N}\log 
&E_g^\ast\left(e^{N\,\left[\Phi_{\hb,\hh}(R_N^\o)
+ \bb\,\Phi^{-M}(R^\o_N)\right]}\right)\cr
&\geq \liminf_{N\to\infty}\frac{1}{N}
\log E_g^\ast\left(e^{N\,\left[\Phi_{\hb,\hh}(R_N^\o)
+ \bb\,\Phi^{-M}(R^\o_N)\right]}1_{A_\epsilon(Q)}(R_N^\o)\right)\cr
&\geq \bb\Phi^{-M}(Q)+\Phi_{\hb,\hh}(Q)
-\epsilon-\inf_{Q'\in A_\epsilon(Q)}I_g^{\rm que}(Q')\cr
&\geq \bb\Phi^{-M}(Q)+\Phi_{\hb,\hh}(Q)-I_g^{\rm que}(Q)-\epsilon\cr
&\geq  \bb\Phi(Q)+\Phi_{\hb,\hh}(Q)-I_g^{\rm que}(Q)-\epsilon\cr
&= \bb\Phi(Q)+\Phi_{\hb,\hh}(Q)-gm_Q-I^{\rm ann}(Q)-\log\cN(g)-\epsilon.
\end{split}
\end{equation}
The second inequality uses Theorem~\ref{qLDP}, the third inequality uses that $Q\in
A_\epsilon(Q)$, the fourth inequality follows from the fact that $\Phi\leq\Phi^{-M}$,
while the equality follows from Lemma~\ref{Iannz}. It therefore follows from 
(\ref{smineq}--\ref{trunerr}), \eqref{fullMgr} and the comment below \eqref{spdef}
that, after taking the supremum over $\cC^{\rm fin}\cap\cR$ followed by $M\to\infty$, 
$\epsilon\to 0$ and $p \downarrow 1$, 
\begin{equation}
s^{\rm que}(\hb,\hh,\bb;g) \geq \sup_{Q\in\cC^{\rm fin}\cap\cR}
\left[\bb\Phi(Q)+\Phi_{\hb,\hh}(Q)-gm_Q-I^{\rm ann}(Q)\right]
= S^{\rm que}(\hb,\hh,\bb;g).
\end{equation}

\medskip\noindent
{\bf 2b. Upper bound:} 
Let $\chi(\ho_{I_i})=\log\phi_{\hb,\hh}(\ho_{I_i})$. For $M>0$, define 
\begin{equation}
\label{trunc}
\begin{split}
\Phi^M(Q)&=\int_{\bE}(\bar\pi_{1,1}Q)(d\bo_1)(\bo_1\wedge M),\cr
\Phi_{\hb,\hh}^{M}(Q)
&=\int_{\widetilde{\hE}}(\hat\pi_{1}Q)(d\ho_{I_1})(\chi(\ho_{I_1})\wedge M).
\end{split}
\end{equation}
Note that $\Phi^M$ and $\Phi_{\hb,\hh}^{M}$ are upper semi-continuous and 
\begin{equation}
\label{phiineq1}
\begin{split}
\bb \Phi(Q)+ \Phi_{\hb,\hh}(Q)-\int_{\ho_1\geq M} \bo_1\,(\bar\pi_{1,1}Q)(d\bo_1)
- \int_{\chi\geq M}(\hat\pi_{1}Q)(d\ho_{I_1})\chi(\ho_{I_1})\leq
\bb\Phi^M(Q)+ \Phi_{\hb,\hh}^{M}(Q). 
\end{split}
\end{equation}
Therefore, for any $-\infty<q<0<p<1$ with $q^{-1}+p^{-1}=1$, the reverse H\"older 
inequality gives
\begin{equation}
\label{upineq}
\begin{split}
&\frac{1}{N} \log 
E^\ast_g\left(e^{N\left[\bb\Phi^M(R_N^\o)+\Phi^M_{\hb,\hh}(R_N^\o)\right]}\right)\cr
&\qquad \geq \frac{1}{qN} \log E^\ast_g\left(e^{-q\left[\bb\sum_{i=1}^N\bo_{k_{i}}
1_{\{\bo_{k_{i}\geq M}\}}
+ \sum_{i=1}^N\chi(\ho_{I_i})1_{\{\chi(\ho_{I_i})\geq M\}}\right]}\right)\cr    
&\qquad\qquad + \frac{1}{pN} \log E^\ast_g\left(e^{pN\left[\bb\Phi(R_N^\o)
+\Phi_{\hb,\hh}(R_N^\o)\right]}\right). 
\end{split}
\end{equation}
The rest of the proof for the upper bound follows after showing that the left-hand side 
of \eqref{upineq} gives rise to the desired upper bound, while the right-hand side 
gives rise to $s^{\rm que}(\hb,\hh,\bb;g)$ after taking appropriate limits. 

It follows from \cite{BodHoOp11}, Step 2 in the proof of Lemma B.1, that 
\begin{equation}
\label{gnfin}
\frac gq k_N + \sum_{i=1}^N\chi(\ho_{I_i})1_{\{\chi(\ho_{I_i})\geq M\}}\leq 0
\end{equation}
for $M$ large enough. Hence, for $M$ large enough, it follows from \eqref{mest}, 
\eqref{gnfin} and $q<0$ that
\begin{equation}
\label{copadtrunerr}
\begin{split}
&\limsup_{N\to\infty}\frac{1}{qN}\log 
E^\ast_g\left(e^{-q\left[\bb\sum_{i=1}^N\bo_{k_{i}}
1_{\{\bo_{k_{i}\geq M}\}}+\sum_{i=1}^N\chi(\ho_{I_i})
1_{\{\chi(\ho_{I_i})\geq M\}}\right]}\right)\cr
&=-\frac1q\log\cN(g)+\limsup_{N\to\infty}\frac{1}{qN}\log 
E^\ast_0\left(e^{-q\left[\frac{g}{q}
k_N+\sum_{i=1}^N\chi(\ho_{I_i})
1_{\{\chi(\ho_{I_i})\geq M\}}+\bb\sum_{i=1}^N\bo_{k_{i}}
1_{\{\bo_{k_{i}\geq M}\}}\right]}\right)\cr
&\geq \frac1q\left(\log 
\int_{\R}e^{-q\bb\bo_1 1_{\{\bo_{1}\geq M\}}}\bm(d\bo_1)-\log\cN(g)\right),
\end{split}
\end{equation}
which tends to zero as $M\to\infty$ followed by $q\to-\infty$. Furthermore, it follows 
from (\ref{fullgr}--\ref{spdef}) and the remark below \eqref{spdef} that 
\begin{equation}
\label{upfullgr}
s^{\rm que}(\hb,\hh,\bb;g)-\log\cN(g)
=\lim_{p \uparrow 1}\limsup_{N\to\infty}\frac{1}{pN}\log 
E^\ast_g\left(e^{pN\left[\bb\Phi(R_N^\o)+\Phi_{\hb,\hh}(R_N^\o)\right]}\right).
\end{equation}
Since $\bb\Phi^M+\Phi^M_{\hb,\hh}$ is upper semi-continuous, it follows from Dembo 
and Zeitouni \cite{DeZe98}, Lemma 4.3.6, and Theorem~\ref{qLDP} that  
\begin{equation}
\label{upfullmgr}
\begin{split}
\limsup_{N\to\infty}\log E^\ast_g
&\left(e^{N\left[\bb\Phi^M(R_N^\o)+\Phi^M_{\hb,\hh}(R_N^\o)\right]}\right)
\leq \sup_{Q\in\cP^{\rm inv}(\widetilde{E}^{\otimes\N})}
\left[\bb\Phi^M(Q)+\Phi^M_{\hb,\hh}(Q)-I_g^{\rm que}(Q)\right]\cr
&=\sup_{Q\in\cR}\left[\bb\Phi^M(Q)+\Phi^M_{\hb,\hh}(Q)-I_g^{\rm que}(Q)\right]\cr
&=\sup_{Q\in\cC^{\rm fin}\cap\cR}\left[\bb\Phi^M(Q)+\Phi^M_{\hb,\hh}(Q)-gm_Q-I^{\rm ann}(Q)\right]-\log\cN(g)\cr
&\leq \sup_{Q\in\cC^{\rm fin}\cap\cR}\left[\bb\Phi(Q)+\Phi_{\hb,\hh}(Q)-gm_Q-I^{\rm ann}(Q)\right]-\log\cN(g).
\end{split}
\end{equation}
The first equality uses that $I_g^{\rm que}(Q)=\infty$ for $Q\notin\cR$ (recall \eqref{eqgndefinitionIalgz}) and the fact that $\hb\Phi^M+\Phi^M_{\hb,\hh}\leq 
M(1+\bb)$, the second equality uses \eqref{Iannz} and the fact that $I_g^{\rm que}
=I_g^{\rm ann}$ on $\cR$. (The removal of $Q$'s with $m_Q=I^{\rm ann}(Q)=\infty$ 
again follows from $\hb\Phi^M+\Phi^M_{\hb,\hh}\leq M(1+\bb)$), the last inequality
uses that $\Phi^M(Q)\leq \Phi(Q)$ and $\Phi^M_{\hb,\hh}(Q)\leq \Phi_{\hb,\hh}(Q)$. 
Therefore, combining (\ref{upineq}--\ref{upfullmgr}) and letting $M\to\infty$ and 
$p\uparrow 1$ in the appropriate order, we conclude the proof of the upper bound.

\medskip\noindent
{\bf Step 3.} For $g=0$ we show that 
\begin{equation}
\label{eqatg0}
\lim_{g\downarrow0}s^{\rm que}(\hb,\hh,\bb;g)=\lim_{g\downarrow0}S^{\rm que}(\hb,\hh,\bb;g)
=S^{\rm que}(\hb,\hh,\bb;0)=S_*^{\rm que}(\hb,\hh,\bb)
\end{equation}
(recall \eqref{Sdefalt}). The first equality follows from Steps 1 and 2 of the proof.
The second uses that the map  $g\mapsto S^{\rm que}(\hb,\hh,\bb;g)$ is decreasing and 
lower semi-continuous  on $[0,\infty)$. 

Furthermore,  note from \eqref{Sdef} and \eqref{Sdefalt} that 
\begin{equation}
\label{sastgeqs}
\begin{split}
S_*^{\rm que}(\hb,\hh,\bb)
&=\sup_{Q\in\cC^{\rm fin}}\left[\bb\Phi(Q)
+\Phi_{\hb, \hh}(Q)-I^{\rm que}(Q)\right]\cr
&\geq \sup_{Q\in\cC^{\rm fin}\cap\cR}\left[\bb\Phi(Q)
+\Phi_{\hb, \hh}(Q)-I^{\rm que}(Q)\right]\cr
&= S^{\rm que}(\hb,\hh,\bb;0).
\end{split}
\end{equation} 
Here we use that $I^{\rm que}=I^{\rm ann}$ on $\cR$. The rest of the 
proof will follow once we show the reverse of \eqref{sastgeqs}. To do so we  proceed as 
follows:

For $L\in\N$ and $-\infty<q<0<p<1$, with $p^{-1}+q^{-1}=1$, it follows from \eqref{Slimdef} and the 
reverse H\"older inequality that 
\begin{equation}
\label{squelb}
\begin{split}
&  s^{\rm que}(\hb,\hh,\bb;g)-\log\cN(g)= \limsup_{N\rightarrow\infty}\frac{1}{N} 
\log E_g^\ast\left(e^{N\left[\bar\beta \Phi(R_N^\o)+\Phi_{\hb,\hh}(R_N^\o)\right]}\right) \cr
&\geq\limsup_{N\rightarrow\infty}\frac{1}{pN} 
\log E_g^\ast\left(e^{Np\left[\bar\beta \Phi^{-L}(R_N^\o)+\Phi_{\hb,\hh}(R_N^\o)\right]}\right)+
\limsup_{N\rightarrow\infty}\frac{1}{qN}\log E_g^\ast\left(e^{Nq\bar\beta\sum_{i=1}^N 
\bar\o_{k_i}1_{\{\bar\o_{k_i}<-L\}}}\right)\cr
&\geq - \frac1p\log\cN(g)+ \frac1p\sup_{Q\in\cC^{\rm fin}\cap \cR}\left[p\left(\bar\beta\Phi^{-L}(Q)+
\Phi_{\hb,\hh}(Q)\right)-gm_Q-I^{\rm ann}(Q)\right]\cr
&+\frac1q\log \int_\R e^{\bar\beta q 
\bar\o_11_{\{\bar\o_{1}<-L\}}}\bar\mu(d\bar\o_1)\cr
&=\frac1p s_{L}^{\rm que}(\hb,\hh,\bb,p;g)+  \frac1q\log \int_\R e^{\bar\beta q 
\bar\o_1 1_{\{\bar\o_{1}<-L\}}}\bar\mu(d\bar\o_1)- \frac1p\log\cN(g).
\end{split}
\end{equation}
Here, $\Phi^{-L}$ is defined in \eqref{pintrunc}. The second inequality uses Steps 1 and 2 
of the proof, particularly \eqref{mest}  and the remark below \eqref{B5}. 
 
Below we show that,  for any  $p\in[0,\infty)$ and $L\in\N$,
 \begin{equation}
 \label{claim1}
 \lim_{g\downarrow0} s_{L}^{\rm que}(\hb,\hh,\bb,p;g)\geq S_{\ast}^{\rm que}(\hb,\hh,\bb,p)
 =\sup_{Q\in\cC^{\rm fin}}\left[p\left(\bar\beta\Phi(Q)+
\Phi_{\hb,\hh}(Q)\right)-I^{\rm que }(Q)\right].
 \end{equation}
Therefore,  upon taking  limits in the  order  $g\rightarrow0$,  $L\rightarrow\infty$
 and $\liminf_{p\rightarrow1}$, it follows from (\ref{squelb}, \ref{claim1}) and the 
 lower sime-continuity of the map 
$p\mapsto S_{\ast}^{\rm que}(\hb,\hh,\bb,p)$ on $[0,\infty)$  that 
\begin{equation}
\lim_{g\downarrow0} s^{\rm que}(\hb,\hh,\bb;g)\geq S_{\ast}^{\rm que}(\hb,\hh,\bb).
\end{equation}

The following lemma, which is  proved in  \cite{BodHoOp11}, Appendix C,
 will be used in the proof of  \eqref{claim1}.

\begin{lemma}
\label{squeat0rel}
Suppose that $E$ is finite. Then for every $Q\in\cP^\mathrm{inv}(\widetilde{E}^\N)$
there exists a sequence $(Q_n)$ in $\cR^\mathrm{fin}$ such that ${\rm w}-\lim_{n\to\infty}
Q_n = Q$ and $\lim_{n\to\infty} I^\mathrm{ann}(Q_n)=I^\mathrm{que}(Q)$, where  
$\cR^\mathrm{fin} = \{Q\in\cR\colon\,
m_Q<\infty\}$ and ${\rm w}-\lim$ means weak limit. 
\end{lemma}
In our case $E=\R^2$.  

For the rest of the proof we proceed as in \cite{BodHoOp11}, Appendix C. For $M\in\N$, let
\begin{equation}
\begin{aligned}
 D^\ast_M&=\left\{ -M,-M+1/M,\ldots,M-1/M,M  \right\}  
\end{aligned}
\end{equation}
be a grid of spacing $1/M$ in $[-M,M]$, which represents a finite 
set of letters approximating $\R$. Put $D_{M}= D^\ast_M\times  D^\ast_{M}$ and let
$\widetilde{D_{M}}=\cup_{n\in\N}(D^\ast_{M})^n$ be the set of finite words drawn from $D^\ast_{M}$. 
Furthermore, let $\hat T_{M}: \R\rightarrow  D^\ast_{M}$ and  $\bar T_{M}: \R\rightarrow 
D^\ast_{M}$ be the letter maps 
\begin{equation}\label{TMs}
\begin{aligned}
\hat T_M(x)&=\left\{\begin{array}{ll}
M &\mbox{ for} \, x\in[M,\infty),\\
\lceil xM\rceil/M &\mbox{ for} \, x\in(-M,M),\\
-M &\mbox{ for} \, x\in(-\infty,-M],
\end{array}
\right.\\
\bar T_{M}(x)&=\left\{\begin{array}{ll}
M &\mbox{ for} \, x\in[M,\infty),\\
\lfloor xM\rfloor/M &\mbox{ for} \, x\in(-M,M),\\
-M &\mbox{ for} \, x\in(-\infty,-M].
\end{array}
\right.
\end{aligned}
\end{equation}
Thus, $\hat T_M$ moves points upwards  on $(-\infty,M)$, while 
$\bar T_M$ does the opposite on $(-M,\infty)$. 
Let $T_M:\R^2\rightarrow D_M$ be such that  $T_M((x,y))=
(\hat T_{M}(x), \bar T_{M}(y))$, for all $(x,y)\in\R^2$. This map naturally extends to $\widetilde{\R^2}$, 
$\widetilde{\R^2}^{\otimes\N}$ and to the sets of probability measures on them. Furthermore, 
put $\hat\mu_M=\hat\mu\circ \hat T_M^{-1}$, $\bar\mu_M=\bar\mu\circ \bar T^{-1}_M$, 
\begin{equation}
\begin{aligned}
\Phi_{\hat\beta,\hat h,M}(Q^M)&=\int_{\widetilde{D^\ast_M}}\hat\pi_1Q^M(dx)\log 
\phi^M_{\hat\beta,\hat h}(x), \\
\Phi_{L,M}(Q^M)&=\int_{D^\ast_M}\bar\pi_{1,1}Q^M(dy_1) \phi^M_{L}(y_1),
\end{aligned}
\end{equation}
where 
\begin{equation}\label{B37a}
\begin{aligned}
\phi^M_{\hat\beta,\hat h}(x)&=\left\{\begin{array}{ll}
\phi_{\hat\beta,\hat h}(x)&\mbox{for } x=(x_1,\ldots,x_m)\in(D_M^\ast\setminus\{M\})^m,\\
\frac12 &\mbox{otherwise,}
\end{array}
\right.\\
\phi^M_{L}(y_1)&=\left\{\begin{array}{ll}
y_1&\mbox{for } y_1\in D_M^\ast \cap (-L,M],\\
-L &\mbox{otherwise.}
\end{array}
\right.
\end{aligned}
\end{equation}
Here, $\hat\pi_1$ is the projection onto the first word in the word sequence formed
 by the copolymer disorder $\hat\o$,  and $\bar\pi_{1,1}$ is the projection onto the
  first letter of the first word in the word sequence  formed by the pinning 
 disorder $\bar\o$.

Next, note from \eqref{squelb} that 
\begin{equation}\label{squeL}
 s_{L}^{\rm que}(\hb,\hh,\bb,p;g)- \log\cN(g)
=\limsup_{N\rightarrow\infty} \frac1N\log E_g^\ast\left(e^{Np\left[\bar\beta 
\Phi^{-L}(R_N^\o)+\Phi_{\hb,\hh}(R_N^\o)\right]}\right).
\end{equation}
For the rest of the proof we assume that $L\leq M$. Consider a combined model with 
disorder taking values in $D_M^\ast$ instead of $\R$, and with $\hat\mu$ and $\bar \mu$ 
replaced by $\hat\mu_M$ and $\bar\mu_M$, respectively. In this set-up, if in \eqref{squeL}
  we replace $\Phi_{\hb,\hh}$ and $\Phi^{-L}$  by $\Phi_{\hb,\hh, M}$ and 
$\Phi_{L,M}$ respectively, then we obtain
\begin{equation}\label{squeLM}
 s_{L,M}^{\rm que}(\hb,\hh,\bb,p;g)- \log\cN(g)
=\limsup_{N\rightarrow\infty} \frac1N\log E_g^\ast\left(e^{Np\left[\bar\beta\Phi_
{L,M}(R_{N}^{M,\o})+\Phi_{\hb,\hh, M}(R_{N}^{M,\o})\right]}\right).
\end{equation}
Here, $R_{N}^{M,\o}$ is the empirical process of $N$-tuple of words cut out from the i.i.d. 
sequence of letters drawn from $D_M$ according to the marginal law $\hat\mu_M\otimes\bar\mu_M$. 
Note from \eqref{pintrunc} that
\begin{equation}\label{B40a}
\begin{aligned}
\Phi^{-L}(Q)&=\int_{\R}\bar\pi_{1,1}Q(dy_1) \phi_L(y_1),\\
\phi_L(y_1)&=\left\{\begin{array}{ll}
y_1 &\mbox{for} \, y_1\in (-L,\infty)\\
-L &\mbox{otherwise.}
\end{array}
\right.
\end{aligned}
\end{equation}
It therefore  follows from (\ref{squeL}--\ref{squeLM}) that 
\begin{equation}\label{squeLMeq}
 s_{L}^{\rm que}(\hb,\hh,\bb,p;g)\geq s_{L,M}^{\rm que}(\hb,\hh,\bb,p;g),
\end{equation}
 since, for each $x,y\in\R$, $ \phi_L(y)\geq  \phi^M_L(\bar T_M(y))$ and 
 $\phi_{\hat\beta,\hat h}(x)\geq \phi^M_{\hat\beta,\hat h}(\hat T_M(x))$ by  \eqref{TMs}, 
 \eqref{B37a} and \eqref{B40a}. 
 
 In the sequel we will write, respectively, $\cC^{\rm fin}_M$, $\cR_M$, $I^{\rm que}_M$,   
 $I^{\rm ann}_M$  associated with $D_M$ and $\hat\mu_M\otimes\bar\mu_M$ as the analogues 
 of $\cC^{\rm fin}$, $\cR$, $I^{\rm que}$, $I^{\rm ann}$ associated with $\R^2$ and 
 $\hat\mu\otimes\bar\mu$. It follows from Steps 1 and 2 above that 
\begin{equation}\label{squeLMvf}
  s_{L,M}^{\rm que}(\hb,\hh,\bb,p;g)=\sup_{Q^M\in\cC_M^{\rm fin}\cap \cR_M}\left[
  p\bar\beta\Phi_{L,M}(Q^M)+p\Phi_{\hb,\hh, M}(Q_M)-gm_{Q_M}-I^{\rm ann}_M(Q^M)\right].
\end{equation}
Note from  Lemma \ref{squeat0rel} that for any $Q^M\in \cC^{\rm fin}_M$ 
with $I_M^{\rm que}(Q^M)<\infty$ there exists a sequence $(Q^M_n)$ in 
$\cC_M^{\rm fin}\cap \cR_M$  such that ${\rm w}-\lim_{n\rightarrow\infty}\,Q^M_n=Q^M$ 
and $\lim_{n\rightarrow\infty} I_M^{\rm ann}(Q^M_n)=I_M^{\rm que}(Q^M)$, because $D_M$ 
is finite. Therefore 
\begin{equation}\label{B43y}
\begin{aligned}
\lim_{g\downarrow0}s_{L,M}^{\rm que}(\hb,\hh,\bb,p;g)&=s_{L,M}^{\rm que}(\hb,\hh,\bb,p;0)\\
&=\sup_{Q^M\in\cC_M^{\rm fin}\cap \cR_M}\left[
  p\bar\beta\Phi_{L,M}(Q^M)+p\Phi_{\hb,\hh, M}(Q^M)-I^{\rm ann}_M(Q^M)\right]\\
&=\sup_{Q^M\in\cC_M^{\rm fin}}\left[
  p\bar\beta\Phi_{L,M}(Q^M)+p\Phi_{\hb,\hh, M}(Q^M)-I^{\rm que}_M(Q^M)\right].
  \end{aligned}
\end{equation} 
The first equality uses the second equality  of  \eqref{eqatg0} and the remark below it. The 
third equality uses the above remark about Lemma  \ref{squeat0rel}, the boundedness and the 
continuity of the map\\ $Q^M\mapsto 
p\bar\beta\Phi_{L,M}(Q^M)+p\Phi_{\hb,\hh, M}(Q_M)$ on $\cC^{\rm fin}_M$, and  the fact that 
$I_M^{\rm que}=I_M^{\rm ann}$ on $\cC_M^{\rm fin}\cap \cR_M$. Note also that those $Q^M$'s 
with $I_M^{\rm que}(Q^M)=\infty$ do not contribute to the above supremum. For each 
$Q\in\cP^{\rm inv}(\widetilde{\R^2}^{\otimes \N})$, define $[Q]^M=Q\circ T_M^{-1}\in 
\cP^{\rm inv}(\widetilde{D_M}^{\otimes \N}).$ Since $T_M$ is a projection map we have that 
$I^{\rm que}_M([Q]^M)\leq I^{\rm que}(Q)$. It therefore  follows from \eqref{B43y} that
\begin{equation}\label{B43yy}
\begin{aligned}
\lim_{g\downarrow0}s_{L,M}^{\rm que}(\hb,\hh,\bb,p;g)
&=\sup_{Q\in\cC^{\rm fin}}\left[
  p\bar\beta\Phi_{L,M}([Q]^M)+p\Phi_{\hb,\hh, M}([Q]^M)-I^{\rm que}_M([Q]^M)\right]\\
  &\geq \sup_{Q\in\cC^{\rm fin}}\left[
  p\bar\beta\Phi_{L,M}([Q]^M)+p\Phi_{\hb,\hh, M}([Q]^M)-I^{\rm que}(Q)\right]
  \end{aligned}
\end{equation}

Moreover,  for any $Q\in\cP^{\rm inv}(\widetilde{\R^2}^{\otimes \N})$ and $y\in\R$,  
${\rm w}-\lim_{M\rightarrow\infty}[Q]^M=Q$, \\
$\lim_{M\rightarrow\infty}\phi^M_{\hat\beta,\hat h}(\hat T_M(y))=
\phi_{\hat\beta,\hat h}(y) $  and $\lim_{M\rightarrow\infty}\phi^M_{L}(\bar T_M(y))=
\phi_{L}(y)$. Fatou's lemma therefore tells us that $\liminf_{M\rightarrow\infty}
\Phi_{\hb,\hh, M}([Q]^M)\geq\Phi_{\hb,\hh}(Q)$  and $\liminf_{M\rightarrow\infty}
\Phi_{L, M}([Q]^M)\geq\Phi^{-L}(Q)$. These, together with  \eqref{squeLMeq} and 
\eqref{B43y}, imply that 
\begin{equation}\label{B43yyy}
\begin{aligned}
\lim_{g\downarrow0}s_{L}^{\rm que}(\hb,\hh,\bb,p;g)&\geq
\lim_{g\downarrow0}s_{L,M}^{\rm que}(\hb,\hh,\bb,p;g)\\
  &\geq \sup_{Q\in\cC^{\rm fin}}\left[
  p\bar\beta\Phi^{-L}(Q)+p\Phi_{\hb,\hh}(Q)-I^{\rm que}(Q)\right]\\
  &\geq \sup_{Q\in\cC^{\rm fin}}\left[
  p\bar\beta\Phi(Q)+p\Phi_{\hb,\hh}(Q)-I^{\rm que}(Q)\right]\\
  &=S_\ast^{\rm que}(\hb,\hh,\bb,p).
  \end{aligned}
\end{equation}
The last inequality uses that $\Phi^{-L}\geq\Phi.$ 
\end{proof}


\section{Proof of Lemma~\ref{hizlimts} }
\label{appc} 

In this Appendix we prove Lemma~\ref{hizlimts}. To do so we need another lemma,
which we state and prove in Section~\ref{appc1}. In Section~\ref{appc2} we use 
this lemma to prove Lemma~\ref{hizlimts}.

\subsection{A preparatory lemma}
\label{appc1} 

\begin{lemma}
\label{hizfin}
For every $\hb>0$, $\bb\geq0$ and $\hh\geq \hat h^{\rm ann}_c(\hb)$,
\begin{equation}
\label{hizmain}
S^{\rm que}(\hb,\hh,\bb;0) \leq \sup_{\bQ\in\bar\cC^{\rm fin}}
\left[\bb\Phi(\bQ)+\Xi_{\hb, \hh}(r_{\bQ})-\bar I^{\rm que}(\bQ)\right], 
\end{equation}
where
\begin{equation}
\Xi_{\hb, \hh}(r_{\bQ})
=\sum_{n\in\N}r_{\bQ}(n)\log\left(\tfrac12
\left[1+e^{n[\hM(2\hb)-2\hb \hh]}\right]\right)
\end{equation}
and $r_{\bQ}$ is the word length distribution under $\bQ$.
\end{lemma}

\begin{proof}
Throughout the proof, $\hb>0$, $\bb\geq0$ and $\hh\geq \hat h^{\rm ann}_c(\hb)$ are 
fixed. Put
\begin{equation}
S^\o_N = E_0^\ast\left(e^{N\left[\bb\Phi(R_N^\o)
+\Phi_{\hb, \hh}(R_N^\o)\right]}\right)
= E_0^\ast\left(e^{N\left[\bb\Phi(R_N^{\bo})
+\Phi_{\hb, \hh}(R_N^{\ho})\right]}\right).
\end{equation}
Note from \eqref{B5} and the Borel-Cantelli lemma that, for every $\epsilon>0$ 
and $\bo$-a.s., there exists an $N_0=N_0(\bo,\epsilon)<\infty$ such that 
\begin{equation}
\label{C3}
E_0^\ast\left(e^{N\bb\Phi(R_N^\o)}\right)
= E_0^\ast\left(e^{N\bb\Phi(R_N^{\bo})}\right)
\leq e^{N [\bar M(-\bb)+\epsilon]} \qquad \forall\,N\geq N_0.
\end{equation}
Therefore, $\bo$-a.s.\ and for all $N\geq N_0$, 
\begin{equation}
\begin{split}
\E_{\ho}\left(S^\o_N\right)
&=\sum_{0=k_0<k_1<\ldots<k_N<\infty}\prod_{i=1}^N \rho(k_i-k_{i-1})
\,e^{\bb \bo_{k_i}} \tfrac12
\left[1+\E_{\ho}\left(e^{-2\hb\sum_{k=k_{i-1}+1}^{k_i}(\ho_k+\hh)}\right)\right]\cr
&=\sum_{0=k_0<k_1<\ldots<k_N<\infty}\prod_{i=1}^N \rho(k_i-k_{i-1})\,
e^{\bb \bo_{k_i}} \tfrac12
\left(1+e^{(k_i-k_{i-1})[\hM(2\hb)-2\hb \hh]}\right)\cr
&=\sum_{0=k_0<k_1<\ldots<k_N<\infty}
\left(\prod_{i=1}^N \rho(k_i-k_{i-1})\right)\,
\left(e^{N[\bb \Phi(R_N^\o)+ \Xi_{\hb,\hh}(r_{R_N^\o})]}\right)\cr
&=E_0^\ast\left(e^{N[\bb  \Phi(R_N^\o)+\Xi_{\hb, \hh}(r_{R_N^\o})]}\right)
<\infty.
\end{split}
\end{equation}
Finiteness follows from \eqref{C3} and the fact that $\Xi_{\hb, \hh}\leq0$ if $\hh\geq
\hat h^{\rm ann}_c(\hb)$. Therefore, for every $\delta>0$, $\bo$-a.s.\ and $N\geq N_0$, 
we have 
\begin{equation}
\begin{split}
\P_{\ho}\left(\frac1N\log S^\o_N\geq \frac1N\log\E_{\ho}( S^\o_N)+\delta\right)
=\P_{\ho}\left( S^\o_N\geq \E_{\ho}( S^\o_N) \,e^{N\delta}\right)\leq e^{-N\delta}. 
\end{split}
\end{equation}
From the Borel-Cantelli lemma we therefore obtain that $\o$-a.s. 
\begin{equation}
\begin{split}
s^{\rm que}(\hb,\hh,\bb;0)=\limsup_{N\rightarrow\infty}\frac1N\log S^\o_N
&\leq \limsup_{N\rightarrow\infty}\frac1N\log\E_{\ho}(S^\o_N)+\delta\cr
&=\sup_{\bQ\in\bar\cC^{\rm fin}}\left[\bb  \Phi(\bQ)+ \Xi_{\hb,\hh}(r_{\bQ})
-\bar I^{\rm que}(\bQ)\right]+\delta.
\end{split}
\end{equation}
The equality uses Steps 1 and 2 in the proof of Lemma~\ref{varlem} and the observation 
that $\Xi_{\hb, \hh}$ is independent of $\o$ (i.e., only pinning disorder is present), 
and $-\log 2\leq \Xi_{\hb,\hh}\leq0$ for $\hh\geq\hat h_c^{\rm ann}(\hb)$, where we 
use \eqref{eqgndefinitionIalg} instead of \eqref{eqgndefinitionIalgz}. Finally, let 
$\delta \downarrow 0$.
\end{proof}


\subsection{Proof of Lemma \ref{hizlimts}} 
\label{appc2}

\begin{proof} 
Throughout the proof, $\bb\geq0$ and $\hb>0$ are fixed and $\hh\geq \hat h^{\rm ann}_c(\hb)$. 
Note from \eqref{coPhidef} and \eqref{phidef} that $\Phi_{\hb,\infty}\equiv-\log 2$. 
Therefore, replacing $\bb\bar\Phi+\Phi_{\hb, \hh}$ by $\bb\bar\Phi-\log 2$ in
\eqref{Sdefalt}, we get
\begin{equation}
\begin{split}
S^{\rm que}(\hb,\infty,\bb;0)
&=S_\ast^{\rm que}(\hb,\infty,\bb)=s^{\rm que}(\hb,\infty,\bb;0)\cr
&=\limsup_{ N\rightarrow\infty}\frac1N\log 
E^\ast_0\left(e^{N[\log\frac12+\bb \Phi(R^{\bo}_N)]}\right) \cr
&=-\log 2+\sup_{\bQ\in\bar\cC^{\rm fin}}
\left[\bb\Phi(\bQ)-\bar I^{\rm que}(\bQ)\right]\cr
&=-\log 2+\bar h_c^{\rm que}(\bb).
\end{split}
\end{equation} 
The fourth equality follows from the proof of Lemma~\ref{varlem}, while the last equality 
uses \cite{ChdHo10}, Theorem 1.3. Next, note that
\begin{equation}
S^{\rm que}(\hb,\infty^-,\bb;0)
= \lim_{\hh\uparrow\infty}S^{\rm que}(\hb,\hh,\bb;0)
\geq S^{\rm que}(\hb,\infty,\bb;0),
\end{equation} 
since the map $\hh\mapsto S^{\rm que}(\hb,\hh,\bb;0)$ is non-increasing. For $\hh
\geq\hat h^{\rm ann}_c(\hb)$ it follows from \eqref{hizmain} that 
\begin{equation}
\label{hizmain1}
\begin{split}
S^{\rm que}(\hb,\hh,\bb;0) &\leq \sup_{\bQ\in\bar\cC^{\rm fin}}
\left[\bb\Phi(\bQ)+\Xi_{\hb, \hh}(r_{\bQ})-\bar I^{\rm que}(\bQ)\right]\cr
&=\sup_{r\in\cP(\N);\atop{m_r<\infty}}\sup_{\bQ\in\bar\cC^{\rm fin};\atop{r_{\bQ}=r}}
\left[\bb\Phi(\bQ)+\Xi_{\hb, \hh}(r)-\bar I^{\rm que}(\bQ)\right]\cr
&\leq\sup_{r\in\cP(\N);\atop{m_r<\infty}}\left(\Xi_{\hb,\hh}(r)
+\sup_{\bQ\in\bar\cC^{\rm fin}}
\left[\bb\Phi(\bQ)-\bar I^{\rm que}(\bQ)\right]\right)\cr
&\leq\log\left[\tfrac12\left(1+e^{[\hM(2\hb)-2\hb \hh]}\right)\right]
+\sup_{\bQ\in\bar\cC^{\rm fin}}
\left[\bb\Phi(\bQ)-\bar I^{\rm que}(\bQ)\right]\cr
&=\log\left[\tfrac12\left(1+e^{[\hM(2\hb)-2\hb \hh]}\right)\right]
+\bar h^{\rm que}_c(\bb).
\end{split}
\end{equation}
The third inequality uses that, for $\hh\geq \hat h^{\rm ann}_c(\hb)$, 
$\Xi_{\hb,\hh}(r)\leq\log\left[\tfrac12\left(1+e^{[\hM(2\hb)-2\hb \hh]}\right)\right]$ 
for all $r\in\cP(\N)$. Therefore  
\begin{equation}
\lim_{\hh\uparrow\infty} S^{\rm que}(\hb,\hh,\bb;0)
\leq -\log 2 +\bar h^{\rm que}_c(\bb)=S^{\rm que}(\hb,\infty,\bb;0).
\end{equation}
\end{proof}



\newpage

\end{document}